\documentclass[preprint, 3p]{elsarticle}

\usepackage{lipsum}
\makeatletter
\def\ps@pprintTitle{%
	\let\@oddhead\@empty
	\let\@evenhead\@empty
	\def\@oddfoot{}%
	\let\@evenfoot\@oddfoot}
\makeatother

\usepackage{algorithm}
\usepackage{algpseudocode}
\usepackage{amsmath}
\usepackage{amssymb}
\usepackage{amsthm}
\usepackage{amscd}
\usepackage{appendix}
\usepackage{bbm}
\usepackage{braket}
\usepackage{bookmark}
\usepackage{booktabs}
\usepackage{caption}
\usepackage{cases}
\usepackage{enumitem}
\usepackage[T1]{fontenc} 
\usepackage{float}
\usepackage{lineno,hyperref}
\modulolinenumbers[5]
\usepackage{indentfirst}
\usepackage[utf8]{inputenc}
\usepackage[autostyle,italian=guillemets]{csquotes}
\usepackage{graphicx}
\graphicspath{{Immagini/}}
\DeclareGraphicsExtensions{{.png},{.pdf},{.jpg}}
\graphicspath{{./}}
\usepackage{lmodern}
\usepackage{listings}
\usepackage{yhmath}
\usepackage{mathtools}
\usepackage{mathrsfs}
\usepackage{microtype} 
\usepackage{multirow}
\usepackage{setspace}
\usepackage{standalone}
\usepackage{subcaption}
\usepackage{tabularx}
\usepackage[table]{xcolor}
\usepackage{wasysym}
\usepackage{wrapfig}
\usepackage{url} 
\usepackage[overload]{empheq}
\usepackage{cleveref}

\usepackage{numcompress}

\usepackage{pgfplots}
\pgfplotsset{compat=newest}
\usetikzlibrary{plotmarks}
\usetikzlibrary{arrows.meta}
\usepgfplotslibrary{patchplots}
\usepackage{grffile}

\newtheorem{thm}{Theorem}

\newtheorem{prop}[thm]{Proposition}

\newtheorem{remark}{Remark}

\begin{document}
	
\begin{frontmatter}
\title{Efficient and certified solution of parametrized one-way coupled problems through DEIM-based data projection across non-conforming interfaces}
		
\author[1]{Elena Zappon} 
\ead{elena.zappon@polimi.com}
\author[1]{Andrea Manzoni}
\ead{andrea1.manzoni@polimi.com}
\author[1,2]{Alfio Quarteroni}
\ead{alfio.quarteroni@polimi.it}
		
		
\address[1]{MOX - Dipartimento di Matematica, Politecnico di Milano, P.zza Leonardo da Vinci 32, I-20133 Milano, Italy}
		
\address[2]{Institute of Mathematics, Ecole Polytechnique Federale de Lausanne, Station 8,  CH-1015 Lausanne, Switzerland (Professor Emeritus)}
		
\begin{abstract}
One of the major challenges of coupled problems is to manage nonconforming meshes at the interface between two models and/or domains, due to different numerical schemes or domains discretizations employed.
Moreover, very often complex submodels depend on (e.g., physical or geometrical) parameters. Understanding how outputs of interest are affected by parameter variations thus plays a key role to gain useful insights on the problem's physics; however, expensive repeated solutions of the problem using high-fidelity, full-order models are often unaffordable. 
In this paper, we propose a parametric reduced order modeling (ROM) technique for parametrized one-way coupled problems made by a first independent model, the \emph{master model}, and a second model, the \emph{slave model}, that depends on the master model through Dirichlet interface conditions. We combine a reduced basis (RB) method, applied to each subproblems, with the discretized empirical interpolation method (DEIM) to efficiently interpolate or project Dirichlet data across conforming and non-conforming meshes at the domains interface, building a low-dimensional representation of the overall coupled problem. The proposed technique is then numerically verified by considering a series of test cases involving both steady and unsteady problems, and deriving \emph{a-posteriori} error estimates on the solution of the coupled problem in both cases. 
This work arises from the need to solve staggered cardiac electrophysiological models and represents the first step towards the setting of ROM techniques for the more general two-way Dirichlet-Neumann coupled problems solved with domain decomposition sub-structuring methods, when interface non-conformity is involved.\end{abstract}
\end{frontmatter}

\section{Introduction}

Fast simulation techniques for multi-scale and multi-physics problems are nowadays of key relevance in many field of applied science and engineering (such as, e.g., in biomedicine, naval and aeronautical engineering, ranging from fluid-structure interaction (FSI) problems to electro-mechanical (EM) couplings \cite{Bazilevs2013,Discacciati2009,Korvink2005,Piersanti2021,QuarteroniDede2019,Wong2013,zhao2018computational}). Often, these problems are characterized by two (or more) nonlinear partial differential equations (PDEs) representing detailed parametric physical systems interacting through their boundaries. When a full-order model (FOM) based, e.g., on the finite element (FE) method, is used, solving accurately such coupled systems becomes computationally demanding, especially when strong constraints on spatial mesh sizes or on time steps must be imposed to deal with steep fronts solutions or fast dynamics.

In several cases, such coupled problems describe complex phenomena through physical and geometrical parameters dependency. Very often, gaining useful insights of the prescribed physics relies upon expensive repeated solutions of the model at hand, to assess how outputs of interest are affected by the parameters variation \cite{Bonomi2017,FortiRozza2014,Fresca2021,Geneser2008,Pacciarini2015,Pagani2018,Swenson2011}. In this respect, numerical simulations carried out by high-fidelity FOMs may easily become out of reach. 
In addition, when dealing with nonconforming meshes at the interface, special techniques, \emph{e.g} MORTAR method and INTERNODES \cite{Bernardi2005,Chan1996,Deparis2016,Gervasio2019,Hesch2014, QuarteroniValli1999}, need to be employed, making the corresponding FOM even more computationally demanding. 

In these contexts, the application of efficient reduced order models (ROMs) can be successfully applied to decrease the overall computational costs. Preliminary studies on reduced coupled problems, especially in a FSI context, were carried out in \cite{Ripepi2018,Amsallem2010,Ballarin2017,Lassila2012,Lassila2013a}, where ROMs have been applied to one or both subproblems, considering splitting schemes to handle the model coupling, while in \cite{Ballarin2016} POD is used to reduced an FSI problem solved with a monolithic FE scheme. However, geometrical and numerical interface conformity between the fluid and solid domains is always necessary to set up the numerical schemes.

Domain decomposition techniques coupled with reduced basis (RB) methods \cite{Benner2017, Hesthaven2016,QuarteroniManzoniNegri2016}, instead, have been explored in several works, \emph{e.g} \cite{Lovgren2006,IAPICHINO201263, IAPICHINO2016408,PEGOLOTTI2021113762}, trying to solve expensive models set on involved  domains by splitting the considered geometries in generic building blocks, often exploiting their topological similarity, and applying the reduction locally, in the context of each building block. The final global approximation is then computed gluing together the local solutions through different techniques, \emph{e.g.} Lagrangian multipliers or by Fourier basis functions. Moreover, RB-DD methods have been considered to construct efficient preconditioners \cite{santo2018multi} or to perform static condensation \cite{Eftang2014,Huynh2013}.  

In this paper, we propose a RB method to solve efficiently one-way coupled problems. Precisely, we consider two parametric second order elliptic and/or parabolic models defined on two domains with a common interface. The first model, the \emph{master model}, is solved as an independent model, imposing homogeneous Neumann boundary conditions at the interface. 
Instead, the second model, the \emph{slave model}, is dependent from the master model through Dirichlet interface conditions, \emph{i.e.} the slave solution at the interface is equal to the master solution.
Note that the Dirichlet interface conditions inherit naturally the parameters dependency from the master model. As high-fidelity FOM, we consider the finite element method involving either conforming or non-conforming discretizations at the interface domains. In particular, considering the high fidelity FE discretization, the slave solution at the interface corresponds to the master one in the conforming case while, in the non-conforming one, the slave interface solution must be computed with an interpolation method \emph{e.g.} relying on Radial Basis Functions. 

Then, we implement RB methods to reduce completely the parametric one-way coupled problem, including the interface conditions. To this end, we consider a modular approach: the two sub-problems are reduced independently through a POD-Galerkin approach, while the interface data -- and, therefore, the coupling -- are handled by setting a further efficient interpolation -- or projection -- stage relying on the discrete empirical interpolation method (DEIM) \cite{Barrault2004,Chaturantabut2009,Chaturantabut2010,Grepl2007,Maday2008,Negri2015}. In particular, we show the possibility to use this reduction paradigm to transfer Dirichlet data across both conforming and non-conforming domain interface, effectively replacing in the reduced problem any high fidelity interpolation technique with the DEIM. 

Differently from previous works, this new approach is able to decrease the overall computational costs of solving parametric one-way coupled problems through RB methods, including the possibility to consider interface grid non-conformity. The interface Dirichlet data can be easily transferred between the two problems without implementing other techniques that can increase the algebraic system dimensions, \emph{e.g.} this is the case of the Lagrange multipliers. Furthermore, several test cases and the derived \emph{a-posteriori} error estimate show that the modular approach allows to have full control of the solution accuracy in each reduction step, and to tailor the corresponding reduced order model on every sub-problems. The proposed technique can be seen as a reduced form of the DD method applied to one-way coupled problems, \emph{i.e.} the two involved models are handled as independent models and solved accordingly in sequence. 

The present work arises from the need of solving staggered cardiac electrophysiological models, \emph{i.e.} bidomain-torso models \cite{Boulakia2010,Boulakia2011}, and is a preliminary investigation toward the solution of a more general (and challenging) parametrized two-way coupled problem obtained through Dirichlet-Neumann interface conditions. In particular, the DEIM based interface reduction can be extended also to the Neumann conditions. Therefore, the described modular ROM can be used when domain decomposition sub-structuring methods \cite{QuarteroniValli1999,Bjorstad2018} are applied to steady or unsteady two-way coupled problems, in case of both conforming or non-conforming interface grids, and will be the focus of a future manuscript.
 
The structure of the paper is as follows: the formulation of the parametrized one-way coupled problem is summarized in Section$~\ref{Sub_FOM_coupled}$. In Section $\ref{Sec:high_fidelity_dis}$ we describe the high fidelity discretization prior to the presentation of the reduced formulation for each problems in Section$~\ref{Sub_ROM_subproblems}$, while the treatment of the interface data is described in Section$~\ref{Sub_ROM_interface}$. Section$~\ref{Sub_error_estimate}$ presents a posteriori error estimate of the proposed techniques. All the theoretical formulations and the error estimates are defined for both steady and unsteady problems in the respective Subsections. FOM and ROM solutions are then compared by means of simple numerical test cases in Section$~\ref{Sub_numerical_test}$. Conclusions and perspectives follow in the final Section of this work.

\section{Problem formulation}
\label{Sub_FOM_coupled}

In this Section we introduce a general formulation of the parametrized one-way coupled problems, treating separately the steady and the unsteady cases.  To better highlight the components of the proposed strategy, and its versatility, we pursue an algebraic formulation of the problems, assuming to deal with the finite element method as high-fidelity FOM. Indeed, the proposed strategy, relying on {\em (i)} a projection-based ROM built through the RB method, and {\em (ii)} the discrete empirical interpolation method, is independent of the employed FOM, however it can be easily described in an algebraic form. 

Let us consider a \textit{d}-dimensional domain (\textit{d} = 2,3) partitioned into two non-overlapping subdomains $\Omega_1$ and $\Omega_2$, sharing  a common interface $\Gamma := \overline{\Omega}_1 \cap \overline{\Omega}_2$, and denote by $\mathbf{n}_i$, $i = 1,2$ the outer unit normal directions of the two domains with respect to the interface $\Gamma$. Hereon, to simplify the notation, we call \emph{master model} the problem set in $\Omega_1$, and \emph{slave model} the one set in $\Omega_2$; correspondingly, we refer to their solutions as to the \emph{master solution} and the \emph{slave solution}, respectively.

\subsection{Steady case}
We first consider a steady time independent coupled problem; the unsteady counterpart will be described in the following Subsection. As an abstract instance of parameter dependent models set over each subdomain, we consider the following ones: given two sets of parameters $\boldsymbol{\mu}_1 \in \mathscr{P}^{d_1}$ and $\boldsymbol{\mu}_2 \in \mathscr{P}^{d_2}$, $d_1$, $d_2 \geq 1$, and two functions $\mathbf{f}_1(\boldsymbol{\mu}_1)$ and $\mathbf{f}_2(\boldsymbol{\mu}_2)$ defined on $\Omega_i$, $i = 1,2$  respectively, we look for $\mathbf{u}_1$ in $\Omega_1$ and $\mathbf{u}_2$ in $\Omega_2$ such that
\begin{equation}
	\label{Eq_probl11}
	\begin{cases}
		\mathcal{L}_1(\boldsymbol{\mu}_1) \mathbf{u}_1(\boldsymbol{\mu}_1) = \mathbf{f}_1 &\text{in } \Omega_1 \\
		\text{\textbf{BCs}}(\boldsymbol{\mu}_1) &\text{on }\partial \Omega_1\backslash \Gamma\\
		\frac{\partial \mathbf{u}_1(\boldsymbol{\mu}_1)}{\partial n}= 0 &\text{on }\Gamma,
	\end{cases}
\end{equation}
and
\begin{equation}
	\label{Eq_probl21}
	\begin{cases}
		\mathcal{L}_2(\boldsymbol{\mu}_2) \mathbf{u}_2(\boldsymbol{\mu}_2) = \mathbf{f}_2 &\text{in } \Omega_2 \\
		\text{\textbf{BCs}}(\boldsymbol{\mu}_2) &\text{on }\partial \Omega_2\backslash \Gamma\\
		\mathbf{u}_2(\boldsymbol{\mu}_2) = \mathbf{u}_1(\boldsymbol{\mu}_1) &\text{on }\Gamma.
	\end{cases}
\end{equation} 
Here, $\mathcal{L}_1$ and $\mathcal{L}_2$ denote two second order elliptic operators; several examples will be provided in Section $\ref{Sub_numerical_test}$. More specifically, problems $\eqref{Eq_probl11}$ and $\eqref{Eq_probl21}$ can represent either two different physical systems, or the multi-domain form of the same system. The master model is made independent from the slave one imposing homogeneous Neumann boundary conditions on $\Gamma$, while Dirichlet boundary conditions 
\begin{equation}
\label{eq_interface_dirichlet_general}
\mathbf{u}_2(\boldsymbol{\mu}_2) = \mathbf{u}_1(\boldsymbol{\mu}_1) \quad \text{on }\Gamma,
\end{equation}
are applied at the interface of problem $\eqref{Eq_probl21}$.
The Dirichlet boundary conditions $\eqref{eq_interface_dirichlet_general}$ are, then, unidirectional because they express the imposition of the continuity of the solution from the first to the second model.

One-way coupled problems such as this can be, for examples, the results of a partitioned scheme's splitting operation on a two-way coupled problem featured by Dirichlet-Neumann interface conditions. 
In this paper, \emph{(i)} we  solve the master model with FE method, then \emph{(ii)} we extract the master solution at the interface domains and  \emph{(iii)} we use the gained data as Dirichlet boundary conditions to solve the slave model with FE method, too. This procedure is affordable if the grids used are conforming at the interface of the two domains, however it usually involves the application of methods such as the MORTAR \cite{Bernardi2005,Hesch2014,QuarteroniValli1999,Bernardi1994} or the INTERNODES \cite{Deparis2016,Gervasio2019,Gervasio2018,Gervasio20182} methods in the non-conforming case, which might become very expensive especially for three dimensional domains. \\

Given the parameterized nature of the problem, the solutions $\mathbf{u}_i$ can be seen as two maps $\mathbf{u}_1 : \mathscr{P}^{d_1} \rightarrow V_1$ and $\mathbf{u}_2 : \mathscr{P}^{d_2} \rightarrow V_2$ that to any $\boldsymbol{\mu}_1 \in \mathscr{P}^{d_1}$ and $\boldsymbol{\mu}_2 \in \mathscr{P}^{d_2}$ associate the solutions $\mathbf{u}_1(\boldsymbol{\mu}_1)$ and $\mathbf{u}_2(\boldsymbol{\mu}_2)$ to the corresponding functional spaces $V_1$ and $V_2$.  Let us choose
\begin{equation}
V_i = H^1_{\partial \Omega_{i,D}}(\Omega_i) := \{ \mathbf{v} \in H^1(\Omega_i): ~ \mathbf{v}_{\mid \partial \Omega_{i,D}} = 0 \}, ~ i = 1,2
\end{equation} 
being $\partial \Omega_{i,D}$ suitable disjoint subsets of $\partial\Omega_i\backslash \Gamma$ for Dirichlet boundary conditions of problems$~\eqref{Eq_probl11}$ and$~\eqref{Eq_probl21}$, respectively. Then, let us introduce the bilinear and linear forms $a_i(\cdot,\cdot;\boldsymbol{\mu}_i) : V_i \times V_i \rightarrow \mathbb{R}$ for each $\boldsymbol{\mu}_i \in \mathscr{P}^{d_i}$, and $\mathcal{F}_i(\cdot;\boldsymbol{\mu}_i) : V_i \rightarrow \mathbb{R}$ for each $\boldsymbol{\mu}_i \in \mathscr{P}^{d_i}$, $i = 1,2$, such that the weak formulations of problems$~\eqref{Eq_probl11}$ and$~\eqref{Eq_probl21}$ reads: find $\mathbf{u}_i(\boldsymbol{\mu}_i) \in H^1(\Omega_i)$ such that
\begin{equation}
\label{weak_FOM}
a_i(\mathbf{u}_i(\boldsymbol{\mu}_i),\mathbf{v}_i;\boldsymbol{\mu}_i) = \mathcal{F}_i(\mathbf{v}_i;(\boldsymbol{\mu}_i)) \qquad \forall\mathbf{v}_i \in V_i ,
\end{equation}
for each $i=1,2$. For instance, if 
$$ \mathcal{L}_i(\boldsymbol{\mu}_i) \mathbf{u}_i(\boldsymbol{\mu}_i) = - \nabla \cdot(\mathbf{g}_i(\boldsymbol{\mu}_i)\nabla \mathbf{u}_i(\boldsymbol{\mu}_i))$$
with $\mathbf{g}_i(\boldsymbol{\mu}_i)$ a suitable functions in $\Omega_i$ expressing a parametrized diffusion coefficient, the bilinear forms can be expressed as 
$$a_i(\mathbf{u}_i(\boldsymbol{\mu}_i),\mathbf{v}_i;\boldsymbol{\mu}_i) = \int_{\Omega_i} (\mathbf{g}_i(\boldsymbol{\mu}_i)\nabla \mathbf{u}_i(\boldsymbol{\mu}_i)\cdot \nabla \mathbf{v}_i)d\Omega_i. 
$$
Moreover, denoting by $(\cdot,\cdot)_{L^2(\Omega_i)}$ the inner product in $L^2(\Omega_i)$, $i = 1,2$, we set 
$$\mathcal{F}_i(\mathbf{v}_i;\boldsymbol{\mu}_i) = (\mathbf{f}_i(\boldsymbol{\mu}_i),\mathbf{v}_i)_{L^2(\Omega_i)} \text{ plus Neumann terms on }\partial\Omega_{i}\backslash\Gamma \quad \forall \mathbf{v}_i \in H^1(\Omega_i).$$
Hereon, we will consider only homogeneous Neumann boundary conditions on the $\partial\Omega_{i}\backslash\Gamma$ while, in presence of Dirichlet boundary conditions, a lifting technique can be applied. Independently from the nature of the problems, in what follows we assume that the solution of$~\eqref{Eq_probl11}$ and$~\eqref{Eq_probl21}$ exists and is unique for each $\boldsymbol{\mu}_i \in \mathscr{P}^{d_i}$.

\subsection{Unsteady case}
With the same notation of $\eqref{Eq_probl11}$ and $\eqref{Eq_probl21}$, we consider the following time-dependent models: assuming that  $\boldsymbol{\mu}_1 \in \mathscr{P}^{d_1}$ and $\boldsymbol{\mu}_2 \in \mathscr{P}^{d_2}$, $d_1$, $d_2 \geq 1$ are two set of time-independent parameters, $\mathbf{f}_1(t;\boldsymbol{\mu}_1)$, $\mathbf{f}_2(t;\boldsymbol{\mu}_2)$ are two time-dependent functions defined on $\Omega_i \times \{0,T\}$, $i = 1,2$  respectively, and $\mathcal{L}_1$ and $\mathcal{L}_2$ are second order elliptic operators, we look for $\mathbf{u}_1(t;\boldsymbol{\mu}_1)$ in $\Omega_1\times \{0,T\}$ and $\mathbf{u}_2(t;\boldsymbol{\mu}_2)$ in $\Omega_2\times \{0,T\}$ such that
\begin{equation}
\label{Eq_probl11_time}
\begin{cases}
\frac{\partial \mathbf{u}_1(t;\boldsymbol{\mu}_1)}{\partial t} +  \mathcal{L}_1(\boldsymbol{\mu}_1) \mathbf{u}_1(t;\boldsymbol{\mu}_1) = \mathbf{f}_1(t;\boldsymbol{\mu}_1) &\text{in } \Omega_1 \times \{ 0,T\}\\
\text{BCs}(t;\boldsymbol{\mu}_1) &\text{on }\partial \Omega_1\backslash \Gamma \times \{ 0,T\}\\
\mathbf{u}_1(0;\boldsymbol{\mu}_1) = \mathbf{u}_{1,0}(\boldsymbol{\mu}_1) &\text{on } \Omega_1 \times \{0\},
\end{cases}
\end{equation}
and
\begin{equation}
\label{Eq_probl21_time}
\begin{cases}
\frac{\partial \mathbf{u}_2(t;\boldsymbol{\mu}_2)}{\partial t} + \mathcal{L}_2(\boldsymbol{\mu}_2) \mathbf{u}_2(t;\boldsymbol{\mu}_2) = \mathbf{f}_2(t;\boldsymbol{\mu}_2) &\text{in } \Omega_2 \times \{0,T\} \\
\text{BCs}(t; \boldsymbol{\mu}_2) &\text{on }\partial \Omega_2\backslash \Gamma \times\{0,T\}\\
\mathbf{u}_2(0;\boldsymbol{\mu}_2) = \mathbf{u}_{2,0}(\boldsymbol{\mu}_2) &\text{on } \Omega_2 \times \{0\},
\end{cases}
\end{equation}
where $T>0$ represent the final time. As in Section $\ref{Sub_FOM_coupled}$, we consider homogeneous Neumann interface conditions for the master models and Dirichlet interface conditions as coupling conditions on the slave model:
$$ \mathbf{u}_2(t;\boldsymbol{\mu}_2) = \mathbf{u}_1(t;\boldsymbol{\mu}_1) \qquad \text{on }\Gamma.$$
Note that the time variable $t$ has been added to account for the time dependency of the solution, while the spatial variable $\mathbf{x}$ is implicitly considered.

\begin{remark}The differential operators $\mathcal{L}_1$ and $\mathcal{L}_2$  here considered are time independent; however, the method presented in this work can be easily adapted to the case of time dependent partial differential operators, too.\end{remark}

\begin{remark}An unsteady coupled problem can also be obtained coupling a time-dependent problem in $\Omega_1$ and a time-independent problem in $\Omega_2$, or viceversa. For example, if a time-dependent model is coupled in a one way sense with a time-independent model, the slave model becomes time-dependent through the Dirichlet boundary conditions received from the master model (see the test case \emph{ii} of Section $\ref{Subsec:test_case_ii}$).\end{remark}

The same consideration of resolution affordability can be stated as for the steady case.  Time-dependent models such as this, in fact, can be equally solved computing, for each time instant, \emph{(i)} the master solution, \emph{(ii)} extracting the Dirichlet interface conditions from the master model and \emph{(iii)} using them to solve the slave model. Also in this case, we set ROM strategies for unsteady models starting from their discretized form with FE methods. Therefore, we will introduce the weak formulation of $\eqref{Eq_probl11_time}$ and $\eqref{Eq_probl21_time}$. 

In this case, for each $t>0$, the solutions of the two subproblems can be seen as two maps $\mathbf{u}_1: \mathscr{P}^{d_1} \rightarrow V_1$ and $\mathbf{u}_2 : \mathscr{P}^{d_2} \rightarrow V_2$ that to any $\boldsymbol{\mu}_1 \in \mathscr{P}^{d_1}$ and $\boldsymbol{\mu}_2 \in \mathscr{P}^{d_2}$ associate the solutions $\mathbf{u}_1(\boldsymbol{\mu}_1)$ and $\mathbf{u}_2(\boldsymbol{\mu}_2)$. Therefore, we seek for $\mathbf{u}_i(t;\boldsymbol{\mu}_i) \in V_i$ such that \begin{equation}
\label{Eq:unsteady_weak_form}
\int_{\Omega_i} \frac{\partial \mathbf{u}_i(t;\boldsymbol{\mu}_i)}{\partial t}\mathbf{v}_i d\Omega + a_i(\mathbf{u}_i(t;\boldsymbol{\mu}_i),\mathbf{v}_i;\boldsymbol{\mu}_i) = \int_{\Omega_i}\mathbf{f}_i(t;\boldsymbol{\mu}_i)\mathbf{v}_i d \Omega_i \quad \forall \mathbf{v}_i \in V_i.\end{equation}
Here $\mathbf{u}_i(0;\boldsymbol{\mu}_i) = \mathbf{u}_{i,0}(\boldsymbol{\mu}_i)$ and $a_i(\cdot,\cdot;\boldsymbol{\mu}_i)$ is the bilinear form associated to the linear operator $\mathcal{L}_i$. As before, Neumann terms can appear in the formulation according to the problem considered, while in presence of Dirichlet boundary conditions, the lifting method can be implemented. Hereon, we assume that the solution of the proposed problems exists and is unique for each parameter instance and each time $t>0$.

\section{High fidelity discretization}
\label{Sec:high_fidelity_dis}
The RB method proposed in this work aims at reducing the computational costs of solving a parametrized one-way coupled problem, for instance when multiple queries with different parameters values are required; in principle, it can be applied in case of both conforming or non-conforming grids at the domains interface. 
For the sake of generality, here we consider interface non-conformity through two a-priori independent discretizations on the two domains, with two families of triangulations $\mathcal{T}_{h_1} = \cup_m T_{1,m}$ in $\Omega_1$ and $\mathcal{T}_{h_2} = \cup_m T_{2,m}$ in $\Omega_2$, respectively. For instance, one of the two meshes can be made by simplices (triangles or tetrahedra) and the other by quads (quadrilaterals of hexahedra), or both can be made by the same kind of elements, however featuring different mesh sizes $h_1$ and $h_2$. Moreover, different polynomial degrees $p_1$ and $p_2$ can be used to define the finite element spaces. In the rest of the formulation, we will consider only the use of quads since they are the elements used in all the test cases presented in Section \ref{Sub_numerical_test}.

Additionally, we denote by  $\Gamma_1$ and $\Gamma_2$ the internal interfaces of $\Omega_1$ and $\Omega_2$, respectively, induced by the triangulations $\mathcal{T}_{h_1}$ and $\mathcal{T}_{h_2}$. Note that $\Gamma_1 = \Gamma_2 = \Gamma$ if the interface $\Gamma$ is a straight segment (for $d=2$) or a plane (for $d=3$), otherwise $\Gamma_1$ and $\Gamma_2$ can also be different. In case of non-conforming grids at the interface, we have signal interpolation if $\Gamma_1 = \Gamma_2$ or signal projection if $\Gamma_1 \not= \Gamma_2$ (see Fig. $\ref{Fig:Interfaces}$).

\begin{figure}
	\hspace{1.25cm}
	\begin{subfigure}{.3\textwidth}
		\centering
		\includegraphics[width=1\textwidth]{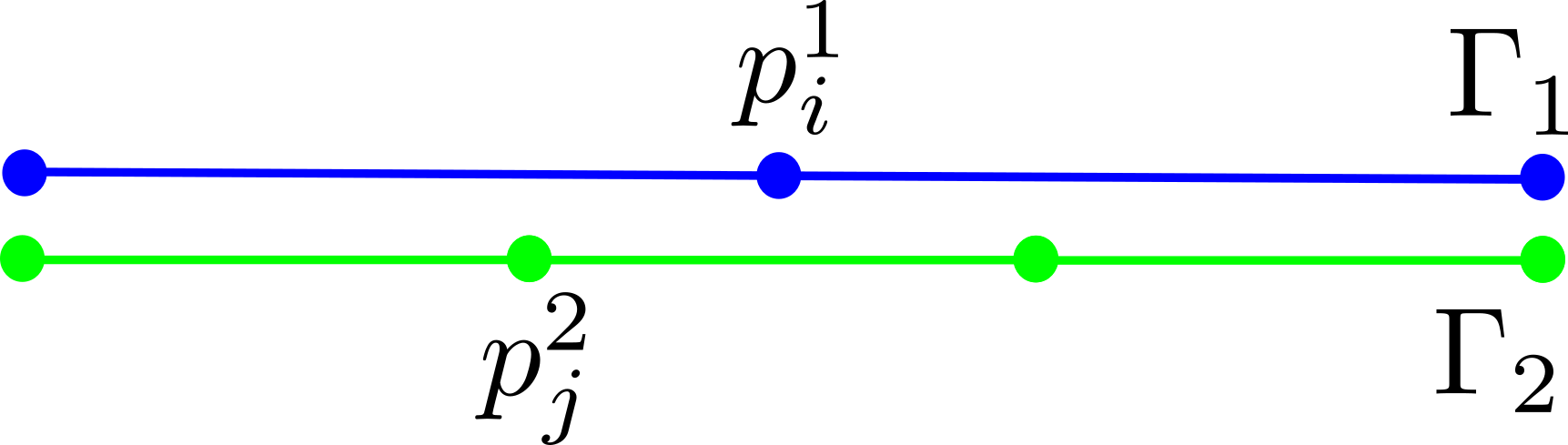}
	\end{subfigure}
	\hspace{1cm}
	\begin{subfigure}{.3\textwidth}
		\centering
		\includegraphics[width=1.5\textwidth]{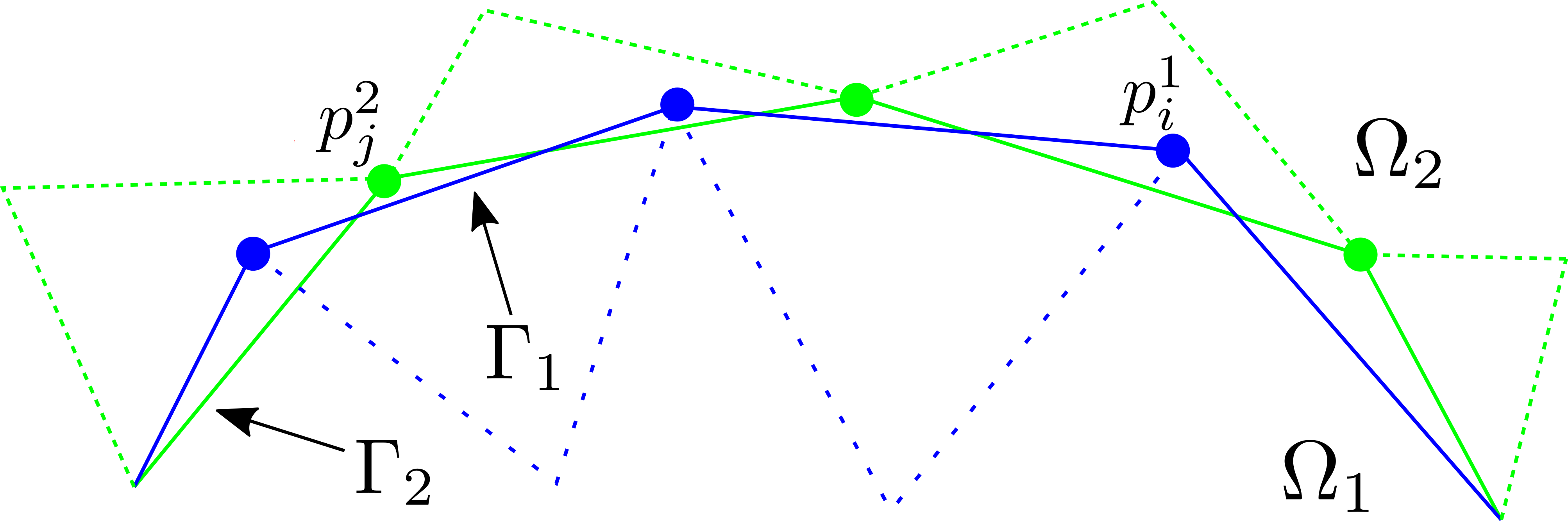}
	\end{subfigure}
	\caption{$\Gamma_1$ and $\Gamma_2$ obtained through the triangulations $\mathcal{T}_{1,h_1}$ and $\mathcal{T}_{2,h_2}$. On the left, since $\Gamma_1 = \Gamma_2 = \Gamma$, we do an interpolation; on the right, since $\Gamma_1 \not =  \Gamma_2$, we do a projection of the Dirichlet interface matching.}
	\label{Fig:Interfaces}
\end{figure}

\subsection{Steady case}
In the steady case, for each partition $\mathcal{T}_{h_i}$, we first define the finite element approximation spaces
\begin{equation}
	\label{eq_X_i}
X^{q_i}_{h_i} = \{ \mathbf{v} \in C^0(\overline{\Omega_i}): ~\mathbf{v}_{\mid T_{i,m}} \in \mathcal{Q}_{q_i}, ~ \forall T_{i,m} \in \mathcal{T}_{h_i}\}, \quad i = 1,2
\end{equation}
where $q_i$ is a chosen integer and $\mathcal{Q}_{q_i}$ represents the elements of the quads space.  

Then, considering the following finite dimensional spaces
\begin{equation}
	\label{eq_V_i}
V_{h_i} = \{\mathbf{v} \in X_{h_i}^{q_i} : ~\mathbf{v}_{\mid \partial \Omega_{i,D}} = 0 \}, \quad i = 1,2,
\end{equation}
with an abuse of notation on $\mathbf{u}_i$, we can write the Galerkin FE  approximation of weak problems$~\eqref{weak_FOM}$ as: find $\mathbf{u}_{i}(\boldsymbol{\mu}_i) \in X^{q_i}_{h_i}$ such that
\begin{equation}
\label{descrete_variationa_problems}
a_i(\mathbf{u}_{i}(\boldsymbol{\mu}_i),\mathbf{v}_{i};\boldsymbol{\mu}_i) = \mathcal{F}(\mathbf{v}_{i};\boldsymbol{\mu}_i) \quad \forall \mathbf{v}_{i} \in V_{h_i}.
\end{equation}
Moreover, we denote by $N_i$ the dimension of $V_{h_i}$, $i=1,2$. Each of the two problems in $\eqref{descrete_variationa_problems}$ yields the solution of a linear system of equations. Indeed, if for each $i = 1,2$ we define a set of basis functions $\{\boldsymbol{\varphi}_i^{(j)}\}_{j=1}^{N_i} $ for the finite-dimensional space $V_{h_i}$, then $\mathbf{v}_{i}$ can be represented in a unique form as 
\begin{equation}
\label{Eq:finite_discretization} \mathbf{v}_{i} = \sum^{N_i}_{j=1}\mathbf{v}_{i}^{(j)}\boldsymbol{\varphi}_i^{(j)} \quad \text{with } \mathbf{v}_i = (\mathbf{v}_{i}^{(1)},\dots,\mathbf{v}_{i}^{(N_i)})^{T} \in \mathbb{R}^{N_i}.\end{equation}
Then, setting $\mathbf{u}_{i}(\boldsymbol{\mu}_i) = \sum^{N_i}_{j=1}\mathbf{u}_{i}^{(j)}(\boldsymbol{\mu}_i)\boldsymbol{\varphi}_i^{(j)}$ and by $\mathbf{u}_{N_i}(\boldsymbol{\mu}_i)$ the vector containing the unknown coefficients $\mathbf{u}_{i}^{(j)}(\boldsymbol{\mu}_i)$, problem$~\eqref{descrete_variationa_problems}$ is equivalent to: find $\mathbf{u}_i(\boldsymbol{\mu}_i) \in \mathbb{R}^{N_i}$ such that
$$ \sum^{N_i}_{j=1} a_i(\boldsymbol{\varphi}_i^{(j)},\boldsymbol{\varphi}_i^{(k)};\boldsymbol{\mu}_i)  \mathbf{u}_{i}^{(j)}(\boldsymbol{\mu}_i) = \mathcal{F}_i(\boldsymbol{\varphi}_i^{(k)};\boldsymbol{\mu}_i) \quad \forall i = 1,\dots,N_i,
$$
that corresponds to
\begin{equation}
\label{FOM_discretization_system}
\mathbb{A}_{N_i}(\boldsymbol{\mu}_i) \mathbf{u}_{N_i}(\boldsymbol{\mu}_i) = \mathbf{f}_{N_i}(\boldsymbol{\mu}_i)
\end{equation}
where $\mathbb{A}_{N_i}(\boldsymbol{\mu}_i) \in \mathbb{R}^{N_i \times N_i}$ is the stiffness matrix with elements $(\mathbb{A}_{N_i})_{kj}(\boldsymbol{\mu}_i) = a_i(\boldsymbol{\varphi}_i^{(j)},\boldsymbol{\varphi}_i^{(k)};\boldsymbol{\mu}_i)$ and $\mathbf{f}_{N_i}(\boldsymbol{\mu}_i) \in \mathbb{R}^{N_i}$ is the vector with elements $(\mathbf{f}_{N_i})_k(\boldsymbol{\mu}_i) = \mathcal{F}_i(\boldsymbol{\varphi}_i^{(k)};\boldsymbol{\mu}_i)$, $j,k=1,\dots,N_i$, $i =1,2$. Hereon, with a slight abuse of notation, we will write $\mathbf{u}_{N_i}(\boldsymbol{\mu}_i) = \mathbf{u}_i(\boldsymbol{\mu}_i)$.
  
Finally, at the interface of the two domains, we prescribe Dirichlet boundary conditions, meaning that 
$$ \mathbf{u}_{2_{\mid \Gamma}}(\boldsymbol{\mu}_2) = \mathbf{u}_{1_{\mid \Gamma}}(\boldsymbol{\mu}_1)$$
if the two interfaces are conforming, or 
$$ \mathbf{u}_{2_{\mid \Gamma_2}}(\boldsymbol{\mu}_2) = \Pi \mathbf{u}_{1_{\mid \Gamma_1}}(\boldsymbol{\mu}_1)$$
in case of nonconforming interfaces, where $\Pi$ is a suitable interpolation operator. In Section $\ref{Sub_numerical_test}$, such $\Pi$ is a linear interpolation operator, \emph{i.e.}, we interpolate the high fidelity Dirichlet interface data through a linear interpolation algorithm based on VTK function \cite{Schroeder}. 

\begin{remark} Note that the reduced order technique here presented can be applied regardless of the high fidelity interpolation scheme considered. Moreover, in the online computation of the ROM, such high fidelity interpolation is substituted by the DEIM. Thus, in principal, any interpolation method can be considered in the high fidelity formulation, according to the precision required by the solution, \emph{e.g.} INTERNODES method define $\Pi$ by means of Rescaled Localized Radial Basis Functions (RL-RBF) \cite{Gervasio2019,Deparis2014ARL}.\end{remark}

\subsection{Unsteady case}
\label{Subsec:unsteady_probl}
The high-fidelity discretization of a time dependent problem yields a dynamical system, obtained through the spatial discretization of the corresponding weak form at each time instant $t$ \cite{QuarteroniValli1999}. Indeed,  considering the same FE approximation spaces $\eqref{eq_X_i}$ and finite dimensional spaces $\eqref{eq_V_i}$ -- with a slight  abuse of notation on $\mathbf{u}_i(t;\boldsymbol{\mu}_i)$ -- the Galerkin approximation of problem $\eqref{Eq:unsteady_weak_form}$ reads as: for each $t>0$, find $\mathbf{u}_{i}(t;\boldsymbol{\mu}_i) \in V_{h_i}$ such that
\begin{equation}
\label{Eq:high_fidelity_steady_form}\int_{\Omega_i}\frac{\partial \mathbf{u}_{i}(t;\boldsymbol{\mu}_i)}{\partial t} \mathbf{v}_{i} d\Omega_i + a_i(\mathbf{u}_{i}(t;\boldsymbol{\mu}_i),\mathbf{v}_{i};\boldsymbol{\mu}_i) = \int_{\Omega_i}\mathbf{f}_i(t;\boldsymbol{\mu}_i)\mathbf{v}_{i} d\Omega_i \quad \mathbf{v}_{i} \in V_{h_i},\end{equation}
with $\mathbf{u}_{i}(0;\boldsymbol{\mu}_i) = \mathbf{u}_{i,0}(\boldsymbol{\mu}_i)$. To obtain the algebraic form of $\eqref{Eq:high_fidelity_steady_form}$,
for each $t>0$ we express the Galerkin FE solution  as
$$ \mathbf{u}_{i}(\mathbf{x},t) = \sum^{N_i}_{j=1}\mathbf{u}_{i}^{(j)}(t)\boldsymbol{\varphi}_i^{(j)}(\mathbf{x}) \quad \text{with } \mathbf{u}_{N_i}(t) = (\mathbf{u}_{i}^{(1)}(t),\dots,\mathbf{u}_{i}^{(N_i)}(t))^{T} \in \mathbb{R}^{N_i};$$
then, it is straightforward to obtain the discretized finite element formulation: find $\mathbf{u}_{N_i}(t;\boldsymbol{\mu}) \in \mathbb{R}^{N_i}$ such that 
\begin{equation}  
\label{Eq:discretization_formula_time} 
\mathbb{M}_{N_i}\frac{d}{dt}\mathbf{u}_{N_i}(t;\boldsymbol{\mu}_i) + \mathbb{A}_{N_i}(\boldsymbol{\mu}_i)\mathbf{u}_{N_i}(t;\boldsymbol{\mu}_i) = \mathbf{f}_{N_i}(t;\boldsymbol{\mu}_i),\end{equation}
where $\mathbb{M}_{N_i} \in \mathbb{R}^{N_i \times N_i}$ is the mass matrix with component $(\mathbb{M}_{N_i})_{kj} = \int_{\Omega_i}\boldsymbol{\varphi}_i^{(j)}\cdot\boldsymbol{\varphi}_i^{(k)}d\Omega_i$.  
While the space dependency is treated with the FEM as in the steady case, the time-dependency can be handled using different numerical schemes \cite{QuarteroniValli1999,Brenan1995,Kreiss2014}, such as, \emph{e.g}, backward differentiation formulas (BDF). Introducing a discretization in time and a corresponding index $n$  to denote discrete time instant as $t^n = n\Delta t$, we get 
$$ \mathbf{u}_i^n(\boldsymbol{\mu}_i) \simeq \mathbf{u}_i(t^n;\boldsymbol{\mu}_i) \quad \forall n = 0, \dots, N_t,$$
and we can approximate the time derivative as
$$\frac{\partial \mathbf{u}_i^{n+1}(\boldsymbol{\mu}_i)}{\partial t} \simeq \frac{\mathbf{u}_i^{n+1}(\boldsymbol{\mu}_i) - \mathbf{u}_i^{n}(\boldsymbol{\mu}_i)}{\Delta t} \qquad \forall n=0,\dots,N_t-1.$$
Here $N_t$ is the total number of selected time instants and $\Delta t = \frac{T}{N_t}$ denotes the time step. Note that the same time discretization must be used in both subproblems. 

Therefore, coming back to the same high fidelity discretization of Section $\ref{Sec:high_fidelity_dis}$, in algebraic form, the time-dependent problems become: find $\mathbf{u}_{i}^{n+1}(\boldsymbol{\mu}_i) \in \mathbb{R}^{N_i}$ such that
\begin{equation}
\label{FOM_discretization_syste_time}
\begin{cases}
\left(\frac{\mathbb{M}_{N_i}}{\Delta t} +  \mathbb{A}_{N_i}(\boldsymbol{\mu}_i)\right) \mathbf{u}_{N_i}^{n+1}(\boldsymbol{\mu}_i) =  \mathbf{f}_{N_i}^{n+1}(\boldsymbol{\mu}_i) + \frac{\mathbb{M}_{N_i}}{\Delta t}\mathbf{u}_{N_i}^{n}(\boldsymbol{\mu}_i), &n = 0,\dots, N_t-1,\\
\mathbf{u}_{N_i}^{0}(\boldsymbol{\mu}_i) = \mathbf{u}_{N_i,0}(\boldsymbol{\mu}_i).  
\end{cases}
\end{equation}
		
Regarding the Dirichlet interface conditions, we end up with 
$$ \mathbf{u}_{2_{\mid \Gamma}}^{n+1}(\boldsymbol{\mu}_2) = \mathbf{u}_{1_{\mid \Gamma}}^{n+1} (\boldsymbol{\mu}_1)\quad \forall n = 0,\dots,N_t-1$$
if the two interfaces are conforming, or 
$$ \mathbf{u}_{2_{\mid \Gamma_2}}^{n+1}(\boldsymbol{\mu}_2) = \Pi \mathbf{u}_{1_{\mid \Gamma_1}}^{n+1}(\boldsymbol{\mu}_1)  \quad \forall n = 0,\dots,N_t-1$$
in the non-conforming interfaces case, where $\Pi$ is an interpolation operator. As for the steady case, in this work $\Pi$ represents a linear interpolation based on VTK function \cite{Schroeder}.

\section{Reduced order modeling}
\label{Sub_ROM_subproblems}
The proposed ROM strategy is a modular procedure aiming at reducing separately the three main parts of the one-way coupled problem considered, that is, the master model, the interface Dirichlet boundary conditions, and the slave model, relying on a POD-Galerkin-DEIM approach \cite{Benner2017, Hesthaven2016,QuarteroniManzoniNegri2016}. 
	
Thus, for both the master and the slave models, the goal is to approximate the FOM solution by means of a small number of basis functions (that is, a reduced basis) obtained from a set of FOM snapshots -- that is, solutions obtained through the FOM for selected values of the parameters -- while we seek for a low-dimensional representation of the parametric Dirichlet data. Also in this latter case, we therefore exploit a well chosen set of basis functions, starting from a set of snapshots of the Dirichlet data. The resulting reduction strategy to handle them can then be used as an efficient interpolation or projection method across the non-conforming interface grids. For the sake of notation, in this Section we introduce the reduction procedure used for both the master and slave problems, while a detailed description of the treatment of interface data will be addressed in the following Section. Note that the reduction of the slave problem depends not only on the number of basis functions introduced to approximate its solution, but also on the chosen approximation of the interface data. 

To simplify the notation, we remark that in this Section we denote by $\mathbf{u}_i$, $i = 1,2$, the algebraic representation of the models solutions, \emph{i.e.} $\mathbf{u}_{N_i}$, both in the steady and the unsteady case.

\subsection{Steady case}

Since the slave model is dependent on non-homogeneous Dirichlet conditions at least at the interface, a lifting technique must be applied. Therefore, we can express the slave solution as 
\[
\mathbf{u}_{2}(\boldsymbol{\mu}_2) = \tilde{\mathbf{u}}_{2}(\boldsymbol{\mu}_2) + \mathbf{u}_{2,D}(\boldsymbol{\mu}_2)
\]
 where $\mathbf{u}_{2,D}(\boldsymbol{\mu}_2)$ is the lifting vector such that $\mathbf{u}_{2_{\mid \partial\Omega_{2,D}}}(\boldsymbol{\mu}_2) = \mathbf{u}_{2,D}(\boldsymbol{\mu}_2)$. Then, the FOM slave model can be replaced by an equivalent FOM involving Dirichlet boundary conditions, in which the contribution of the Dirichlet data has been moved to the right hand side: find $\tilde{\mathbf{u}}_{2}(\boldsymbol{\mu}_2)$ such that
\begin{equation}
\label{Eq:slave_separate_bc}
\mathbb{A}_{N_2}(\boldsymbol{\mu}_2) \tilde{\mathbf{u}}_{2}(\boldsymbol{\mu}_2) = \mathbf{f}_{N_2}(\boldsymbol{\mu}_2) - \mathbb{A}_{N_2}\mathbf{u}_{2,D}(\boldsymbol{\mu}_2). 
\end{equation}

Following a standard POD-Galerkin approach (see, e.g., \cite{QuarteroniManzoniNegri2016}) during the offline stage, we define the set of snapshots for the master model as $\mathbf{S}_1 = \{ \mathbf{u}_{1}(\boldsymbol{\mu}_1^k) ,~\boldsymbol{\mu}_1^k \in \mathscr{P}^{d_1}\}$, and for the slave model considering $ \tilde{\mathbf{S}}_2 = \{\tilde{\mathbf{u}}_{2}(\boldsymbol{\mu}_2^k),~\boldsymbol{\mu}_2^k \in \mathscr{P}^{d_2}\}$, by solving the FOM problems$~\eqref{descrete_variationa_problems}$ for suitably chosen parameter values, sampled, e.g., through latin hypercube sampling \cite{Mckay1979,Iman2006}. 

For each set of snapshots, we construct a global reduced basis through POD and the corresponding matrix $\mathbb{V}_{i} \in \mathbb{R}^{N_i \times n_i}$, $n_i \ll N_i$, $i=1,2$, respectively, whose columns yield the obtained basis functions. To define the ROM, we rely on a Galerkin-RB strategy, that is, we project the original FOM onto the reduced spaces defined by $\mathbb{V}_{i}$. Thus, in the online stage, the approximation of the master solution can be sought under the approximated form 
\begin{equation}
\label{eq:reduced_master_solution}
\mathbf{u}_{1}(\boldsymbol{\mu}_1) \approx \mathbb{V}_1 \mathbf{u}_{n_1}(\boldsymbol{\mu}_1),
\end{equation} 
where $\mathbf{u}_{n_1}(\boldsymbol{\mu}_1)$ is the solution of the \emph{reduced master problem} 
$$ \mathbb{A}_{n_1}(\boldsymbol{\mu}_1)\mathbf{u}_{n_1}(\boldsymbol{\mu}_1) = \mathbf{f}_{n_1}(\boldsymbol{\mu}_1),$$
where $\mathbb{A}_{n_1}(\boldsymbol{\mu}_1) = \mathbb{V}_1^T \mathbb{A}_{N_1}(\boldsymbol{\mu}_1) \mathbb{V}_1$ and $\mathbf{f}_{n_1}(\boldsymbol{\mu}_1) = \mathbb{V}_1^T \mathbf{f}_{N_1}(\boldsymbol{\mu}_1)$.

Similarly, the approximation of the slave solution is given by
\begin{equation}\label{eq:slave_lifting_eq}\mathbf{u}_{2}(\boldsymbol{\mu}_2) \approx \mathbb{V}_2 \tilde{\mathbf{u}}_{n_2}(\boldsymbol{\mu}_2) + \mathbf{u}_{2,D}(\boldsymbol{\mu}_2),\end{equation}
where $\tilde{\mathbf{u}}_{n_2}(\boldsymbol{\mu}_2)$ is the solution of the \emph{reduced slave problem}
\begin{equation}
\label{eq:slave_reduced_model} \mathbb{A}_{n_2}(\boldsymbol{\mu}_2) \tilde{\mathbf{u}}_{n_2}(\boldsymbol{\mu}_2) = \mathbf{f}_{n_2}(\boldsymbol{\mu}_2) - \mathbb{V}_2^T \mathbb{A}_{N_2}(\boldsymbol{\mu}_2) \mathbf{u}_{2,D}(\boldsymbol{\mu}_2),\end{equation}
calling $\mathbb{A}_{n_2}(\boldsymbol{\mu}_2) = \mathbb{V}_2^T \mathbb{A}_{N_2}(\boldsymbol{\mu}_2) \mathbb{V}_2$ and $\mathbf{f}_{n_2}(\boldsymbol{\mu}_2) = \mathbb{V}_2^T \mathbf{f}_{N_2}(\boldsymbol{\mu}_2)$.

We remark that even if $\tilde{\mathbf{u}}_{n_2}(\boldsymbol{\mu}_2)$ has homogeneous Dirichlet boundary conditions, it depends on the interface data. Thus, the interface boundary conditions must be effectively considered as a parameter-dependent quantity in the reduction of the slave model.

Finally, the parametric Dirichlet data exchange at the interface of the two domains has to be reduced. Assuming that there is a portion of the slave boundary $\partial \Omega_{2,D}\backslash \Gamma$ with Dirichlet boundary conditions different from the interface data, we can write 
$$\mathbf{u}_{2,D}(\boldsymbol{\mu}_2) = \mathbf{u}_{2_{\mid \partial \Omega_{2,D}\backslash \Gamma}}(\boldsymbol{\mu}_2) + \mathbf{u}_{2_{\mid \Gamma}}(\boldsymbol{\mu}_2),$$ so that $\mathbf{u}_{2_{\mid \partial \Omega_{2,D}\backslash \Gamma}}(\boldsymbol{\mu}_2)$ is just another term of the right hand side of $\eqref{eq:slave_reduced_model}$, while $\mathbf{u}_{2_{\mid \Gamma}}(\boldsymbol{\mu}_2) = \mathbf{u}_{1_{\mid \Gamma}}(\boldsymbol{\mu}_1)$. For the sake of simplicity, here we assume that $\mathbf{u}_{2,D}(\boldsymbol{\mu}_2)$ corresponds to the data coming from the master model, meaning $\mathbf{u}_{2,D}(\boldsymbol{\mu}_2) = \mathbf{u}_{1_{\mid \Gamma}}(\boldsymbol{\mu}_1)$. The interface data reduction is described in the following Section. 

\begin{remark}The possible presence of nonlinear terms in the master or slave problems can also be treated relying on a suitable hyper-reduction technique using, for instance, the DEIM \cite{Barrault2004,Chaturantabut2009,Chaturantabut2010,Grepl2007,Maday2008,Negri2015,farhat20205}. For simplicity, in this paper we only focus on linear problems.\end{remark}

\subsection{Unsteady case}
Regarding the reduction of time-dependent problems, when dealing with POD-Galerkin ROMs the time variable can be considered as an additional parameter of the model. In this case, the set of snapshots is then made by FOM solutions collected at each time step, for each selected parameters value. As for the steady case, we consider the lifting technique to isolate the non-homogeneous Dirichlet interface conditions, \emph{i.e} the slave solution $\mathbf{u}_{2}^{n+1}(\boldsymbol{\mu}_2)$ becomes
\[
\mathbf{u}_{2}^{n+1}(\boldsymbol{\mu}_2) = \tilde{\mathbf{u}}_{2}^{n+1}(\boldsymbol{\mu}_2) + \mathbf{u}_{2,D}^{n+1}(\boldsymbol{\mu}_2) \quad \forall n = 0,\dots,N_t - 1.
\]
The slave FOM model can therefore be replaced by: find $\tilde{\mathbf{u}}_{2}^{n+1}(\boldsymbol{\mu}_2) \in \mathbb{R}^{N_2}$ such that 
\[
 \begin{cases}
 \begin{split}
 &\left(\frac{\mathbb{M}_{N_2}}{\Delta t} + \mathbb{A}_{N_2}(\boldsymbol{\mu}_2)\right) \tilde{\mathbf{u}}_{2}^{n+1}(\boldsymbol{\mu}_2) =  \mathbf{f}_{N_2}^{n+1}(\boldsymbol{\mu}_2) + \frac{\mathbb{M}_{N_2}}{\Delta t}\mathbf{u}_{2}^{n}(\boldsymbol{\mu}_2) \\
 & \qquad -\left(\frac{\mathbb{M}_{N_2}}{\Delta t} + \mathbb{A}_{N_2}(\boldsymbol{\mu}_2)\right)\mathbf{u}_{2,D(\boldsymbol{\mu}_2)}(\boldsymbol{\mu}_2)^{(n+1)}, &n = 0,\dots, N_t-1,
 \end{split}\\
 \tilde{\mathbf{u}}_2^{0}(\boldsymbol{\mu}_2) = \mathbf{u}_{2,0}(\boldsymbol{\mu}_2)-\mathbf{u}_{2,D}^{0}(\boldsymbol{\mu}_2). 
\end{cases}
\]

Considering a POD-Galerkin approach, we define the snapshots sets for each problems, including the time variable as parameter. For examples, the set of snapshots of the master problem will be $\mathbf{S}_1 = \{ \mathbf{u}_{1}^{t_1}(\boldsymbol{\mu}_1^k),\dots,\mathbf{u}_{1}^{N_t}(\boldsymbol{\mu}_1^k) \},~\boldsymbol{\mu}_1^k \in \mathscr{P}^{d_1}$, which means that the number of elements of $\mathbf{S}_1$ will be the product between $N_t$ and the number of selected parameters $\boldsymbol{\mu}_1^k$. Similarly, for the slave model we consider the set $ \tilde{\mathbf{S}}_2 = \{\tilde{\mathbf{u}}_{2}^{t_1}(\boldsymbol{\mu}_2^k),\dots, \tilde{\mathbf{u}}_{2}^{N_t}(\boldsymbol{\mu}_2^k)\},~\boldsymbol{\mu}_2^k \in \mathscr{P}^{d_2}$.

We then can represent reduced order time-depedent models using an equivalent formulation as in the case of time-independent models. In particular, recalling equation $\eqref{FOM_discretization_syste_time}$, multiplying both equation sides for the basis functions matrix $\mathbb{V}_i$ constructed according to the new set of snapshots, we can easily get the reduced order form of the master problem: find $\mathbf{u}_{n_1}^{n+1}(\boldsymbol{\mu}) \in \mathbb{R}^{n_1}$ such that
\[
\begin{cases}
\left(\frac{\mathbb{M}_{n_1}}{\Delta t} + \mathbb{A}_{n_1}(\boldsymbol{\mu}_1)\right) \mathbf{u}_{n_1}^{n+1}(\boldsymbol{\mu}_1) =  \mathbf{f}_{n_1}^{n+1}(\boldsymbol{\mu}_1) + \frac{\mathbb{M}_{n_1}}{\Delta t}\mathbf{u}_{n_1}^{n}(\boldsymbol{\mu}_1), &n = 0,\dots, N_t-1,\\
\mathbf{u}_{n_1}^{0}(\boldsymbol{\mu}_1) = \mathbf{u}_{n_1,0}(\boldsymbol{\mu}_1),
\end{cases}\]
with $\mathbb{M}_{n_1} = \mathbb{V}_1^T \mathbb{M}_{N_1} \mathbb{V}_1$, $\mathbb{A}_{n_1}(\boldsymbol{\mu}_1) = \mathbb{V}_1^T \mathbb{A}_{N_1}(\boldsymbol{\mu}_1) \mathbb{V}_1$, 
$\mathbf{f}_{n_1}^{n+1}(\boldsymbol{\mu}_1) = \mathbb{V}_1^T \mathbf{f}_{N_1}^{n+1}(\boldsymbol{\mu}_1)$ and $\mathbf{u}_{n_1,0}(\boldsymbol{\mu}_1)$ the projection of the initial solution  $\mathbf{u}_{N_1,0}(\boldsymbol{\mu}_1)$ on the master reduced basis.
Note that now the right hand side $\mathbf{f}_{N_1}^{n+1}(\boldsymbol{\mu}_1)$ is time dependent and must be reassembled at each time step. Supposing, instead, that the matrix $\mathbb{A}_{N_1}(\boldsymbol{\mu}_1)$ is time-independent, as often happens, then the reduced $\mathbb{A}_{N_1}(\boldsymbol{\mu}_1)$ can be pre-computed and stored during the offline time and directly used in the online time. 

Similarly, the slave model can also be written as in equation $\eqref{Eq:slave_separate_bc}$, so that the reduced formulation is: find $\tilde{\mathbf{u}}_{n_2}^{n+1}(\boldsymbol{\mu}_2)\in \mathbb{R}^{n_2}$ such that
\[
\begin{cases}
\begin{split}
&\left(\frac{\mathbb{M}_{n_2}}{\Delta t} + \mathbb{A}_{n_2}(\boldsymbol{\mu}_2)\right) \tilde{\mathbf{u}}_{n_2}^{n+1}(\boldsymbol{\mu}_2) =  \mathbf{f}_{n_2}^{n+1}(\boldsymbol{\mu}_2) + \frac{\mathbb{M}_{n_2}}{\Delta t}\mathbf{u}_{n_2}^{n}(\boldsymbol{\mu}_2) \\
& \qquad -\left(\frac{\mathbb{V}_2^T\mathbb{M}_{N_2}}{\Delta t} + \mathbb{V}_2^T\mathbb{A}_{N_2}(\boldsymbol{\mu}_2)\right)\mathbf{u}_{2,D}(\boldsymbol{\mu}_2)^{(n+1)}, &n = 0,\dots, N_t-1,
\end{split}\\
\tilde{\mathbf{u}}_{n_2}^{0}(\boldsymbol{\mu}_2) = \tilde{\mathbf{u}}_{n_2,0}(\boldsymbol{\mu}_2)-\mathbf{u}_{n_2,D}^{0}(\boldsymbol{\mu}_2).
\end{cases}
\]
Here $\mathbb{M}_{n_2} = \mathbb{V}_2^T \mathbb{M}_{N_2} \mathbb{V}_2$, $\mathbb{A}_{n_2}(\boldsymbol{\mu}_2) = \mathbb{V}_2^T \mathbb{A}_{N_2}(\boldsymbol{\mu}_2) \mathbb{V}_2$, $\mathbf{f}_{n_2}(\boldsymbol{\mu}_2) = \mathbb{V}_2^T \tilde{\mathbf{f}}_{N_2}^{n+1}(\boldsymbol{\mu}_2)$ and $\tilde{\mathbf{u}}_{n_2,0}(\boldsymbol{\mu}_2)$ and $\mathbf{u}_{n_2,D}^{0}(\boldsymbol{\mu}_2)$ are the projections on the slave reduced basis of the initial slave solution and initial interface Dirichlet data, respectively.
Suppose that $\mathbb{A}_{N_2}(\boldsymbol{\mu}_2)$ is time dependent, its reduced version can be stored in the offline phase, while the right hand side must be assembled at each time step. Given the equivalent formulation of $\eqref{Eq:slave_separate_bc}$, the snapshots can be computed as in the previous Section but considering also the evolution of the solution in time for each selected parameter, as previously done for the master model.

\section{Interface DEIM Reduction}
\label{Sub_ROM_interface}
Continuity of the solution at the interface of two problems domains is essential, and usually easy to be achieved if the discretization meshes are conforming. In this case, given the possible different global numbering of the degrees of freedom of the two discretizations, a map between the interface grids numbering must be computed. Moreover, in presence of non-conforming meshes, some interpolation method must be applied. Especially when large domains and very fine discretizations are considered, in both conforming or non-conforming cases, this procedure becomes quite expensive. 
When dealing with parametrized solution and, thus, parametrized interface conditions, we must reduce the computational costs of this information transfer, too. Using the DEIM \cite{Barrault2004,Chaturantabut2009,Chaturantabut2010,Grepl2007}, we aim at reducing the dimension of the data to be transferred between the two domains in the conforming case and, at the same time, to interpolate or project such data in presence of non-conforming grids. 

Moreover, given the independent nature of this kind of reduction from the master and slave model, the interface DEIM reduction can be used at first as an alternative interface interpolation or projection method during the forward computation of the coupled problem. 
In this Section we start to describe the interface reduction technique, that can be used as interpolation or projection method between the master and slave FOMs. Then, we add further details when also the master and the slave ROMs are considered.

Following the standard DEIM approach \cite{Chaturantabut2009,Chaturantabut2010}, during the offline phase we compute the set of snapshots for the Dirichlet interface data. 
Denoting by 
$$ Y_{h_k} = \{ \boldsymbol{\lambda} = \mathbf{v}_{\mid \Gamma}, ~\mathbf{v} \in X^{q_k}_{h_k} \}$$
the space of traces of functions on $\Gamma$, we define the operator 
$$\Pi : Y_{h_1}\rightarrow Y_{h_2}$$
to transfer the information from the master to the slave model. When $\Gamma_1$ and $\Gamma_2$ coincide, $\Pi$ is the classical Lagrange interpolation operator defined by the relation:
$$\Pi ~\boldsymbol{\phi}_{h_2}(\mathbf{v}_{i_{\mid \Gamma_1}}) = \boldsymbol{\phi}_{h_2}(\mathbf{v}_{i_{\mid \Gamma_1}}), \quad i = 1,\dots,N_{1_{\mid \Gamma_1}} \quad \forall \boldsymbol{\phi}_{h_2} \in Y_{h_2},$$
where $N_{1_{\mid \Gamma_1}}$ is the dimension of $\Gamma_1$. In the following, we will equally denote by $N_{2_{\mid \Gamma_2}}$ the dimension of $\Gamma_2$.

In case of conforming interface grids, meaning that $h_1 = h_2$ and $q_1 = q_2$, $\Pi$ represents the map between the interface DoFs numbering; if the interface meshes are non-conforming, $\Pi$ is the interpolation or projection operator from the master to the slave interface, \emph{e.g.} a linear interpolation as in Section $\ref{Sub_numerical_test}$. 
Then, we compute multiple instances of the parametrized interface master solution $\mathbf{u}_{1_{\mid \Gamma_1}}(\boldsymbol{\mu}_1)$ solving several times the master model, one for each $\boldsymbol{\mu}_1^k \in \mathscr{P}^{d_1}$. Then, the set of snapshots is represented by the Dirichlet data already interpolated on the slave interface grids, that is, we consider as snapshot set 
\begin{equation}
\label{Eq:snapshots_deim}
\mathbf{S}_D = \{\mathbf{u}_{2_{\mid \Gamma_2}} (\boldsymbol{\mu}_1^k), \quad \boldsymbol{\mu}_1^k \in \mathscr{P}^{d_1}\},
\end{equation}
where $\mathbf{u}_{2_{\mid \Gamma_2}} (\boldsymbol{\mu}_1^k) = \Pi(\mathbf{u}_{1_{\mid \Gamma_1}} (\boldsymbol{\mu}_1^k))$. According to the assumption of Section $\ref{Sub_ROM_subproblems}$, we have $\mathbf{u}_{2,D} (\boldsymbol{\mu}_1^k) = \Pi(\mathbf{u}_{1_{\mid \Gamma_1}} (\boldsymbol{\mu}_1^k))
.$

Then, using  POD, we build the basis $\mathbf{\Phi}_D$ to define a low-dimensional representation of Dirichlet data, that is, to approximate $\mathbf{u}_{2,D}(\boldsymbol{\mu}_1^k)$ by 
\[ \mathbf{u}_{2,D}(\boldsymbol{\mu}_1^k)\simeq \mathbf{\Phi}_D \mathbf{u}_{2,M}(\boldsymbol{\mu}_1^k) \]
where $\mathbf{u}_{2,M}(\boldsymbol{\mu}_1^k)$ is a vector of coefficients of dimension $M \ll N_{2_{\mid \Gamma_2}}$. 
Moreover, according to the DEIM construction, using a greedy algorithm \cite{Maday2008}, we select iteratively $M$ indices \vspace{-0.05cm}
\begin{equation}\label{eq:set_indices_2}\mathcal{I}_{2,D} \subset \{1,\dots,N_{2_{\mid \Gamma_2}}\}, \quad \mid  \mathcal{I}_{2,D}\mid  = M \vspace{-0.05cm} \end{equation}
from the basis $\mathbf{\Phi}_D$ which minimize the interpolation error over the snapshots set according to the maximum norm. This set of indices represents those  indices of the DoFs at which to extract the FOM data from the slave interface -- they are usually referred to as  {\em magic points} in ROM computations. Then, in the online phase, given a new parameter $\boldsymbol{\mu}_1 \in \mathscr{P}^{d_1}$, the coefficient vector $\mathbf{u}_{2,M}(\boldsymbol{\mu}_1)$ can be found imposing M interpolation constraints at the $M$ points corresponding to the selected indices, that is, by solving the linear system \vspace{-0.05cm} $$ \mathbf{\Phi}_{D_{\mid \mathcal{I}_{2,D}}} \mathbf{u}_{2,M}(\boldsymbol{\mu}_1) = \mathbf{u}_{2_{\mid \mathcal{I}_{2,D}}}(\boldsymbol{\mu}_1), \vspace{-0.05cm}$$
where $\mathbf{\Phi}_{D_{\mid \mathcal{I}_{2,D}}} \in \mathbb{R}^{M\times M }$ is the matrix containing the $\mathcal{I}_{2,D}$ rows of $\mathbf{\Phi}_{D}$. In practice, we can express  \vspace{-0.05cm}
\begin{equation}\label{eq.interface_data_reduction_eq} \mathbf{u}_{2,D}(\boldsymbol{\mu}_1) \simeq \mathbf{\Phi}_D \mathbf{\Phi}_{D_{\mid \mathcal{I}_{2,D}}}^{-1}\mathbf{u}_{2_{\mid \mathcal{I}_{2,D}}}(\boldsymbol{\mu}_1). \vspace{-0.05cm}\end{equation}

To include the interpolation or projection of the Dirichlet data, during the offline phase, we replace $\mathbf{u}_{2_{\mid \mathcal{I}_{2,D}}}(\boldsymbol{\mu}_1)$ with the interface solution of the master problem in the corresponding DoFs. Therefore, for each index $i_2 \in \mathcal{I}_{2,D}$, we extract the corresponding DoF $\mathbf{p}_2$ in Cartesian coordinates and we search for \vspace{-0.05cm}
$$\mathbf{p}_1  = \min_{\mathbf{p}_1^j \in DoFs_{\Gamma_1}}(\text{dist}(\mathbf{p}_2 - \mathbf{p}_1^j)), \vspace{-0.05cm}$$ 
meaning the nearest DoF in the master interface with respect to $\mathbf{p}_2$. Then, we search for the index of $\mathbf{p}_1$ in the master numerations to construct a set of indices \vspace{-0.05cm} $$\mathcal{I}_{1,D} = \{ i_1^{~i_2}\}_{i_2 \in \mathcal{I}_{2,D}}. \vspace{-0.05cm}$$

Hence, during the online phase, the values needed to reconstruct the Dirichlet data directly on the slave interface are given by the values of the master solution $\mathbf{u}_1$ extracted in correspondence to the $i_1$-th DoF according to $\mathcal{I}_{1,D}$, \emph{i.e.} in the $i_1$-th magic point, meaning that $ \mathbf{u}_{2_{\mid \mathcal{I}_{2,D}}}(\boldsymbol{\mu}_1) = \mathbf{u}_{1_{\mid \mathcal{I}_{1,D}}}(\boldsymbol{\mu}_1)$ and  \vspace{-0.05cm}
$$  \mathbf{u}_{2,D}(\boldsymbol{\mu}_1)\simeq \mathbf{\Phi}_D \mathbf{\Phi}_{D_{\mid \mathcal{I}_{2,D}}}^{-1}\mathbf{u}_{1_{\mid \mathcal{I}_{1,D}}}(\boldsymbol{\mu}_1). \vspace{-0.05cm}$$
We summarize the interface DEIM reduction in Algorithm $\ref{alg:DEIMInterface}$ (see also Fig. $\ref{Fig:DEIM_interface_reduction}$).
\begin{figure}[h!]
	\centering
	\includegraphics[width =0.9\textwidth]{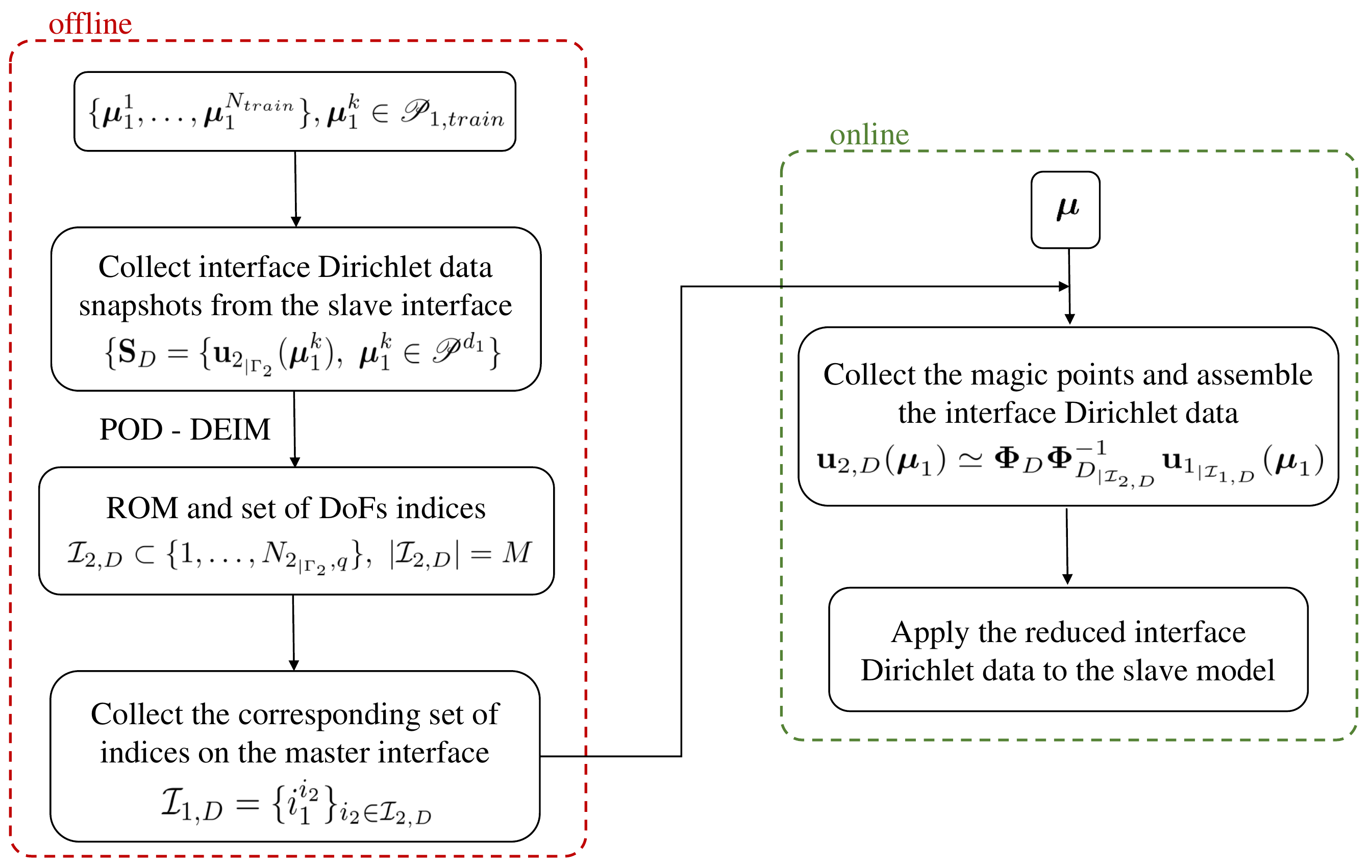}
	\smallskip
	\caption{Scheme for the offline and online phases of the interface Dirichlet data reduction.}
	\label{Fig:DEIM_interface_reduction}
\end{figure} 

\begin{remark}The indices in $\mathcal{I}_{2,D}$ are not necessary in ascending order. Moreover, $\mathcal{I}_{1,D}$ must be ordered in the same way as $\mathcal{I}_{2,D}$.\end{remark}

\begin{remark}When $\Gamma_1 = \Gamma_2 = \Gamma$ and the meshes are conforming, we can find a perfect match between corresponding interface DoFs, meaning that we are only going  to reduce the dimension of the data to be transmitted between the two problems. Instead, when the meshes are non-conforming or $\Gamma_1 \not= \Gamma_2$, in searching for corresponding DoFs we are introducing an error, especially when the two meshes considered are very coarse. Obviously, if the discretization used in the master model is much finer than the one used for the slave model, the possibility to find a good match between the DoFs in the two interfaces increases. Therefore, to minimize such error, we suggest to choose a finer mesh in the master domain than the one in the slave domain. \end{remark}
\begin{remark}Differently from the POD used for the master and slave reduction, the DEIM-based interpolation presented above does not  depend on time. Then, in the unsteady case, it is sufficient to include in the set of snapshots $\eqref{Eq:snapshots_deim}$ the time-dependent interface data to be able to apply the presented method at the domains interface (see Section $\ref{Sub_numerical_test}$).\end{remark}

\begin{algorithm}[h!]
	\caption{Interface DEIM procedure} \label{alg:DEIMInterface}
	\begin{algorithmic}[1]
		\Procedure{[ROM arrays] = Offline}{FOM arrays, $P_{1,train}$, $\epsilon_{tol_{D}}$}
		\State \emph{Dirichlet data snapshots}
		\For{$\boldsymbol{\mu}_1 \in P_{1,train}$}
		\State $\mathbf{u}_1 \gets$ solve the master model $~\eqref{FOM_discretization_system}$;
		\State $\mathbf{u}_{1_{\mid \Gamma_1}} \gets$ extract interface master solution $\Gamma_1$;
		\State $\mathbf{u}_{2_{\mid \Gamma_2}} \gets$ interpolate/project $\mathbf{u}_{1_{\mid \Gamma_1}}$ to the slave interface $\Gamma_2$;
		\State $\mathbf{S}_{D} = [\mathbf{S}_D, \mathbf{u}_{2_{\mid \Gamma_2}}]$;
		\EndFor
		\State \emph{DEIM reduced-order arrays:}
		\State $\mathbf{\Phi}_D \gets$ POD($\mathbf{S}_{D},\epsilon_{tol_{D}}$);$\quad \mathcal{I}_{2,D} \gets$ DEIM-indices($\mathbf{\Phi}_D$);
		\State \emph{Master magic points:}
		\For{$i_2 \in \mathcal{I}_{2,D}$}
		\State $p_2 \gets$ get Cartesian coordinates of $i_2$ DoF; 
		\State $p_1 = \min_{p_1^j \in DoF_{\Gamma_1}}(\text{dist}(p_2 - p_1^j)) \gets$ search the nearest DoF of $p_2$ in $\Gamma_2$;
		\State $i_1 \gets$ get the master index for $p_1$;
		\State $\mathcal{I}_{1,D} = [\mathcal{I}_{1,D}, i_1]$;
		\EndFor
		\EndProcedure
		\State
		\Procedure{[$\mathbf{u}_2$] = Online Query}{ROM arrays, FOM arrays,$\boldsymbol{\mu}_1$, $\boldsymbol{\mu}_2$}
		\State  $\mathbf{u}_1 \gets$ solve master model with $\boldsymbol{\mu}_1~\eqref{FOM_discretization_system}$;
		\State $\mathbf{u}_{1\mid {\mathcal{I}_{1,D}}} \gets$ extract magic points;
		\State $\mathbf{u}_{2_{\mid \Gamma_2}} = \mathbf{\Phi}_D \mathbf{\Phi}_{D_{\mid \mathcal{I}_{2,D}}}^{-1}\mathbf{u}_{1\mid {\mathcal{I}_{1,D}}} \gets$ DEIM approximation;
		\State apply $\mathbf{u}_{2_{\mid \Gamma_2}}$ and solve slave model with $\boldsymbol{\mu}_2~\eqref{FOM_discretization_system}$.	\EndProcedure
	\end{algorithmic}
\end{algorithm}	

During the DEIM reduction of the interface data we have  considered all vectors related to $\mathbf{u}_{1}(\boldsymbol{\mu}_1)$ and $\mathbf{u}_{2}(\boldsymbol{\mu}_2)$ in FOM form. Therefore, when also the master and the slave models are reduced, during the online computation $\mathbf{u}_{1}(\boldsymbol{\mu}_1)$ and $\mathbf{u}_{2}(\boldsymbol{\mu}_2)$ must be reconstructed from the corresponding reduced vectors $\mathbf{u}_{1}(\boldsymbol{\mu}_1)$ and $\mathbf{u}_{2}(\boldsymbol{\mu}_2)$. This is an expensive procedure especially when the problem is time-dependent and the FOM solutions must be reconstructed at each time step. To reduce the computational costs, we compute the FOM solution only for the selected magic points. This is done considering an extraction matrix $\mathbb{U} \in \mathbb{R}^{M \times N_1}$ that can be assembled during the offline phase once the index set $\mathcal{I}_{1,D}$ has been computed. Then, each row of $\mathbb{U}$ has all entries equal to zero except that in the column corresponding to the index $i_1 \in \mathcal{I}_{1,D}$ which is one. For example, suppose that $\mathcal{I}_{1,D}$ has only three elements, e.g. $\mathcal{I}_{1,D} = \{ 5, 3, N_i\}$, then $\mathbb{U} \in \mathbb{R}^{3 \times N_i}$ is such that
$$ \mathbb{U}= \left [ \begin{matrix}
0 &0 &0 &0 &1 &0 &0 &\dots &0 \\
0 &0 &1 &0 &0 &0 &0 &\dots &0 \\
0 &0 &0 &0 &0 &0 &0 &\dots &1
\end{matrix}\right]. $$
Therefore, according to equation $\eqref{eq:reduced_master_solution}$, the magic points can be directly computed as 
\begin{equation}\label{eq:extraction_process} \mathbf{u}_{1_{\mid \mathcal{I}_{1,D}}}(\boldsymbol{\mu}_1) = \mathbb{U}\mathbb{V}_1\mathbf{u}_{n_1}(\boldsymbol{\mu}_1),\end{equation}
and the Dirichlet data can be found through
$$  \mathbf{u}_{2,D}(\boldsymbol{\mu}_1)\simeq \mathbf{\Phi}_D \mathbf{\Phi}_{D_{\mid \mathcal{I}_{2,D}}}^{-1}\mathbb{U}\mathbb{V}_1\mathbf{u}_{n_1}(\boldsymbol{\mu}_1).$$
Since the basis function matrices $\boldsymbol{\Phi}$ and $\mathbb{V}_1$ and the extraction matrix $\mathbb{U}$ are parameters independent, during the offline phase the matrices product $\mathbf{\Phi}_D \mathbf{\Phi}_{D_{\mid \mathcal{I}_{2,D}}}^{-1}\mathbb{U}\mathbb{V}_1 \in \mathbb{R}^{M\times n_1}$ can be stored to be used directly in the online phase. 

Similarly, it is possible to compute directly the reduced Dirichlet term $\mathbb{V}_2^T \mathbb{A}_{N_2} \mathbf{u}_{2,D}(\boldsymbol{\mu}_1)$ in equation $\eqref{eq:slave_reduced_model}$ without reconstructing the FOM vector $\mathbf{u}_{2,D}(\boldsymbol{\mu}_1)$ dimension $N_2$. In fact, the lifting term is
$$  \mathbb{V}_2^T \mathbb{A}_{N_2} \mathbf{u}_{2,D}(\boldsymbol{\mu}_1) = \mathbb{V}_2^T \mathbb{A}_{N_2}\mathbf{\Phi}_D\mathbf{\Phi}_{D_{\mid \mathcal{I}_{2,D}}}^{-1}\mathbb{U}\mathbb{V}_1\mathbf{u}_{n_1}(\boldsymbol{\mu}_1), $$ 
where also $\mathbb{V}_2$ and $\mathbb{A}_{N_2}$ are parameter independent. Hence, in the offline phase, the matrix product $\mathbb{V}_2^T \mathbb{A}_{N_2}\mathbf{\Phi}_D\mathbf{\Phi}_{D_{\mid \mathcal{I}_{2,D}}}^{-1}\mathbb{U}\mathbb{V}_1 \in \mathbb{R}^{n_2 \times n_1}$ can be computed and stored. 
The complete reduction of the one-way coupled problem can be found in Algorithm $\ref{alg:CompleteCoupledROM}$.

\begin{remark}We remark that following the proposed procedure, in the online phase only reduced dimension operations are needed.\end{remark}

\begin{algorithm}[h!]
	\caption{Complete ROM procedure} \label{alg:CompleteCoupledROM}
	\begin{algorithmic}[1]
		\Procedure{[ROM arrays] = Offline}{FOM arrays, $P_{1,train}$,$P_{2,train}$, $\epsilon_{tol_1}$,$\epsilon_{tol_2}$, $\epsilon_{tol_{D}}$}
		\State \emph{Solution and Dirichlet data snapshots}
		\For{$\boldsymbol{\mu}_1 \in P_{1,train}$}
		\For{$\boldsymbol{\mu}_2 \in P_{2,train}$}
		\State $\mathbf{u}_1 \gets$ solve master model $~\eqref{FOM_discretization_system}$;
		\State $\mathbf{u}_2 \gets$ solve slave model $~\eqref{FOM_discretization_system}$;
		\State $\mathbf{u}_{1_{\mid \Gamma_1}} \gets$ extract interface master solution $\Gamma_1$;
		\State $\mathbf{u}_{2_{\mid \Gamma_2}} \gets$ interpolate/project $\mathbf{u}_{1_{\mid \Gamma_1}}$ to the slave interface $\Gamma_2$;
		\State $\mathbf{S}_{1} = [\mathbf{S}_1, \mathbf{u}_{1}]$;
		\State $\mathbf{S}_{2} = [\mathbf{S}_2, \mathbf{u}_{2}]$;
		\State $\mathbf{S}_{D} = [\mathbf{S}_D, \mathbf{u}_{2_{\mid \Gamma_2}}]$;
		\EndFor
		\EndFor
		\State $\mathbb{V}_{1} \gets$ POD($\mathbf{S}_1$,$\epsilon_{tol_1}$);
		\State $\tilde{\mathbf{S}}_2 \gets$ zero entries on the interface rows;
		\State $\mathbb{V}_{2} \gets$ POD($\tilde{\mathbf{S}}_2$,$\epsilon_{tol_2}$);
		\State \emph{Reduced-order matrices:}
		\State $\{\mathbb{A}_{n_1}, \mathbf{f}_{n_1}\} \gets $ projection of the full order master arrays onto $\mathbb{V}_1$; 
		\State $\{\mathbb{A}_{n_{2}}, \mathbf{f}_{n_2}\} \gets $ projection of the full order slave arrays onto $\mathbb{V}_2$;
		\State \emph{DEIM reduced-order arrays:}
		\State $\mathbf{\Phi}_D \gets$ POD($\mathbf{S}_{D},\epsilon_{tol}^{D}$);$\quad \mathcal{I}_{2,D} \gets$ DEIM-indices($\mathbf{\Phi}_D$);
		\State \emph{Master magic points:}
		\For{$i_2 \in \mathcal{I}_{2,D}$}
		\State $p_2 \gets$ get Cartesian coordinates of $i_2$ DoF; 
		\State $p_1 = \min_{p_1^j \in DoF_{\Gamma_1}}(\text{dist}(p_2 - p_1^j)) \gets$ search the nearest DoF of $p_2$ in $\Gamma_2$;
		\State $i_1 \gets$ get the master index for $p_1$;
		\State $\mathcal{I}_{1,D} = [\mathcal{I}_{1,D}, i_1]$;
		\EndFor
		\State $\mathbb{U} \gets$ matrix to extract the rows of $\mathbb{V}_1$ corresponding to the magic points in $\mathcal{I}_{1,D}$;
		\State save the matrix product $\mathbb{U}\mathbb{V}_1$; 
		\State save matrices product for slave lifting term $\mathbb{V}_2^T \mathbb{A}_{N_2}\mathbf{\Phi}_D\mathbf{\Phi}_{D_{\mid \mathcal{I}_{2,D}}}^{-1}\mathbb{U}\mathbb{V}_1$;
		\EndProcedure
		\State
		\Procedure{[$\mathbf{u}_2$] = Online Query}{ROM arrays, $\boldsymbol{\mu}_1$, $\boldsymbol{\mu}_2$}
		\State  $\mathbf{u}_{n_1} \gets$ solve ROM master model with $\boldsymbol{\mu}_1$;
		\State $\mathbb{V}_2^T \mathbb{A}_{N_2}\mathbf{\Phi}_D\mathbf{\Phi}_{D_{\mid \mathcal{I}_{2,D}}}^{-1}\mathbb{U}\mathbb{V}_1\mathbf{u}_{n_1}  \gets$ assemble the lifting term;
		\State $\tilde{\mathbf{u}}_{n_2} \gets$ solve ROM slave model with $\boldsymbol{\mu}_2$;
		\State $\mathbf{u}_{2} = \mathbb{V}_2 \tilde{\mathbf{u}}_{n_2} + \mathbf{\Phi}_D\mathbf{\Phi}_{D_{\mid \mathcal{I}_{2,D}}}^{-1}\mathbb{U}\mathbb{V}_1\mathbf{u}_{n_1} \gets$ assemble FOM slave solution.
		\EndProcedure
	\end{algorithmic}
\end{algorithm}	

\section{A posteriori error estimates}
\label{Sub_error_estimate}
In this Section we derive suitable a posteriori estimates for the norm of the errors obtained with the proposed reduced techniques in both the steady and the unsteady case. Since  POD and DEIM techniques are standard, here we  only consider the error between the high-fidelity slave solution and the reduced order slave solution, which is the final result of the coupled ROM, namely 
\begin{equation}\label{eq.error_slave} \| \mathbf{u}_{2}(\boldsymbol{\mu}_2) - \mathbb{V}_2  \mathbf{u}_{n_2}(\boldsymbol{\mu}_2)\|_2.
\end{equation}
To this end, we relate the slave error with the reduced order errors computed for the master solution and the interface data. Since the slave and the master solutions are constructed in two different reduced spaces, we consider only the 2-norm which can be equally applied in both cases.
\begin{remark}Even if the presented steady and unsteady estimates are of little practical use, they can be seen as an heuristic measure of the committed error, need to choose appropriate tolerances during each coupled ROM step. Indeed, given the modular nature of the model, controlling the accuracy of each reduction step, \emph{i.e.}, \emph{(i)} of the master ROM solution, \emph{(ii)} of the interface DEIM approximation and \emph{(iii)} of the slave ROM solution, ensures the final expected accuracy of the slave solution.\end{remark}

\subsection{Steady case}
Since we aim at finding an estimate of the error computed with the POD reduction and DEIM, referring to \cite{QuarteroniManzoniNegri2016,Grepl2005,Quarteroni2011,Rozza2007,Sen2006}, we can define  the discrete residual for a generic time-independent reduced problem  according to equation $\eqref{FOM_discretization_system}$ as
\begin{equation}\label{eq.discrete_residual} \mathbf{r}(\mathbf{u}_{N}(\boldsymbol{\mu})) = \mathbf{f}_{N}(\boldsymbol{\mu}) - \mathbb{A}_{N}(\boldsymbol{\mu})\mathbb{V} \mathbf{u}_{n}(\boldsymbol{\mu}),
\end{equation}
where $N$ and $n$ are the FOM and the ROM dimensions, respectively, and $\boldsymbol{\mu} \in \mathscr{P}^{d}$ is the parameters vector. Then, an error bound can be found by estimating the three  error terms separately, according to the following proposition.

\begin{prop} \label{Prop:error_steady}For any $\boldsymbol{\mu}_1 \in \mathscr{P}^{d_1}$ and $\boldsymbol{\mu}_2 \in \mathscr{P}^{d_2}$, let us denote by $\|  \mathbf{u}_{2,h_2}(\boldsymbol{\mu}_2) - \mathbb{V}_2  \mathbf{u}_{n_2}(\boldsymbol{\mu}_2)\| _2$ the slave reduced order error. Then, the following error estimates holds:
	\[ \begin{split}
	\|  \mathbf{u}_{2}(\boldsymbol{\mu}_2) - \mathbb{V}_2  \mathbf{u}_{n_2}(\boldsymbol{\mu}_2)\| _2 ~\leq &~ \frac{1}{\sigma_{min}(\mathbb{A}_{N_2}(\boldsymbol{\mu}_2))} \| \mathbf{r}_{2}(\tilde{\mathbf{u}}_{2}(\boldsymbol{\mu}_2))\| _2\\ 
	&+ \| \mathbf{\Phi}_{D_{\mid \mathcal{I}_{2,D}}}\| _2\| (\mathbb{I} - \mathbf{\Phi}_D \mathbf{\Phi}_D^{-1}) \mathbf{u}_{2,D}(\boldsymbol{\mu}_1)\|_2  +\frac{C}{\sigma_{min}(\mathbb{A}_{N_1}(\boldsymbol{\mu}_1))} \| \mathbf{r}_{1}(\mathbf{u}_{1}(\boldsymbol{\mu}_1))\| _2,
	\end{split}\]
	where $\sigma_{min}(\mathbb{A}_{N_i})$, $i=1,2$ denote the smallest singular value of $\mathbb{A}_{N_i}$, $\mathbf{r}_i$, $i=1,2$ are the discrete residual $\eqref{eq.discrete_residual}$ of the master model and slave model with homogeneous Dirichlet interface conditions respectively,  $\mathbb{I}$ is the identity matrix and $C = \| \mathbf{\Phi}_D \mathbf{\Phi}_{D_{\mid \mathcal{I}_{2,D}}}^{-1}\mathbb{U}\| _2$. \end{prop}
\begin{proof}
	Exploiting the lifting technique to apply the Dirichlet boundary conditions in $\eqref{eq:slave_lifting_eq}$, and using the triangular inequality, we can first split error $\eqref{eq.error_slave}$ in two contributions,
	\[ \begin{split}
	\|  \mathbf{u}_{2}(\boldsymbol{\mu}_2) - \mathbb{V}_2  \mathbf{u}_{n_2}(\boldsymbol{\mu}_2)\| _2 &= \|  \tilde{ \mathbf{u}}_{2}(\boldsymbol{\mu}_2) + \mathbf{u}_{2,D}(\boldsymbol{\mu}_2) - \mathbb{V}_2  \tilde{\mathbf{u}}_{n_2}(\boldsymbol{\mu}_2) - (\mathbb{V}_2 \mathbf{u}_{n_2}(\boldsymbol{\mu}_2))_{\mid \Gamma_2}\| _2 \\
	&\leq \|  \tilde{ \mathbf{u}}_{2}(\boldsymbol{\mu}_2) - \mathbb{V}_2  \tilde{\mathbf{u}}_{n_2}(\boldsymbol{\mu}_2) \| _2  +  \|  \mathbf{u}_{2,D}(\boldsymbol{\mu}_2) - (\mathbb{V}_2 \mathbf{u}_{n_2}(\boldsymbol{\mu}_2))_{\mid \Gamma_2}\| _2.
	\end{split}\]
	
	Then, since $(\mathbb{V}_2 \mathbf{u}_{n_2}(\boldsymbol{\mu}_2))_{\mid \Gamma_2}$ denote the computed interface Dirichlet data, according to $\eqref{eq.interface_data_reduction_eq}$ we can write
	\[ \begin{split}
	\|  \mathbf{u}_{2}(\boldsymbol{\mu}_2) - \mathbb{V}_2  \mathbf{u}_{n_2}(\boldsymbol{\mu}_2)\| _2 &\leq \|  \tilde{\mathbf{u}}_{2}(\boldsymbol{\mu}_2) - \mathbb{V}_2  \tilde{\mathbf{u}}_{n_2}(\boldsymbol{\mu}_2) \| _2 +  \|  \mathbf{u}_{2,D} - (\mathbb{V}_2 \mathbf{u}_{n_2}(\boldsymbol{\mu}_2))_{\mid \Gamma_2}\| _2 \\	 
	& =\|  \tilde{ \mathbf{u}}_{2}(\boldsymbol{\mu}_2) - \mathbb{V}_2  \tilde{\mathbf{u}}_{n_2}(\boldsymbol{\mu}_2) \| _2  +  \|  \mathbf{u}_{2,D}(\boldsymbol{\mu}) - \mathbf{\Phi}_D \mathbf{\Phi}_{D_{\mid \mathcal{I}_{2,D}}}^{-1}\mathbb{U}\mathbb{V}_1\mathbf{u}_{n_1}(\boldsymbol{\mu}_1)\| _2.\end{split}\]
	
	Adding and subtracting the same quantity $\mathbf{\Phi}_D\mathbf{\Phi}_{D_{\mid \mathcal{I}_{2,D}}}^{-1}\mathbb{U}\mathbf{u}_{1}(\boldsymbol{\mu}_1)$, we finally obtained a relation between the three computed errors of the following form
	\begin{equation}
	\label{Eq:error_correlation}
	\begin{split}
	\|  \mathbf{u}_{2}(\boldsymbol{\mu}_2) - \mathbb{V}_2  \mathbf{u}_{n_2}(\boldsymbol{\mu}_2)\| _2
	\leq & ~ \|  \tilde{ \mathbf{u}}_{2}(\boldsymbol{\mu}_2) -  \mathbb{V}_2\tilde{\mathbf{u}}_{n_2}(\boldsymbol{\mu}_2) \| _2 \\& \qquad \qquad +  \|  \mathbf{u}_{2,D}(\boldsymbol{\mu}_1) - \mathbf{\Phi}_D \mathbf{\Phi}_{D_{\mid \mathcal{I}_{2,D}}}^{-1}\mathbb{U} \mathbf{u}_{1}(\boldsymbol{\mu}_1)\| _2 \\& \qquad \qquad + \| \mathbf{\Phi}_D \mathbf{\Phi}_{D_{\mid \mathcal{I}_{2,D}}}^{-1}\mathbb{U}\| _2\|  \mathbf{u}_{1}(\boldsymbol{\mu}_1) - \mathbb{V}_1 \mathbf{u}_{n_1}(\boldsymbol{\mu}_1)\| _2.
	\end{split}\end{equation}
	Note that in the second term $\mathbf{u}_{2,D_{\mid \mathcal{I}_{2,D}}}(\boldsymbol{\mu}_1) = \mathbb{U}\mathbf{u}_{1}(\boldsymbol{\mu}_1)$, so that 
	\begin{equation}\label{Eq.error_DEIM} \|  \mathbf{u}_{2,D}(\boldsymbol{\mu}_1) - \mathbf{\Phi}_D \mathbf{\Phi}_{D_{\mid \mathcal{I}_{2,D}}}^{-1}\mathbb{U} ~\mathbf{u}_{1}(\boldsymbol{\mu}_1)\| _2 = \|  \mathbf{u}_{2,D}(\boldsymbol{\mu}_1) - \mathbf{\Phi}_D \mathbf{\Phi}_{D_{\mid \mathcal{I}_{2,D}}}^{-1}\mathbf{u}_{2,D_{\mid \mathcal{I}_{2,D}}}(\boldsymbol{\mu}_1)\| _2.\end{equation}
	Denoting by $C = \| \mathbf{\Phi}_D \mathbf{\Phi}_{D_{\mid \mathcal{I}_{2,D}}}^{-1}\mathbb{U}\| _2$, we can finally bound the above quantities by
	\begin{equation}\label{Eq:a_posteriori_POD}
	\begin{split}
	&\|  \tilde{ \mathbf{u}}_{2}(\boldsymbol{\mu}_2) -  \mathbb{V}_2\tilde{\mathbf{u}}_{n_2}(\boldsymbol{\mu}_2) \| _2 \leq \frac{1}{\sigma_{min}(\mathbb{A}_{N_2}(\boldsymbol{\mu}_2))} \| \mathbf{r}_{2}(\tilde{\mathbf{u}}_{2}(\boldsymbol{\mu}_2))\| _2,\\
	&\|  \mathbf{u}_{1}(\boldsymbol{\mu}_1) -  \mathbb{V}_1\mathbf{u}_{n_1}(\boldsymbol{\mu}_1) \| _2 \leq \frac{1}{\sigma_{min}(\mathbb{A}_{N_1}(\boldsymbol{\mu}_1))} \| \mathbf{r}_{1}(\mathbf{u}_{1}(\boldsymbol{\mu}_1))\| _2 
	\end{split}\end{equation}
	and  (see \cite{QuarteroniManzoniNegri2016}, chapters 3 and 10, for further details)
	$$ \|  \mathbf{u}_{2,D}(\boldsymbol{\mu}_1) - \mathbf{\Phi}_D \mathbf{\Phi}_{D_{\mid \mathcal{I}_{2,D}}}^{-1}\mathbb{U} \mathbf{u}_{1}(\boldsymbol{\mu}_1)\| _2  \leq \| \mathbf{\Phi}_{D_{\mid \mathcal{I}_{2,D}}}\| _2\| (\mathbb{I} - \mathbf{\Phi}_D \mathbf{\Phi}_D^{-1}) \mathbf{u}_{2,D}(\boldsymbol{\mu}_1)\| _2. \qedhere$$
\end{proof}

\subsection{Unsteady case}
\label{Subsec:error_time}
To find an estimate of the reduced error $\eqref{eq.error_slave}$ in the time-dependent case, following \cite{Haasdonk2011} and \cite{Wirtz2012}, we can define the generic residual 
\begin{equation}
\label{Eq:res_dynamical}
\mathbf{r}(t; \boldsymbol{\mu}) = \mathbb{A}_{N}(\boldsymbol{\mu})\mathbb{V}\mathbf{u}_{n}(t;\boldsymbol{\mu}) + \mathbf{f}_N(t;\boldsymbol{\mu}) - \mathbb{V}\frac{d}{dt}\mathbf{u}_n(t;\boldsymbol{\mu}) \qquad \forall t \in [0,T].
\end{equation}
Here, with a slight abuse of notation, we call $\mathbb{A}_N(\boldsymbol{\mu}) = - \mathbb{M}_{N}^{-1} \mathbb{A}_{N}(\boldsymbol{\mu})$ and $\mathbf{f}_{N}(t;\boldsymbol{\mu}) = \mathbb{M}_{N}^{-1} \mathbf{f}_{N}(t;\boldsymbol{\mu})$
according to the dynamical system formulation of equation $\eqref{Eq:discretization_formula_time}$, being $\boldsymbol{\mu} \in \mathscr{P}^d$ the parameters vector. For the sake of notation, we add $t$ to recall the time dependency of the model while we implicity consider the spatial dependency.

\begin{prop}
	Assuming that $\mathbb{A}_{N_i}$, $i=1,2$ are two time-invariant matrices and that their eigenvalues have non-negative real part for all parameters $\boldsymbol{\mu}_1 \in \mathscr{P}^{d_1}$ and $\boldsymbol{\mu}_2 \in \mathscr{P}^{d_2}$, respectively, then for each time $t \in [0,T]$, the following error estimate holds:
	\[
	\begin{split}
	\|  \mathbf{u}_{2}(t;\boldsymbol{\mu}_2) - \mathbb{V}_2  \mathbf{u}_{n_2}(t;\boldsymbol{\mu}_2)\|_2 \leq &C_2(\boldsymbol{\mu}_2)~\| \tilde{\mathbf{u}}_{2}(0;\boldsymbol{\mu}_2) - \mathbb{V}_2\tilde{\mathbf{u}}_{n_2}(0;\boldsymbol{\mu}_2)\| _2 + C_2(\boldsymbol{\mu}_2) \int_0^{t} \| \mathbf{r}_2(\tau;\boldsymbol{\mu}_2)\| _2 d\tau \smallskip \\
	&+ \| \mathbf{\Phi}_{D_{\mid \mathcal{I}_{2,D}}}\| _2\| (\mathbb{I} - \mathbf{\Phi}_D \mathbf{\Phi}_D^{-1}) \mathbf{u}_{2,D}(t;\boldsymbol{\mu}_2)\| _2 \smallskip \\
	&+ C_1(\boldsymbol{\mu}_1) C_3~ \|  \mathbf{u}_{1}(0;\boldsymbol{\mu}_1) - \mathbb{V}_1 \mathbf{u}_{n_1}(0;\boldsymbol{\mu}_1)\| _2 + C_1(\boldsymbol{\mu}_1) C_3 \int_0^{t} \| \mathbf{r}_1(\tau;\boldsymbol{\mu}_1)\| _2 d\tau,
	\end{split}
	\]
	where $C_1(\boldsymbol{\mu}_1)$ and $C_2(\boldsymbol{\mu}_2)$ are two constants such that
	$$ \sup_{t \in [0,T]} \| \exp(\mathbb{A}_{N_1}(\boldsymbol{\mu}_1) t)\| _2 \leq C_1(\boldsymbol{\mu}_1) \quad \text{and} \quad \sup_{t \in [0,T]} \| \exp(\mathbb{A}_{N_2}(\boldsymbol{\mu}_2) t)\| _2 \leq C_2(\boldsymbol{\mu}_2),$$
	and $C_3 = \| \mathbf{\Phi}_D \mathbf{\Phi}_{D_{\mid \mathcal{I}_{2,D}}}^{-1}\mathbb{U}\| _2$. Furthermore, $\mathbf{r}_i$, $i=1,2$ are the discrete residual $\eqref{Eq:res_dynamical}$ of the master and slave model with homogeneous Dirichlet boundary conditions, respectively, and $\mathbb{I}$ is the identity matrix.
\end{prop}
\begin{proof}
	Fixing a time instant $t \in [0,T]$, as in proposition $\ref{Prop:error_steady}$, it is possible to relate the slave error $\eqref{eq.error_slave}$ to the master and interface error according to equation $\eqref{Eq:error_correlation}$, namely
	\[
	\begin{split}
	\|  \mathbf{u}_{2}(t;\boldsymbol{\mu}_2) - \mathbb{V}_2  \mathbf{u}_{n_2}(t;\boldsymbol{\mu}_2)\| _2
	\leq & ~ \|  \tilde{ \mathbf{u}}_{2}(t;\boldsymbol{\mu}_2) -  \mathbb{V}_2\tilde{\mathbf{u}}_{n_2}(t;\boldsymbol{\mu}_2) \| _2 \\
	&+  \|  \mathbf{u}_{2,D}(t;\boldsymbol{\mu}_1) - \mathbf{\Phi}_D \mathbf{\Phi}_{D_{\mid \mathcal{I}_{2,D}}}^{-1}\mathbb{U} \mathbf{u}_{1}(t;\boldsymbol{\mu}_1)\| _2 \\&+ \| \mathbf{\Phi}_D \mathbf{\Phi}_{D_{\mid \mathcal{I}_{2,D}}}^{-1}\mathbb{U}\| _2\|  \mathbf{u}_{1}(t;\boldsymbol{\mu}_1) - \mathbb{V}_1 \mathbf{u}_{n_1}(t;\boldsymbol{\mu}_1)\| _2.
	\end{split}\]
	
	Moreover, following proposition 4.1 of \cite{Haasdonk2011} (see appendix $\ref{App:error_estimation}$ for the complete proof), given a time dependent problem reduced with a POD method, it holds that, for each $\boldsymbol{\mu} \in \mathscr{P}^{d}$,
	$$ \|  \mathbf{u}(t;\boldsymbol{\mu}) - \mathbb{V}\mathbf{u}_n(t;\boldsymbol{\mu}) \| _2  \leq C(\boldsymbol{\mu}) \left ( \|  \mathbf{u}(0;\boldsymbol{\mu}) - \mathbb{V}\mathbf{u}_n(0;\boldsymbol{\mu}) \| _2 + \int_0^t \| \mathbf{r}(\tau;\boldsymbol{\mu})\| _2 d\tau \right ),$$
	where $\sup_{t \in [0,T]}\| \exp(\mathbb{A}_N t)\| _2$ if $\mathbb{A}$ is time invariant and has eigenvalues with negative real part. This means that we can write:
	\[ \begin{split}\| \tilde{\mathbf{u}}_{2,h_2}(t;\boldsymbol{\mu}_2) - \mathbb{V}_2\tilde{\mathbf{u}}_{n_2}(t;\boldsymbol{\mu}_2)\| _2 ~\leq~ &C_2(\boldsymbol{\mu}_2) \| \tilde{\mathbf{u}}_{2}(0;\boldsymbol{\mu}_2) - \mathbb{V}_2\tilde{\mathbf{u}}_{n_2}(0;\boldsymbol{\mu}_2)\| _2 +C_2(\boldsymbol{\mu}_2)+ \int_0^t \| \mathbf{r}_2(\tau;\boldsymbol{\mu}_2)\| _2 d\tau\end{split}\]
	and
	\[ \begin{split} 
	\|  \mathbf{u}_{1}(t;\boldsymbol{\mu}_1) - \mathbb{V}_1 \mathbf{u}_{n_1}(t;\boldsymbol{\mu}_1)\| _2 ~\leq~ &C_1(\boldsymbol{\mu}_1) \|  \mathbf{u}_{1}(0;\boldsymbol{\mu}_1) - \mathbb{V}_1 \mathbf{u}_{n_1}(0;\boldsymbol{\mu}_1)\| _2  +C_1(\boldsymbol{\mu}_1) \int_0^{t} \| \mathbf{r}_1(\tau;\boldsymbol{\mu}_1)\| _2 d\tau.
	\end{split} \]
	
	Furthermore, as explained in Section $\ref{Sub_ROM_interface}$, the DEIM applied at the interface is independent from the time variable, which means that the corresponding error can be estimated as in the steady case \cite{QuarteroniManzoniNegri2016} as
	$$ \|  \mathbf{u}_{2,D}(t;\boldsymbol{\mu}_1) - \mathbf{\Phi}_D \mathbf{\Phi}_{D_{\mid \mathcal{I}_{2,D}}}^{-1}\mathbb{U} \mathbf{u}_{1}(t;\boldsymbol{\mu}_1)\| _2  \leq \| \mathbf{\Phi}_{D_{\mid \mathcal{I}_{2,D}}}\| _2\| (\mathbb{I} - \mathbf{\Phi}_D \mathbf{\Phi}_D^{-1}) \mathbf{u}_{2,D}(t;\boldsymbol{\mu}_1)\| _2.$$ 
	
	Hence, the proof is complete denoting $C_3 = \| \mathbf{\Phi}_D \mathbf{\Phi}_{D_{\mid \mathcal{I}_{2,D}}}^{-1}\mathbb{U}\| _2$.
\end{proof}

\section{Numerical results}
\label{Sub_numerical_test}
In this Section we investigate the numerical performances of the proposed reduced strategies by means of three one-way coupled problems. We present a detailed comparison of the numerical results by looking at their efficiency and accuracy, exploring the behavior of ROM coupled problems when dealing with a steady-steady problem, an unsteady-steady problem and an unsteady-unsteady problem. All simulations, both in the online and offline stages, are performed in serial on a notebook with Intel Core i7-10710U processor and 16 GB of RAM. The mathematical models and numerical methods presented in this Section have been implemented in C++ and Python languages and are based on life$^{\text{x}}$ (\url{https://lifex.gitlab.io}), a new in-house high-performance C++ FE library mainly focused on cardiac applications based on deal.II FE core \cite{dealII93} (\url{https://www.dealii.org}).

\subsection{Test case $i$: steady model - steady model}
Let us consider a time independent coupled problem made by a reaction-diffusion problem and a Laplacian as master and slave models, respectively, with suitable boundary conditions, meaning:
\begin{equation}
\begin{cases}
- \nabla \cdot (\alpha \nabla u) + \beta u = f &\text{in } \Omega_1\\
u = 0 &\text{on }\partial \Omega_{1,D}\backslash \Gamma \\
\frac{\partial u}{\partial n_1}= 0 &\text{on }\Gamma,
\end{cases}
\end{equation} 
and 
\begin{equation}
\begin{cases}
- \Delta v = 0 &\text{in }\Omega_2\\
\frac{\partial v}{\partial n_2} = 0 &\text{on }\partial \Omega_{2,N},
\end{cases}
\end{equation}
coupled with the following Dirichlet boundary conditions at the interface
$$v = u \quad \text{on } \Gamma.$$
We define $f(x,y,z) = \frac{\pi}{4}y x^2 \sin \left( \frac{\pi}{2} y\right) e^{z-1}$, and we vary the two parameters $\alpha$ and $\beta$ in $[0.5,5]$ with a latin hypercube sampling (LHS) distribution.

The models are solved in two three-dimensional domains represented by two concentric hollow spheroids centered in the origin. In particular, we define $\Omega_1$ as the internal spheroids with inner and outer radius equal to $0.5~m$ and $1.5~m$, while $\Omega_2$ is the external one, with inner and outer radius equal to $1.5~m$ and $3.5~m$. Therefore, the interface $\Gamma$ between the two domains is the spherical surface corresponding to the external boundary $\Gamma_1$ of $\Omega_1$ and the internal boundary $\Gamma_2$ of $\Omega_2$ (see Fig. $\ref{Fig:RD_domains}$).

We solve this coupled problem first considering a different discretization, meaning $h_1 \not= h_2$ on $\Omega_1$ and $\Omega_2$ and the same FEM order and, then, considering the same discretization on $\Omega_1$ and $\Omega_2$ but different FE order, meaning $q_1 \not = q_2$.

\begin{figure}[h!]
	\centering
	\begin{subfigure}{.3\textwidth}
		\centering
		\includegraphics[width=1\textwidth]{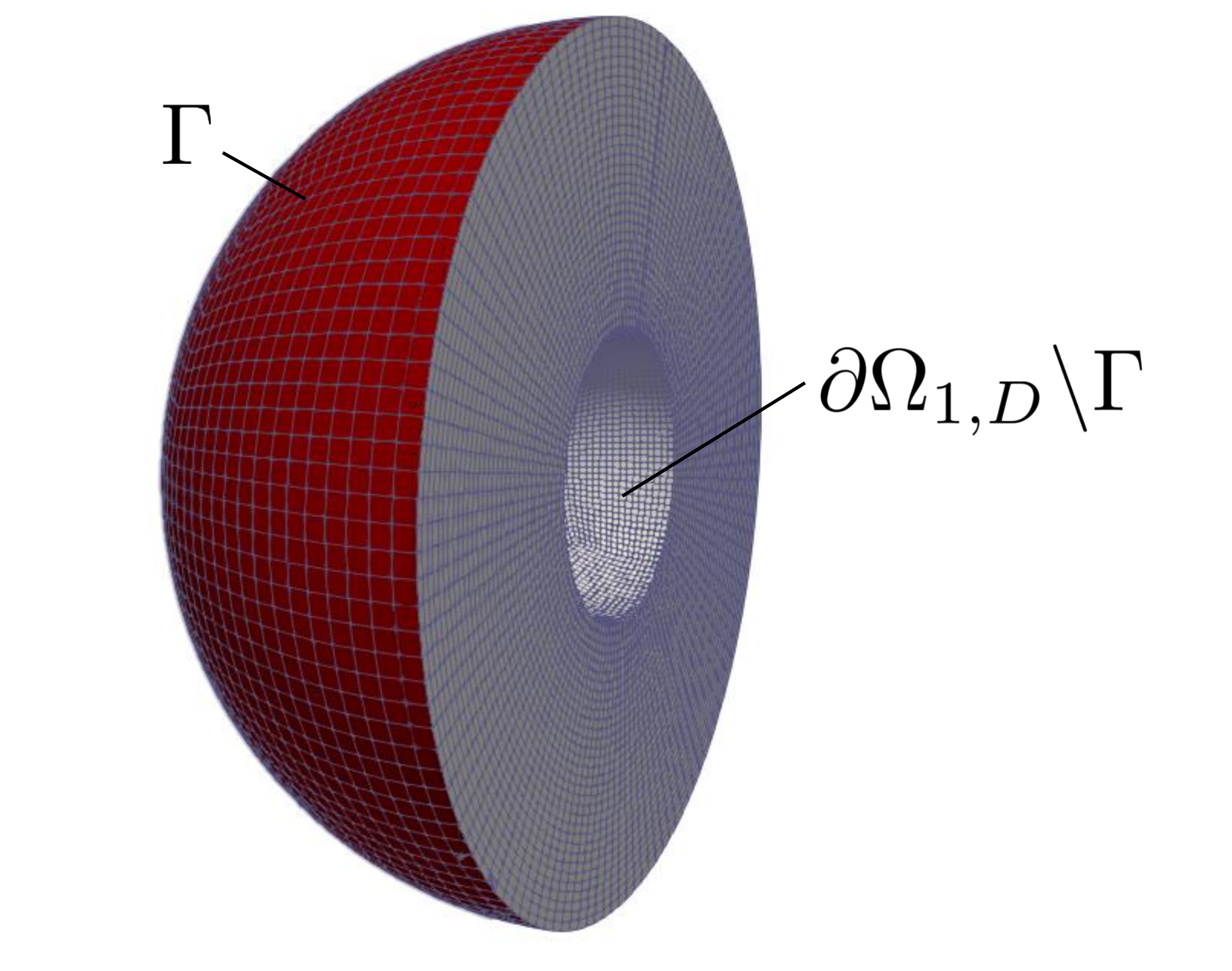}
		\label{slave_domain}
	\end{subfigure}
	\hfill
	\begin{subfigure}{.3\textwidth}
		\centering
		\includegraphics[width=0.87\textwidth]{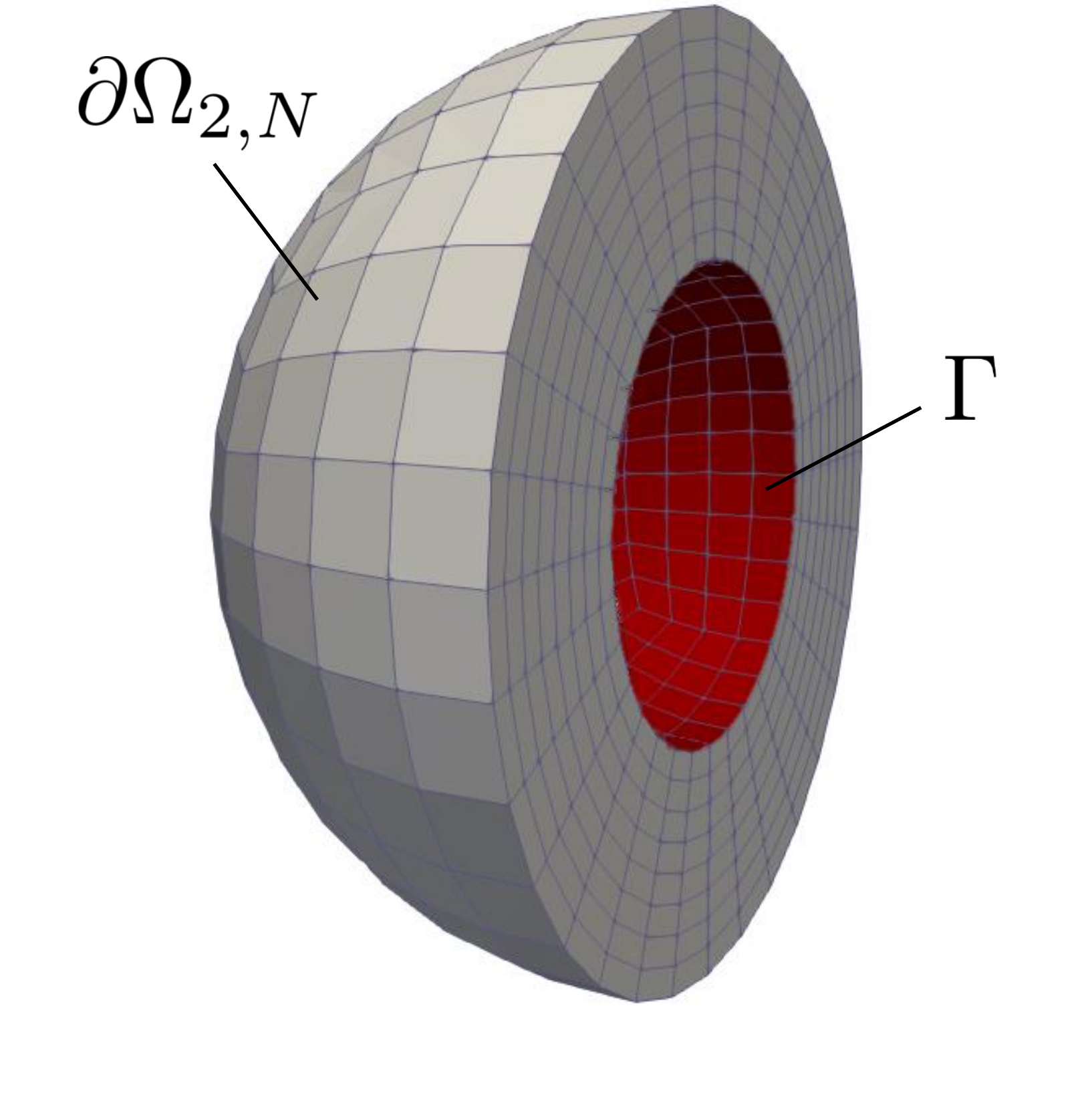}
		\label{master_domain}
	\end{subfigure}
	\hfill
	\begin{subfigure}{.3\textwidth}
		\centering
		\includegraphics[width=1\textwidth]{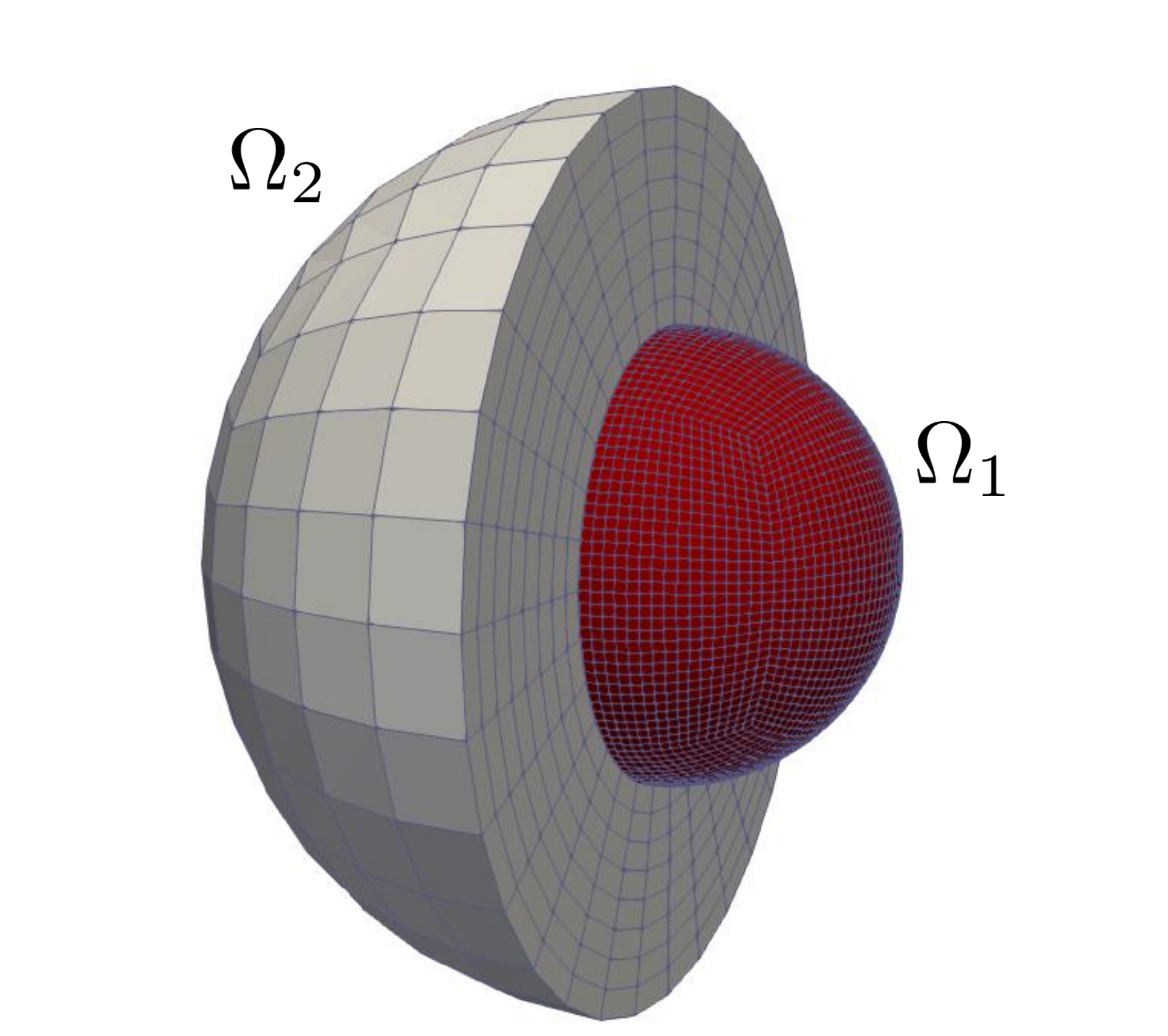}
		\label{RD_domains}
	\end{subfigure}
	\vspace{-0.25cm}
	\caption{\emph{Test case i - different discretizations.} Master (left) and slave (center) domains, two hollow spheroids discretized with different meshes, \emph{i.e.}  $h_1 = 0.107953~m$ and $h_2 = 0.421191~m$. The master domain is inside the slave domain (right). In red the interface boundary $\Gamma$.}
	\vspace{-0.25cm}
	\label{Fig:RD_domains}
\end{figure}

\subsection{Different discretization}
Fig. $\ref{fig:snapshots_probl_1}$ shows some numerical solutions obtained with the FOM considering FEM-$\mathbb{Q}_1$ for both master and slave models but different discretizations. In particular, we choose $h_1 = 0.107953~m$ and $h_2 = 0.421191~m$ so that $N_1 = 202818$ and $N_2 = 26146$ (see Fig. $\ref{Fig:RD_domains}$).

\begin{figure}[h!]
	\centering
	\begin{subfigure}{.3\textwidth}
		\centering
		\includegraphics[width=1\textwidth]{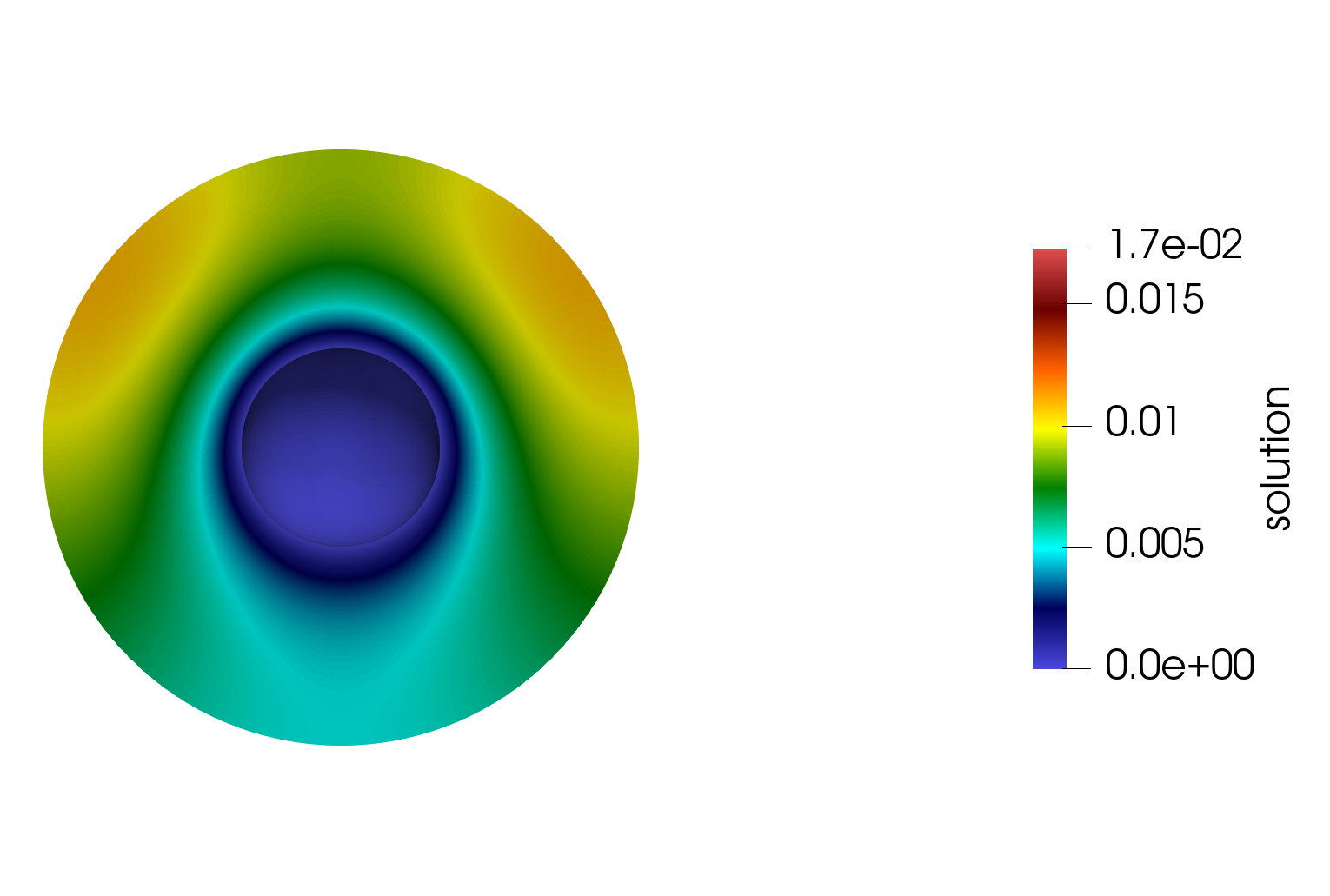}
	\end{subfigure}
	$\quad$
	\begin{subfigure}{.3\textwidth}
		\centering
		\includegraphics[width=1\textwidth]{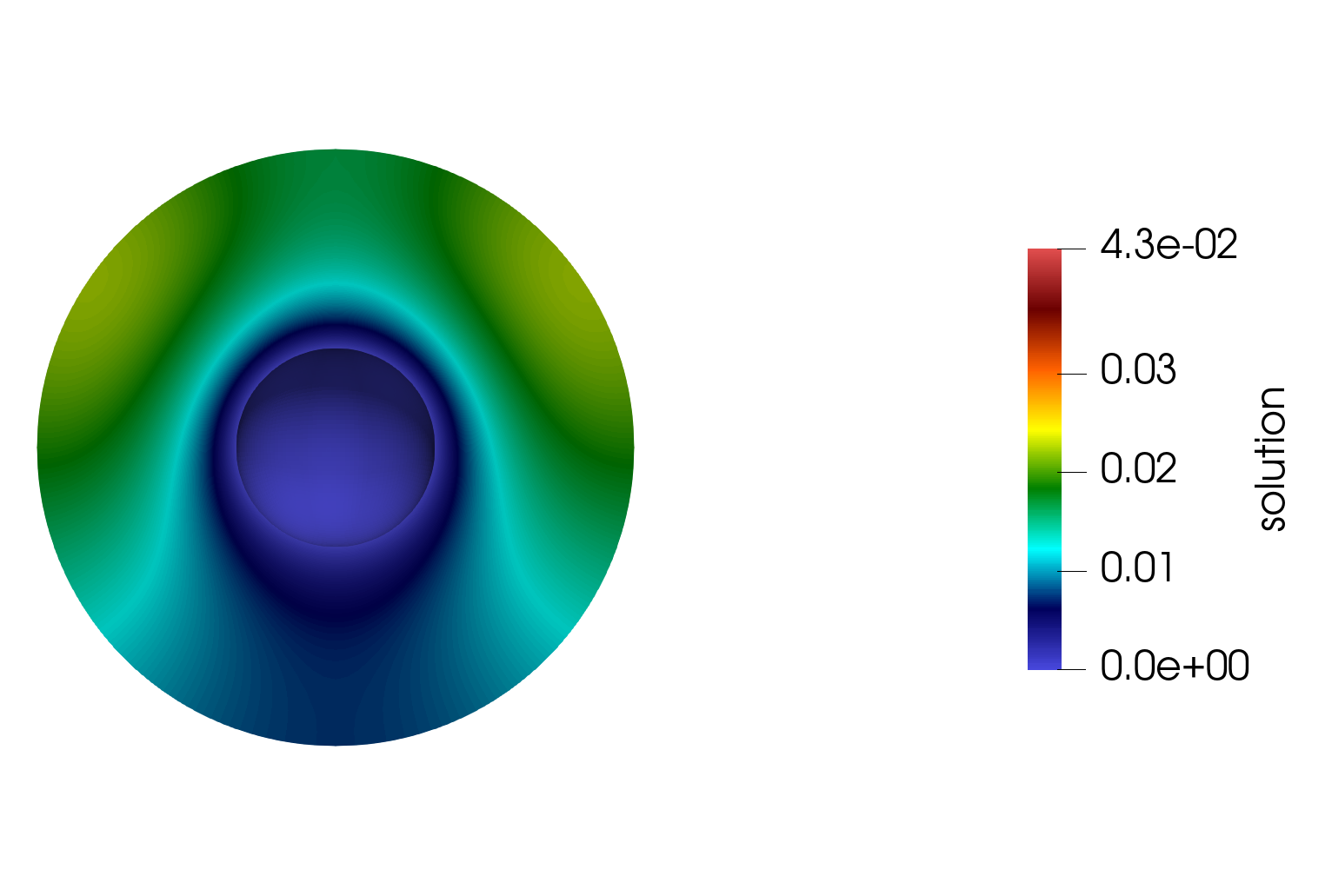}
	\end{subfigure}
	$\quad$
	\begin{subfigure}{.3\textwidth}
		\centering
		\includegraphics[width=1\textwidth]{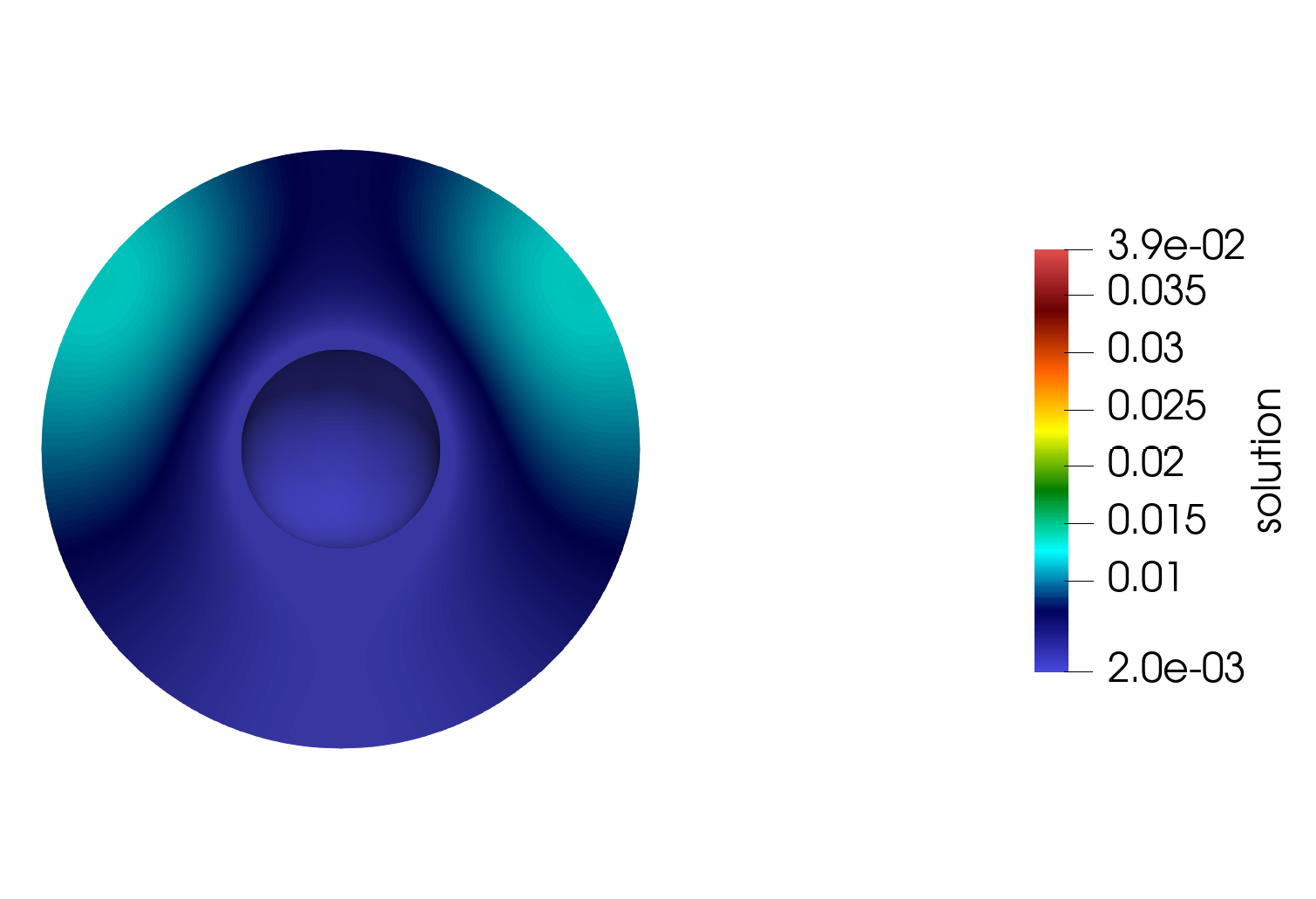}
	\end{subfigure} \\
	\begin{subfigure}{.3\textwidth}
		\centering
		\includegraphics[width=1\textwidth]{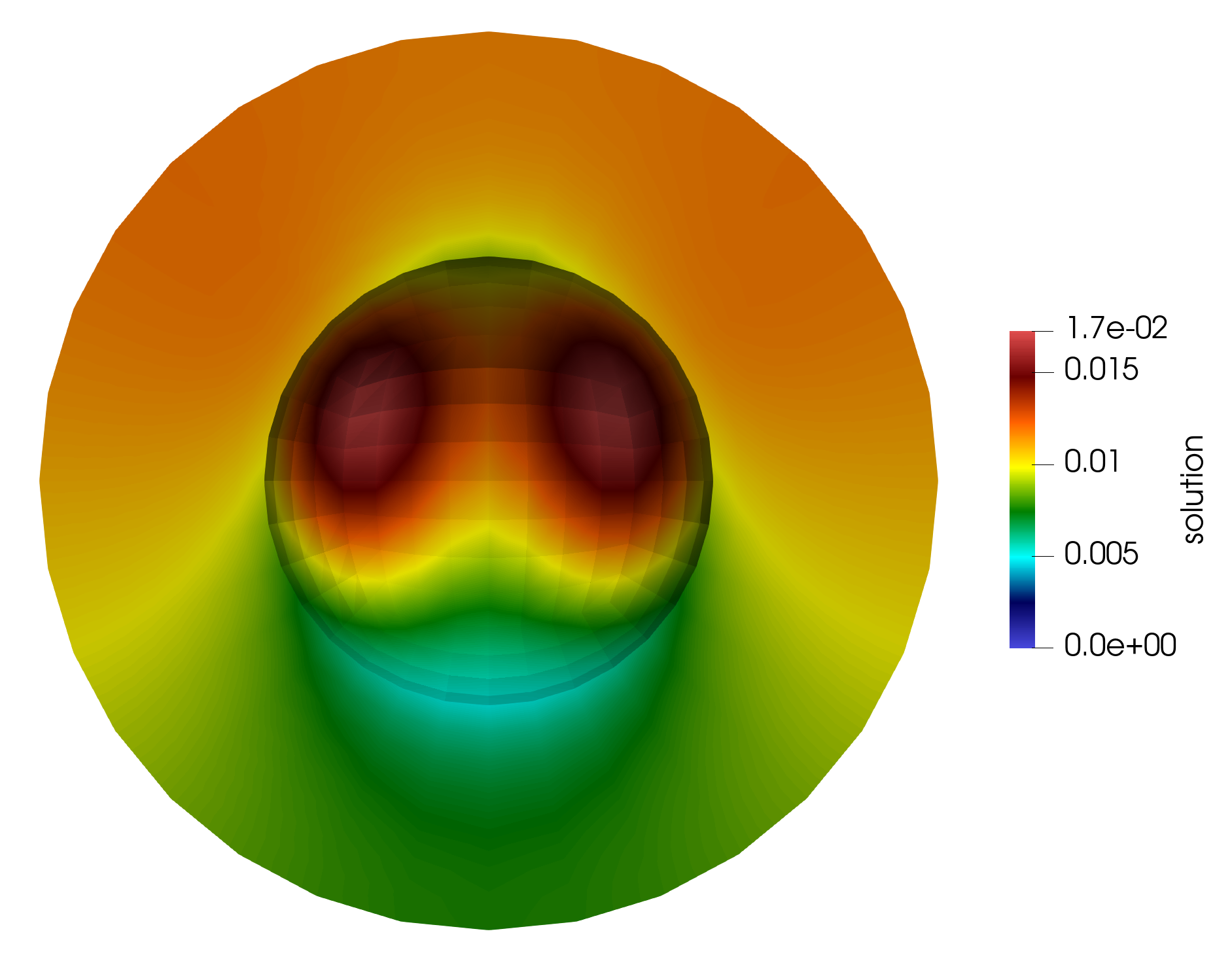}
	\end{subfigure}
	$\quad$
	\begin{subfigure}{.3\textwidth}
		\centering
		\includegraphics[width=1\textwidth]{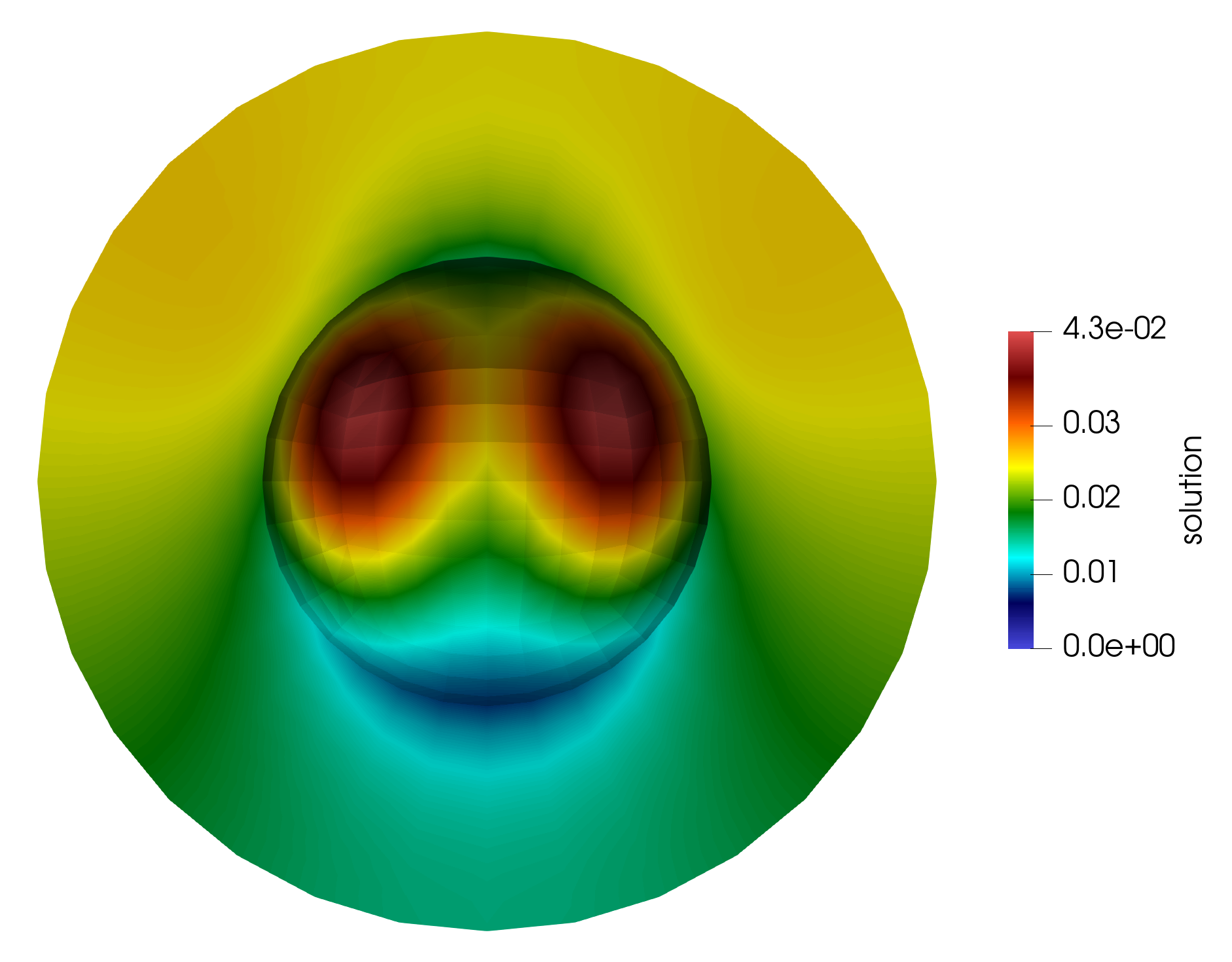}
	\end{subfigure}
	$\quad$
	\begin{subfigure}{.3\textwidth}
		\centering
		\includegraphics[width=1\textwidth]{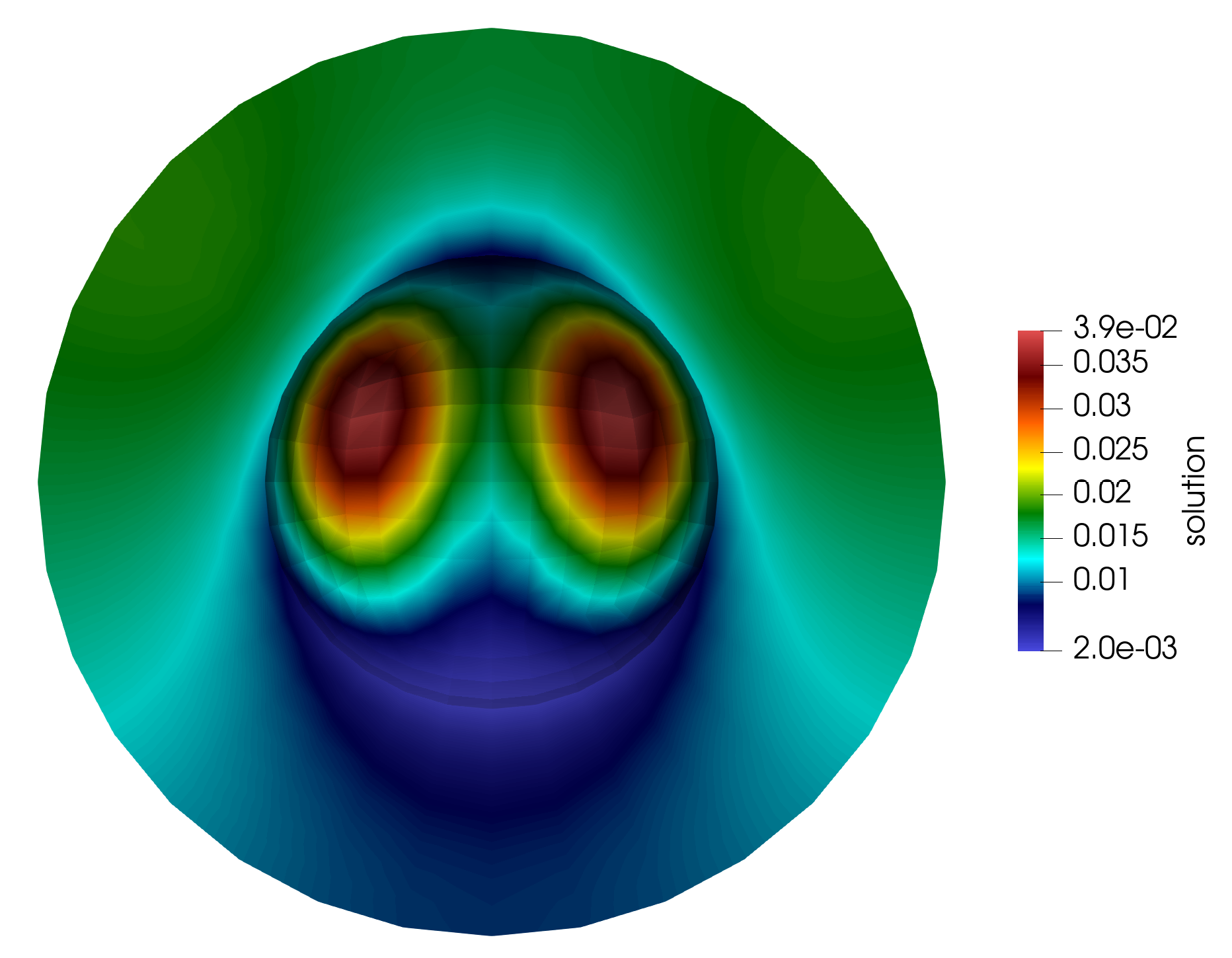}
	\end{subfigure}
	\caption{\emph{Test case i - different discretizations.} Master (top) and slave (bottom) solutions for three different instances of the parameters vector $\boldsymbol{\mu} = [\alpha, \beta]$.}
	\label{fig:snapshots_probl_1}
\end{figure} 

In the training phase, the high fidelity Dirichlet interface data are interpolated on the slave interface using the VTK life$^{\text{x}}$ function. Such function reads reference VTK data from a polygonal surface during construction, performs a linear interpolation of the point data array of the reference surface, and returns as output the results of such interpolation at specific input points. Therefore, given a parameters set, during the computation, \emph{(i)} the master solution is solved and the solution vector are save in Paraview readable files, \emph{(ii)} then, the reference surface is extracted manually from the master solution using the ParaView software in a post-processing procedure and save as a VTP file and, finally, \emph{(iii)} in the slave assembling of the system, the VTK function reads the data from such file and compute the interface vector $\mathbf{u}_{2,D}$ through the linear interpolation on the slave interface DoFs. Unfortunately, this method depends on the user expertise and is quite expensive, especially when the interface surface has a large dimension and the discretization used is very fine. For the test cases in this paper, with the considered domains and discretizations, we have measured an extraction costs of about 6 minutes for each simulation. We remark that different machine characteristics and less experience from the user might  rapidly increase the total computational cost of the interface treatments.

After the interpolation method has been applied, $\mathbf{u}_{2,D}$ can be stored in matrix format and used for the interface DEIM training, so that the POD-DEIM-POD ROM can be constructed according to Algorithm $\ref{alg:CompleteCoupledROM}$.

Then, we first evaluate the singular values decay of the master and slave solutions and interface data by varying the dimension $N_{train} = \{ 20, 40, 60, 80, 100\}$ of the training set. The decay of the singular values, reported in Fig. $\ref{Fig:singular_value_RD_mesh}$, show that the training set needed to get a sufficiently rich reduction is at least $N_{train} = 60$. Moreover, the eigenvalues decay of the slave solution and interface data are very similar, conveying the strong dependency of the slave solution from the Dirichlet data. We then select additional $N_{test} = 50$ values of the parameters vector to test our method.

\begin{figure}[h!]
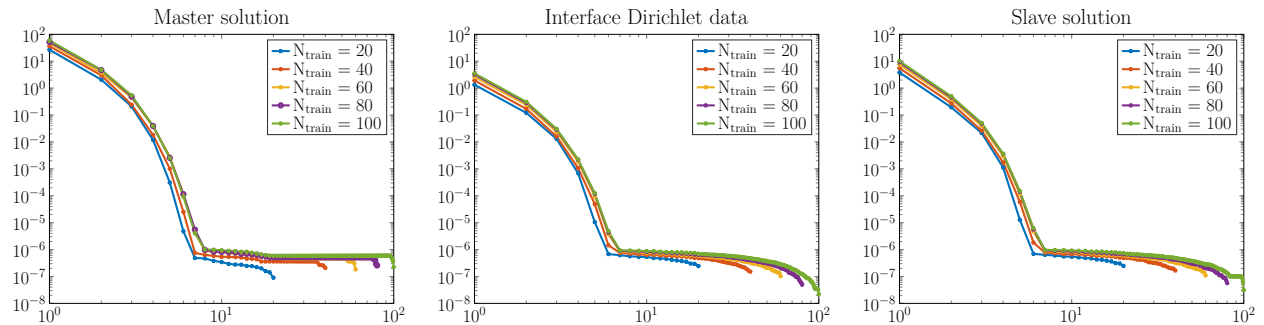

	\centering
	\begin{subfigure}{.32\textwidth}
		\centering
		\includegraphics[width=1\textwidth]{./Images/RD_eigenvalues_master.tex}
	\end{subfigure}\hfill
	\begin{subfigure}{.32\textwidth}
		\centering
		\includegraphics[width=1\textwidth]{./Images/RD_eigenvalues_DEIM.tex}
	\end{subfigure}\hfill
	\begin{subfigure}{.32\textwidth}
		\centering
		\includegraphics[width=1\textwidth]{./Images/RD_eigenvalues_slave.tex}
	\end{subfigure}
	\caption{\emph{Test case i - different discretizations.} Singular values decay of the master solution (left), interface Dirichlet data (center) and slave solution (right). }
	\label{Fig:singular_value_RD_mesh}
\end{figure}

The POD technique applied to reduce the master model is standard, thus we consider only the slave error as proof of the good ability of our model to reconstruct the correct solution (see Section $\ref{Sub_error_estimate}$). In particular, we define the absolute slave error as the mean of the 2-norm error $\eqref{eq.error_slave}$ over the $N_{test}$ solutions. 
Fig. $\ref{fig:RD_fixed_DEIM_error}$ and $\ref{fig:RD_fixed_master_error}$ show the errors computed fixing the prescribed POD tolerance used to reduced the master model and interface data, respectively. We recall that the slave solution is dependent on the interface Dirichlet data which in turn is influenced by the master ROM solution. Thus, the slave error depends both on the master POD and interface DEIM errors. In particular, as expected, a good approximation of the master solution but not of the interface data (and vice versa) yields a high error for the slave solution independently on the slave reduction operated, \emph{e.g.} considering an accuracy of magnitude $10^{-5}$ on the slave reduction but only an accuracy of magnitude $10^{-2}$ for the master or interface reduction returns an overall slave error of $10^{-2}$. Hence, a good approximation of all quantities is required to get a good estimate of the final solution, \emph{i.e} on average, the same order of accuracy for each step of the reduction must be imposed. 

\begin{figure}[h!]
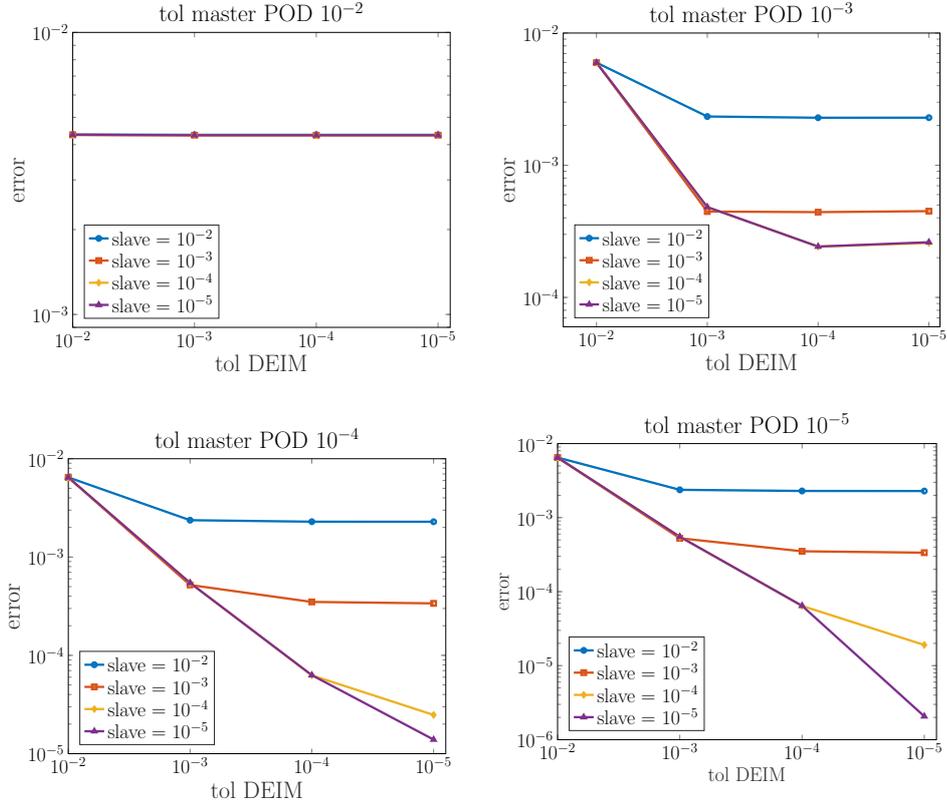

	\centering
	\begin{subfigure}{.3\textwidth}
		\centering 
		\includegraphics[width=1.2\textwidth]{./Images/RD_error_fix_master_1e-2.tex}
	\end{subfigure}\hspace{40pt}
	\begin{subfigure}{.3\textwidth}
		\centering 
		\includegraphics[width=1.2\textwidth]{./Images/RD_error_fix_master_1e-3.tex}
	\end{subfigure}\\
	\bigskip
	\begin{subfigure}{.3\textwidth}
		\centering 
		\includegraphics[width=1.2\textwidth]{./Images/RD_error_fix_master_1e-4.tex}
	\end{subfigure}\hspace{40pt}
	\begin{subfigure}{.3\textwidth}
		\centering 
		\includegraphics[width=1.2\textwidth]{./Images/RD_error_fix_master_1e-5.tex}
		\label{Fig:RD_fixed_master_5}
	\end{subfigure}
	\vspace{-0.25cm}
	\caption{\emph{Test case i - different discretizations.} $L^2(\Omega_2)$ mean slave solution error ($y$-axis) over $N_{test}$ trial fixing the master POD tolerance and varying the DEIM tolerance ($x$-axis) and the slave POD tolerance (legend).}
	\vspace{-0.25cm}
	\label{fig:RD_fixed_master_error}
\end{figure}
\begin{figure}[h!]
	\centering
	\begin{subfigure}{.3\textwidth}
		\includegraphics[width=1.2\textwidth]{./Images/RD_error_fix_DEIM_1e-2.tex}
	\end{subfigure}\hspace{40pt}
	\begin{subfigure}{.3\textwidth}
		\includegraphics[width=1.2\textwidth]{./Images/RD_error_fix_DEIM_1e-3.tex}
	\end{subfigure}\\
	\bigskip
	\begin{subfigure}{.3\textwidth}
		\includegraphics[width=1.2\textwidth]{./Images/RD_error_fix_DEIM_1e-4.tex}
	\end{subfigure}\hspace{40pt}
	\begin{subfigure}{.3\textwidth}
		\includegraphics[width=1.2\textwidth]{./Images/RD_error_fix_DEIM_1e-5.tex}
		\label{Fig:RD_fixed_DEIM_5}
	\end{subfigure}
	\vspace{-0.25cm}
	\caption{\emph{Test case i - different discretizations.} $L^2(\Omega_2)$ mean slave solution error ($y$-axis) over $N_{test}$ trial fixing the DEIM tolerance and varying the master POD tolerance ($x$-axis) and the slave POD tolerance (legend).}
	\label{fig:RD_fixed_DEIM_error}
\end{figure}

Regarding the efficiency, Fig. $\ref{Fig:RD_error_vs_time}$ reports the variations of the computational errors fixing the prescribed POD accuracy of the master and interface data reduction \emph{versus} CPU time. In particular, the graphs in figure $\ref{Fig:RD_error_vs_time}$ correspond to those of Fig. $\ref{fig:RD_fixed_master_error}$ and  $\ref{fig:RD_fixed_DEIM_error}$ when the reduction accuracy is $10^{-5}$, plotting the CPU time in the \emph{x} axis. We observe that increasing the POD accuracy in one of the three reduction steps does not dramatically increase the final computational cost, especially when the prescribed tolerance for the interface reduction is fixed, entailing that the major computational cost is given by the master ROM solution. For examples, fixing the master POD accuracy to $10^{-5}$, prescribing an accuracy of order of magnitude $10^{-5}$ for the slave and interface reduction will cause an increase of only $0.0044 s$ in the computational costs of the solution with respect to the same computation with prescribed tolerance of order $10^{-2}$ for slave and interface reduction, corresponding to an increasing of only the 0.23\% of the total computational costs, which is reduced to 0.04\% if in the less accurate simulation the prescribed accuracy are of order of $10^{-4}$ and $10^{-5}$ for the Dirichlet data and slave reduction, respectively. Therefore, a very accurate reduction can be obtained without loosing ROM efficiency. 

\begin{figure}[h!]
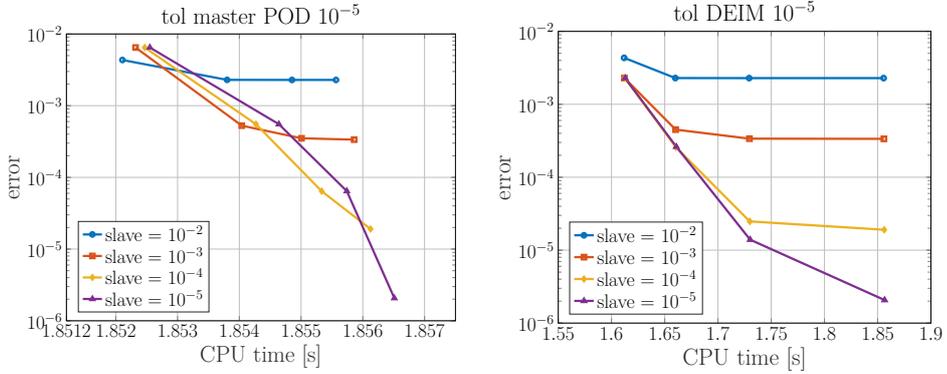

	\centering
	\begin{subfigure}{.3\textwidth}
		\includegraphics[width=1.2\textwidth]{./Images/RD_error_vs_time_fixed_master.tex}
	\end{subfigure}\hspace{40pt}
	\begin{subfigure}{.3\textwidth}
		\includegraphics[width=1.2\textwidth]{./Images/RD_error_vs_time_fixed_DEIM.tex}
	\end{subfigure}
	\vspace{-0.25cm}
	\caption{\emph{Test case i - different discretizations.} $L^2(\Omega_2)$ mean slave solution error vs the CPU time fixing the master POD tolerance (left) and the interface DEIM tolerance (right) to $10^{-5}$ and varying the tolerances used for the reduction of the other quantities.}
	\vspace{-0.25cm}
	\label{Fig:RD_error_vs_time}
\end{figure}

Finally, in Table $\ref{Tab:RD_mesh}$ we report the dimensions and performances of the FOM and ROM offline and online stages for one instance of the parameters vector. We choose to fix the master, slave POD and interface data reduction tolerances to $10^{-5}$. Up to a very expensive offline phase, according to Fig. $\ref{Fig:RD_error_vs_time}$, a satisfying speed up of about 200 times is obtained from the slave model given that, during the online reduced computation, the data reading and interpolation procedure is avoided, saving up about 98\% of the computational costs of the interface extraction. However, the greatest computational cost reduction is gained by the complete coupled problem due to the absence of the manual interface extraction method, which means a saving up of 100\% on the interface extraction.

\begin{table}[h!]
	\centering 
	\begin{tabular}{c| cc| cccc}
		\toprule
		& \multicolumn{2}{c| }{High fidelity model} &\multicolumn{4}{c}{Reduced order model} \\ 
		&\#FE &FE solution &\#RB &Offline&Online &Speed up \\
		& DoFs & time &&time &time \\
		\hline &&&&&&\\ [-2ex]
		Master model &202$k$ &$\sim$ 19.76$s$ &8 &$\sim$ 1901$s$ &$\sim$ 1.82$s$ &\cellcolor{red!25}10.9$x$\\
		Slave model &26$k$ &$\sim$ 1.85$s$ &6 &$\sim$ 103$s$ &$\sim$ 0.04$s$ &\cellcolor{red!15}46.3$x$\\
		Interface data & &\cellcolor{red!25}$\sim$ 6$m$ &7 &$\sim$ 484$s$  & \cellcolor{green!25}0.00$s$&\\ 
		Coupled model & &$\sim$ 381.61$s$ & &$\sim$ 2488$s$ &$\sim$ 1.86$s$ &\cellcolor{green!25}205.2$x$\\
		\bottomrule
	\end{tabular}
	\caption{\emph{Test case i - different discretizations.} High fidelity and reduced order model dimensions and CPU times.  We highlight the performances of the ROM model with respect to the interface Dirichlet data treatment and the speed up using colors from red (worst) to green (best).}
	\label{Tab:RD_mesh}
\end{table}	

\subsection{Different FE order}
We repeat the same experiment considering equal discretization for the two domains and different FE orders. Specifically, we choose $q_1 = 2$ and $q_2 = 1$, and $h_1 = h_2 = 0.421191$ so that $N_1 = 202818$ and $N_2 = 26146$.

As expected, the eigenvalues decays outcome is that of the previous test case (we refer to Fig. $\ref{Fig:singular_value_RD_mesh}$), and we select again $N_{test} = 50$ values of parameters vector to test the coupled ROM. In Fig. $\ref{fig:RD_fixed_master_error_order}$ and $\ref{fig:RD_fixed_DEIM_error_order}$ we report the slave error $\eqref{eq.error_slave}$ over the $N_{test}$ trial fixing the master POD and interface data prescribed tolerances, respectively. Once more, a good approximation of the master solution and the interface Dirichlet data provides a good approximation of the slave solution. Moreover, the influence of the master solution seems to be higher than before, since the decrease of the errors is faster fixing the interface DEIM tolerance than the corresponding master one.

\begin{figure}[h!]
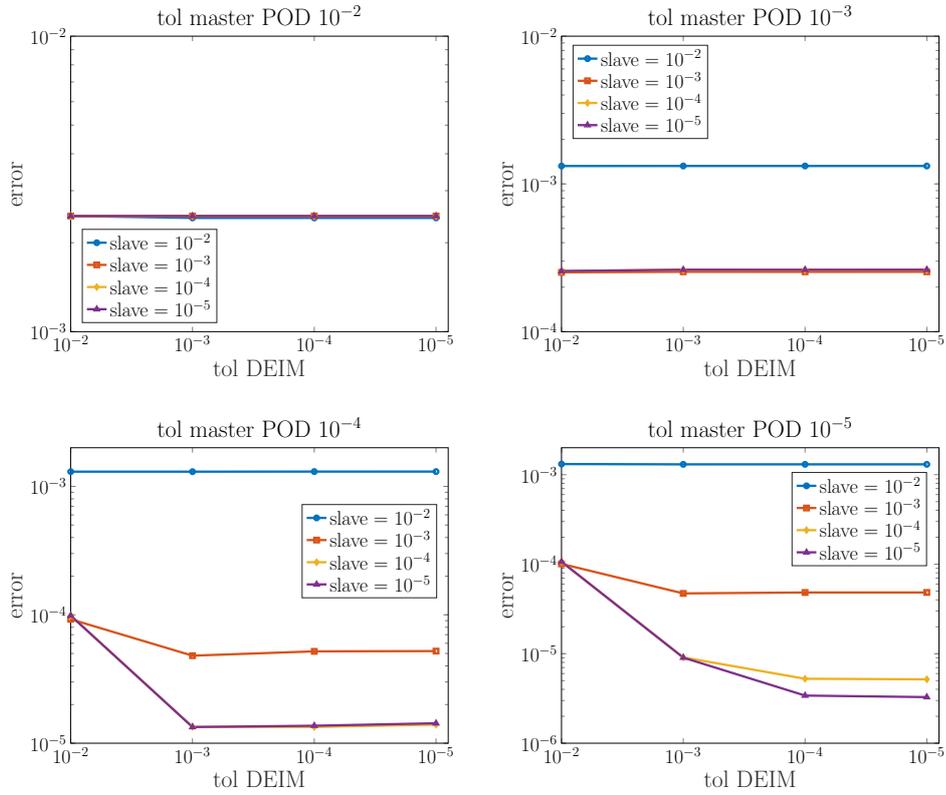

	\centering
	\begin{subfigure}{.3\textwidth}
		\centering 
		\includegraphics[width=1.2\textwidth]{./Images/RD_error_fix_master_1e-2_order.tex}
	\end{subfigure}\hspace{40pt}
	\begin{subfigure}{.3\textwidth}
		\centering 
		\includegraphics[width=1.2\textwidth]{./Images/RD_error_fix_master_1e-3_order.tex}
	\end{subfigure}\\
	\bigskip
	\begin{subfigure}{.3\textwidth}
		\centering 
		\includegraphics[width=1.2\textwidth]{./Images/RD_error_fix_master_1e-4_order.tex}
	\end{subfigure}\hspace{40pt}
	\begin{subfigure}{.3\textwidth}
		\centering 
		\includegraphics[width=1.2\textwidth]{./Images/RD_error_fix_master_1e-5_order.tex}
	\end{subfigure}
	\caption{\emph{Test case i - different FE order.} $L^2(\Omega_2)$ mean slave solution error ($y$-axis) over $N_{test}$ trial fixing the master POD tolerance and varying the DEIM tolerance ($x$-axis) and the slave POD tolerance (legend).}
	\label{fig:RD_fixed_master_error_order}
\end{figure}
\begin{figure}[h!]
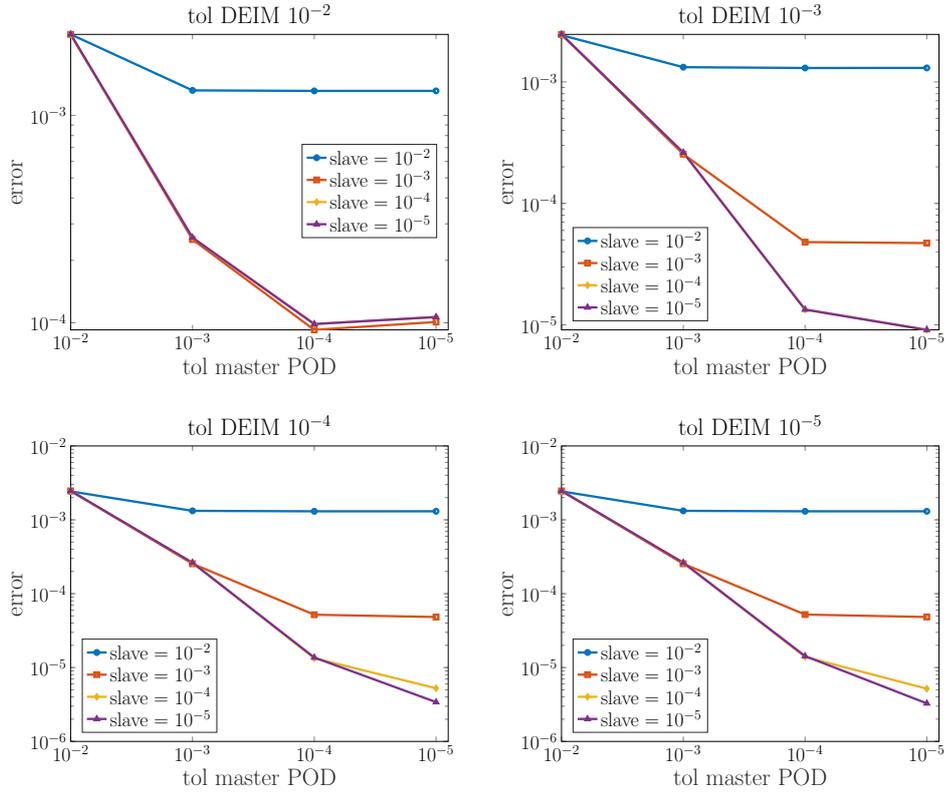

	\centering
	\begin{subfigure}{.3\textwidth}
		\includegraphics[width=1.2\textwidth]{./Images/RD_error_fix_DEIM_1e-2_order.tex}
	\end{subfigure}\hspace{40pt}
	\begin{subfigure}{.3\textwidth}
		\includegraphics[width=1.2\textwidth]{./Images/RD_error_fix_DEIM_1e-3_order.tex}
	\end{subfigure}\\
	\bigskip
	\begin{subfigure}{.3\textwidth}
		\includegraphics[width=1.2\textwidth]{./Images/RD_error_fix_DEIM_1e-4_order.tex}
	\end{subfigure}\hspace{40pt}
	\begin{subfigure}{.3\textwidth}
		\includegraphics[width=1.2\textwidth]{./Images/RD_error_fix_DEIM_1e-5_order.tex}
	\end{subfigure}
	\caption{\emph{Test case i - different FE order.} $L^2(\Omega_2)$ mean slave solution error ($y$-axis) over $N_{test}$ trial fixing the DEIM tolerance and varying the master POD tolerance ($x$-axis) and the slave POD tolerance (legend).}
	\label{fig:RD_fixed_DEIM_error_order}
\end{figure}

Fig. $\ref{Fig:RD_error_vs_time_order}$ and Table $\ref{Tab:RD_mesh_order}$ show  similar results of the computational expensiveness of the operated reductions. We point out that the decrease of the master ROM speed up of about 6 times with respect to the previous test case affects more the coupled problem CPU time reduction; anyway, we can ensure a good overall performance given the absence of manual interface extraction.

\begin{figure}[h!]
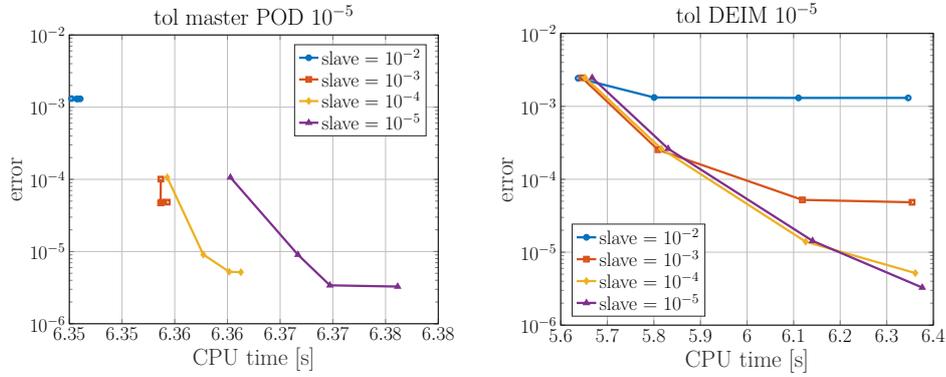

	\centering
	\begin{subfigure}{.3\textwidth}
		\includegraphics[width=1.2\textwidth]{./Images/RD_error_vs_time_fixed_master_order.tex}
	\end{subfigure}\hspace{40pt}
	\begin{subfigure}{.3\textwidth}
		\includegraphics[width=1.2\textwidth]{./Images/RD_error_vs_time_fixed_DEIM_order.tex}
	\end{subfigure}
	\caption{\emph{Test case i - different FE order.} $L^2(\Omega_2)$ mean slave solution error vs the CPU time fixing the master POD tolerance (left) and the interface DEIM tolerance (right) to $10^{-5}$ and varying the tolerances used for the reduction of the other quantities.}
	\label{Fig:RD_error_vs_time_order}
\end{figure}

\begin{table}[h!]
	\centering 
	\begin{tabular}{c| cc| cccc}
		\toprule
		& \multicolumn{2}{c| }{High fidelity model} &\multicolumn{4}{c}{Reduced order model} \\ 
		&\#FE &FE solution  &\#RB &Offline&Online&Speed up
		\\ &DoFs &time &&time &time \\
		\hline &&&&&&\\ [-2ex]
		Master model &202$k$ &$\sim$ 19.49$s$ &8 &$\sim$ 1950$s$ &$\sim$ 6.30$s$ &\cellcolor{red!25}3.1$x$\\
		Slave model &26$k$ &$\sim$ 2.54$s$ &6 &$\sim$ 566$s$ &$\sim$ 0.08$s$ &\cellcolor{red!15}33.4$x$\\
		Interface data & &\cellcolor{red!25}$\sim$ 6$m$ &7 &$\sim$ 646$s$  & \cellcolor{green!25}0.00$s$&\\ 
		Coupled model & &$\sim$ 382.03$s$ & &$\sim$ 3162$s$ &$\sim$ 6.38$s$ &\cellcolor{green!25}$ 59.9x$\\
		\bottomrule
	\end{tabular}
	\vspace{0.15cm}
	\caption{\emph{Test case i - different FE order.} High fidelity and reduced order model dimensions and CPU times.  We highlight the performances of the ROM model with respect to the interface Dirichlet data treatment and the speed up using colors from red (worst) to green (best).}
	\label{Tab:RD_mesh_order}
\end{table}	

\subsection{Test case $ii$: unsteady model - steady model}
\label{Subsec:test_case_ii}
We  now apply the proposed ROM to a time dependent model coupled with a time independent model. In particular we choose the heat equation as master model and a simple Laplacian as slave model. Hence, 
\begin{equation}
\begin{cases}
\frac{\partial u}{\partial t} - \nabla \cdot (\alpha \nabla u) = f &\text{in } \Omega_1\times \{0, T\}\\
u = 0 &\text{on }\partial \Omega_{1,D}\backslash \Gamma \times \{0, T\}\\
\frac{\partial u}{\partial n_1} = 0 &\text{on }\Omega_{1,N}\times \{0, T\}\\
\frac{\partial u}{\partial n_1} = 0 &\text{on }\Gamma \times\{0,T\}\\
u(0) = 0 &\text{on }\Omega_{1,N},
\end{cases}
\end{equation} 
and 
\begin{equation}
\begin{cases}
- \Delta v = 0 &\text{in }\Omega_2\\
\frac{\partial v}{\partial n_2} = 0 &\text{on }\partial \Omega_{2,N},
\end{cases}
\end{equation}
with the usual coupling conditions at the interface
$$ u = v \quad \text{on } \Gamma.$$
We define $f(x,y,z,t) = 1-\sin(\pi y)cos(\frac{\pi}{2}x)$, the time interval $[0,1]$ and we choose to vary $\alpha$ in $[0,5]$ according to LHS distribution. We remark that the time variable is considered as the second parameter of the reduce model.

The FOMs are solved in two three-dimensional cubes with a common face $\Gamma$. We choose $\partial\Omega_{1,D}\backslash \Gamma$ as the face of $\Omega_1$ opposite to $\Gamma$, $\partial\Omega_{1,N}$ as the faces of $\Omega_1$ perpendicular to $\Gamma$, and $\partial\Omega_{2,N} = \partial\Omega_2 \backslash \Gamma$.
We used different discretizations on $\Omega_1$ and $\Omega_2$ and the same FEM-$\mathbb{Q}_1$. In particular, we fixed $h_1 = 0.0541266~m$ and $h_2 = 0.108253~m$, meaning that $N_1 = 35937$ and $N_2 = 4913$. See Fig. $\ref{Fig:HD_domains.}$ for a graphical representation of the domains and Fig. $\ref{fig:snapshots_probl_2}$ for some FOM numerical solutions at the interface.

\begin{figure}[h!]
	\centering
	\begin{subfigure}{.48\textwidth}
		\includegraphics[width=1\textwidth]{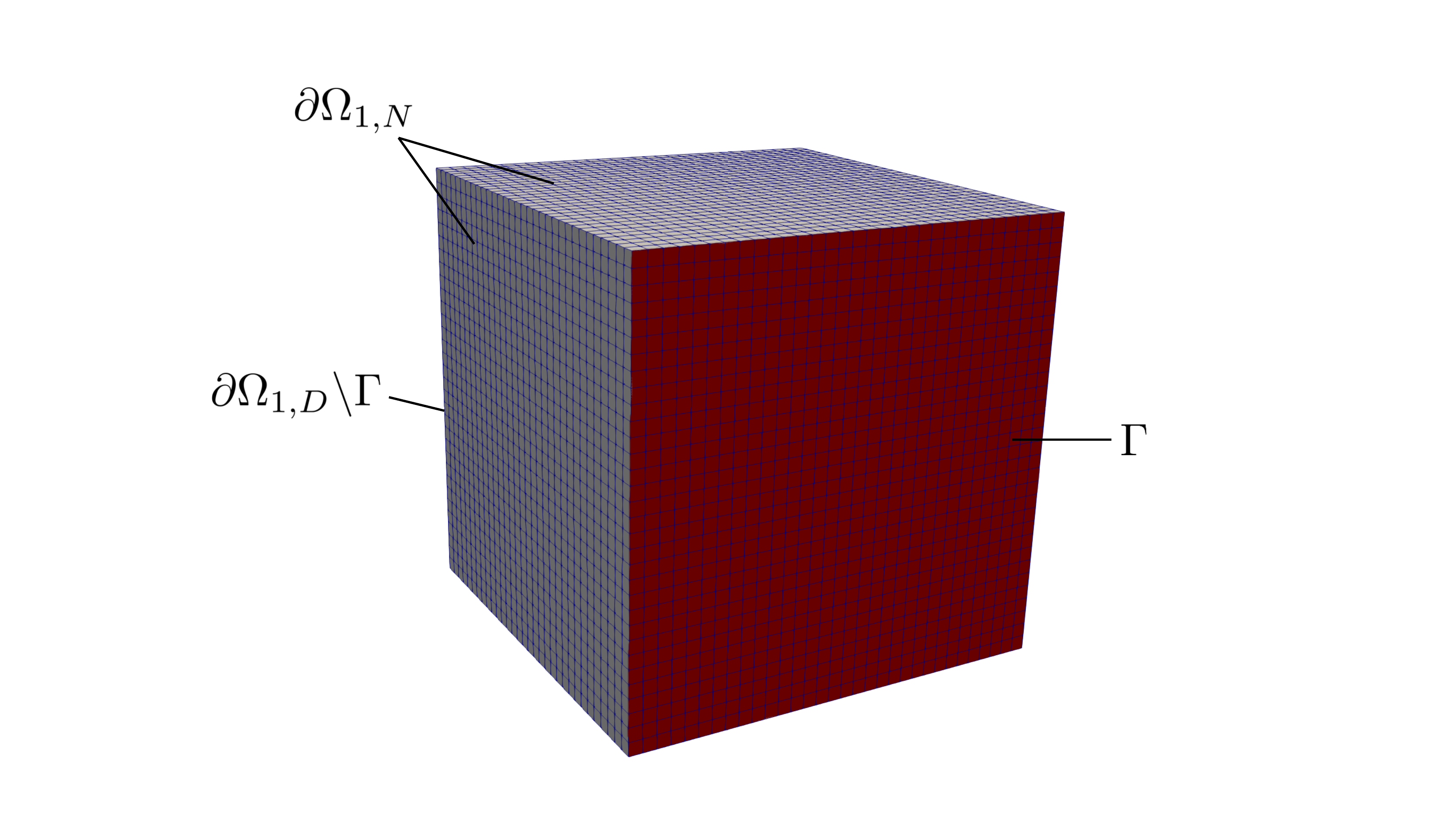}
	\end{subfigure}
	\begin{subfigure}{.48\textwidth}
		\includegraphics[width=1\textwidth]{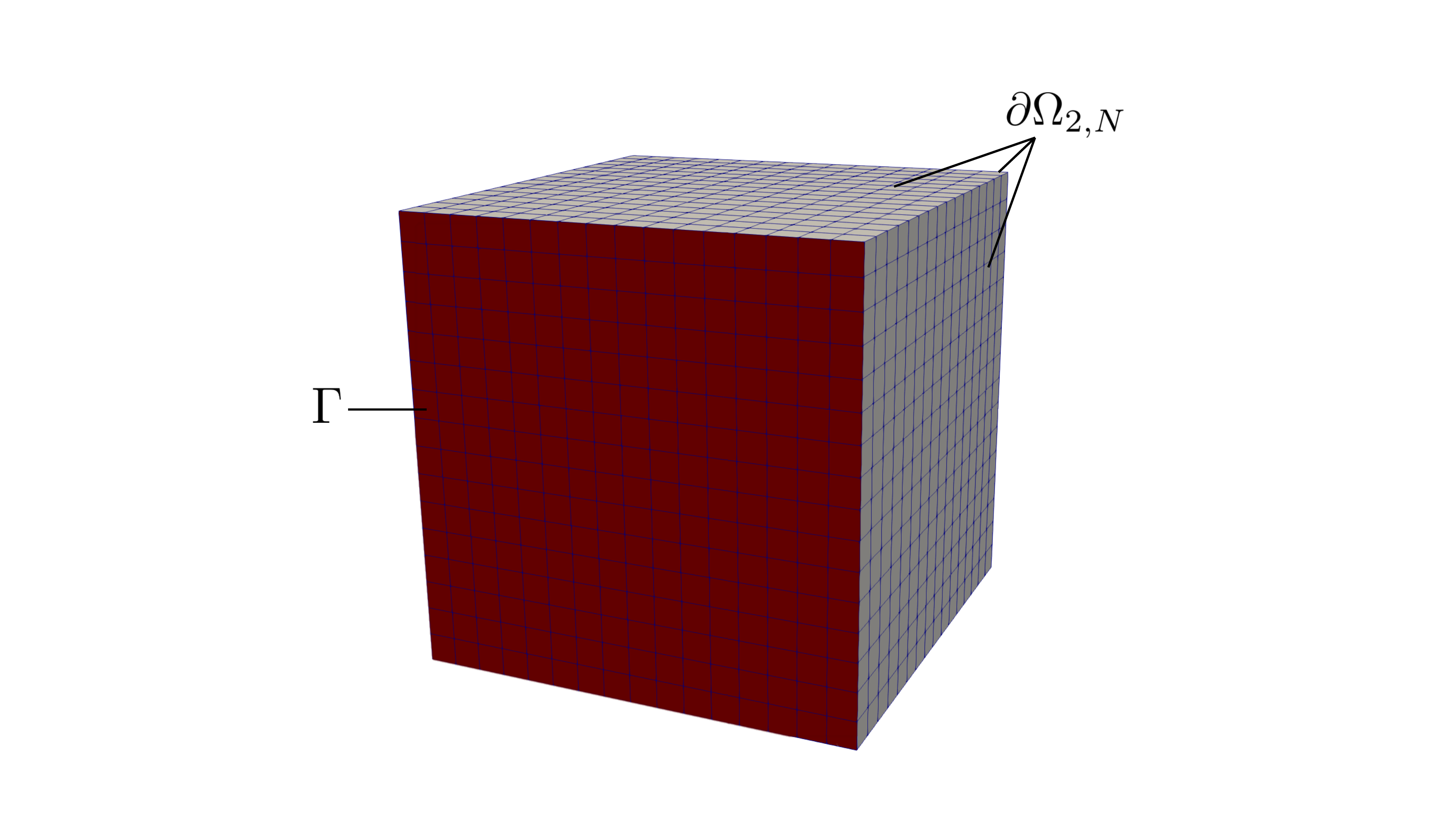}
	\end{subfigure}
	\caption{\emph{Test case ii.} Master (left) and slave (right) domains, two cubes discretized with different meshes, \emph{i.e.} $h_1 = 0.0541266~m$ and $h_2 = 0.108253~m$. In red the interface boundary $\Gamma$.}
	\label{Fig:HD_domains.}
\end{figure}

\begin{figure}[h!]
	\centering
	\begin{subfigure}{.2\textwidth}
		\centering
		\hspace{100pt} $t = 10s$\\
		\vspace{7pt}
		\includegraphics[width=1.2\textwidth]{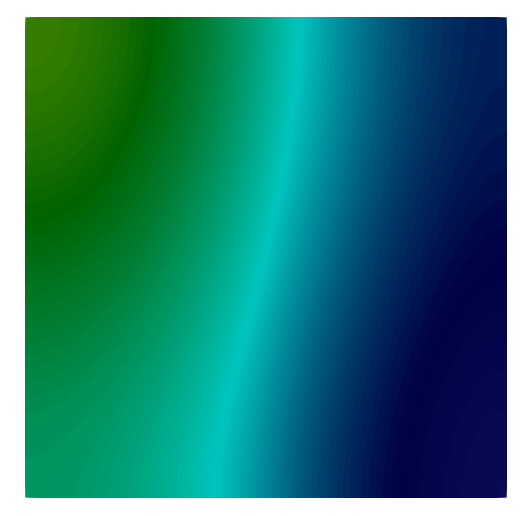}
	\end{subfigure}
	\hspace{15pt}
	\begin{subfigure}{.2\textwidth}
		\centering
		\hspace{70pt} $t = 30s$\\
		\vspace{7pt}
		\includegraphics[width=1.2\textwidth]{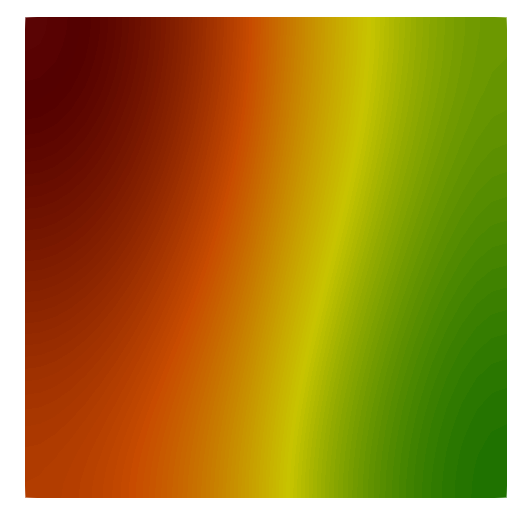}
	\end{subfigure}
	\hspace{15pt}
	\begin{subfigure}{.2\textwidth}
		\centering
		\hspace{70pt} $t = 60s$\\
		\vspace{7pt}
		\includegraphics[width=1.2\textwidth]{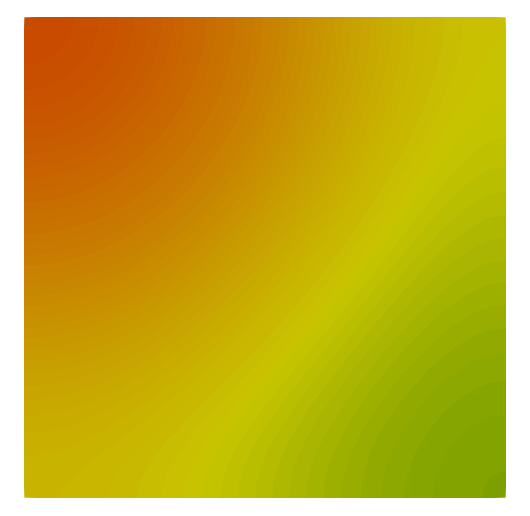}
	\end{subfigure} 
	\hspace{7pt}
	\begin{subfigure}{.2\textwidth}
		\centering
		\hspace{30pt}$ $ \\
		\vspace{18pt}
		\includegraphics[width=0.7\textwidth]{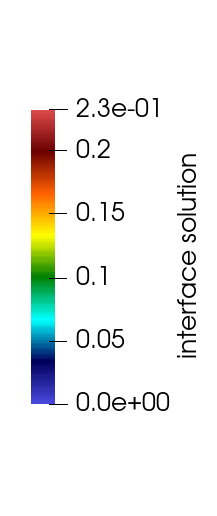}
	\end{subfigure}\\
	\begin{subfigure}{.2\textwidth}
		\centering
		\includegraphics[width=1.2\textwidth]{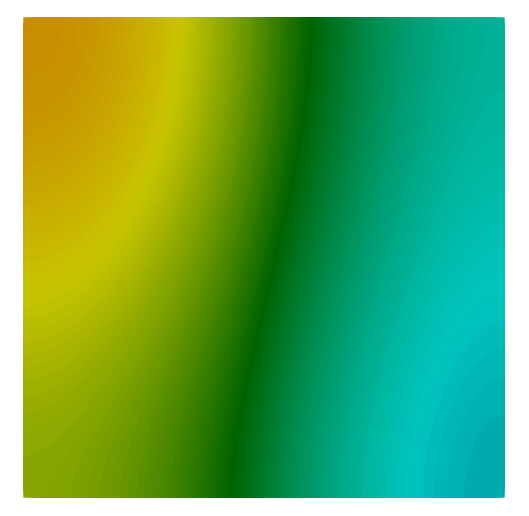}
	\end{subfigure}
	\hspace{15pt}
	\begin{subfigure}{.2\textwidth}
		\centering
		\includegraphics[width=1.2\textwidth]{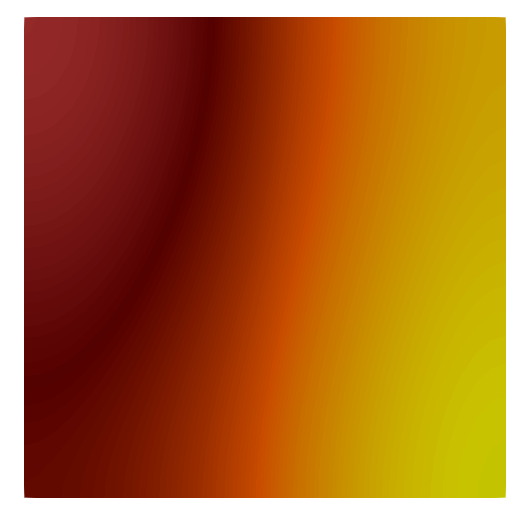}
	\end{subfigure}
	\hspace{15pt}
	\begin{subfigure}{.2\textwidth}
		\centering
		\includegraphics[width=1.2\textwidth]{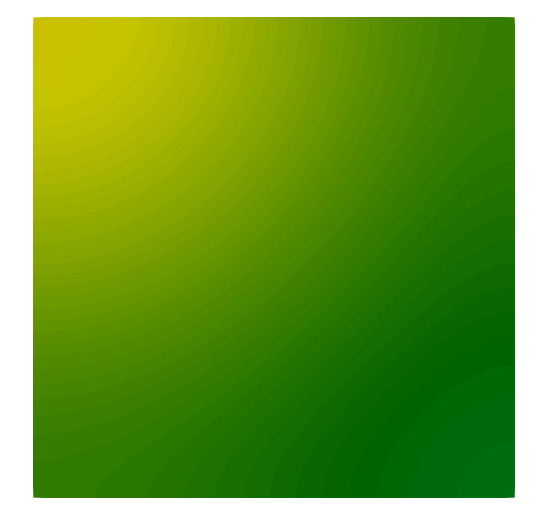}
	\end{subfigure} 
	\hspace{7pt}
	\begin{subfigure}{.2\textwidth}
		\centering
		\includegraphics[width=0.7\textwidth]{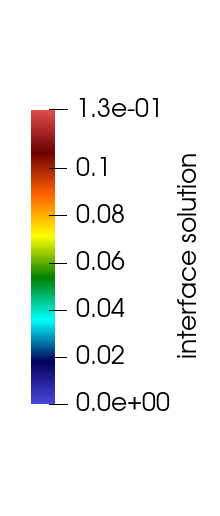}
	\end{subfigure}
	\caption{\emph{Test case ii.} Interface solutions for three different time instant (columns) and two different instances of $\alpha$ (rows).}
	\label{fig:snapshots_probl_2}
\end{figure}

We solve the master model using a BDF scheme of order 1 with $\Delta t = 10^{-2}$. Then, we evaluate the singular values decay of the master and slave solutions and interface data varying $N_{train} = \{10, 20, 40, 60,  80\}$. The corresponding snapshots matrices are formed by $N_s = N_t N_{train}$ full-order vectors, in which $N_t = 100$ is the number of time-steps used to solve the heat equation. The eigenvalues decay reported in Fig. $\ref{Fig:singular_value_HD_mesh}$ show that $N_{train} = 40$ is enough to get a sufficiently rich reduction. As before, the eigenvalue decays of slave solution and interface data are quite similar.
\begin{figure}[h!]
	\centering
	\begin{subfigure}{.32\textwidth}
		\centering
		\includegraphics[width=1.\textwidth]{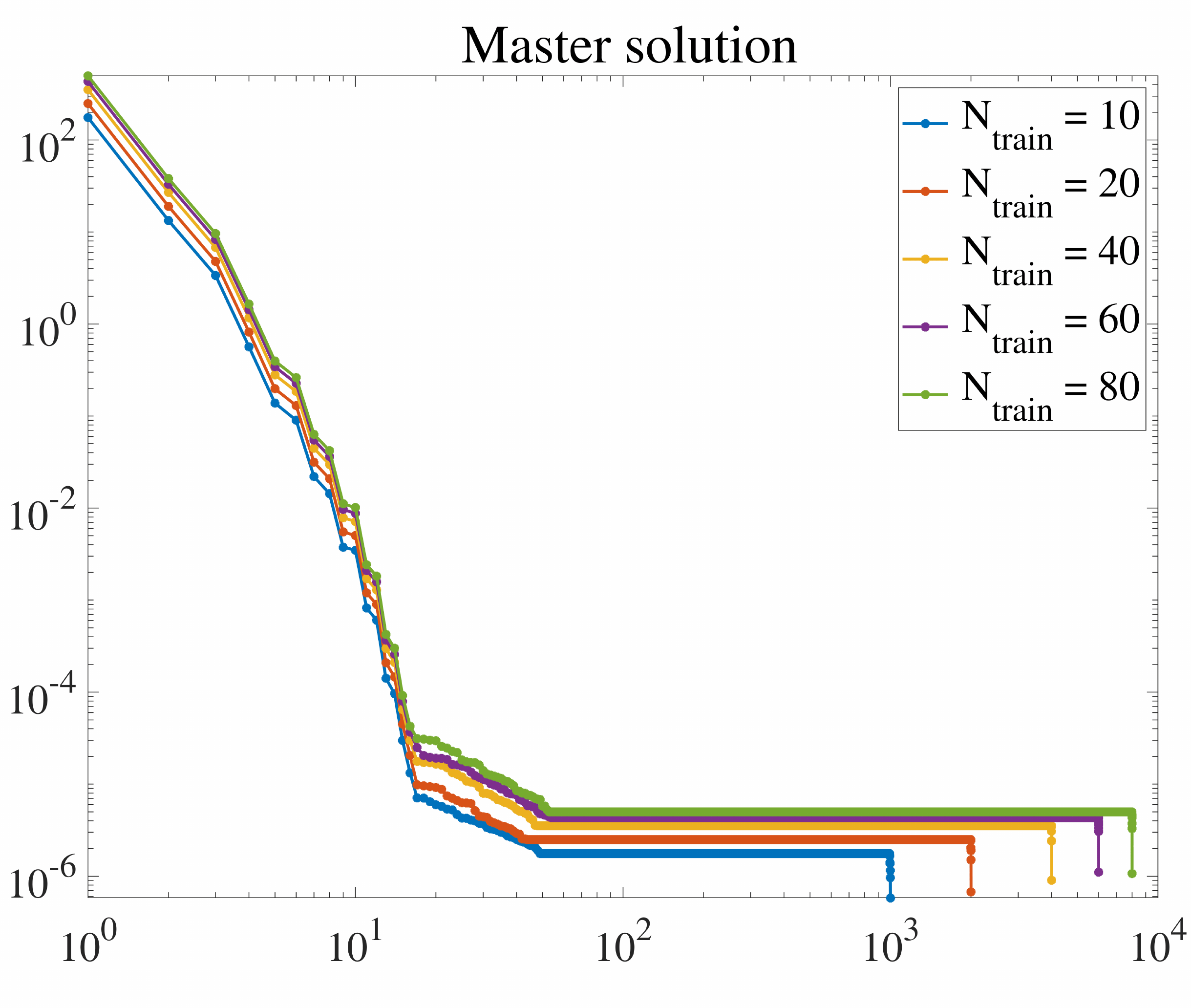}
	\end{subfigure}\hfill
	\begin{subfigure}{.32\textwidth}
		\centering
		\includegraphics[width=1\textwidth]{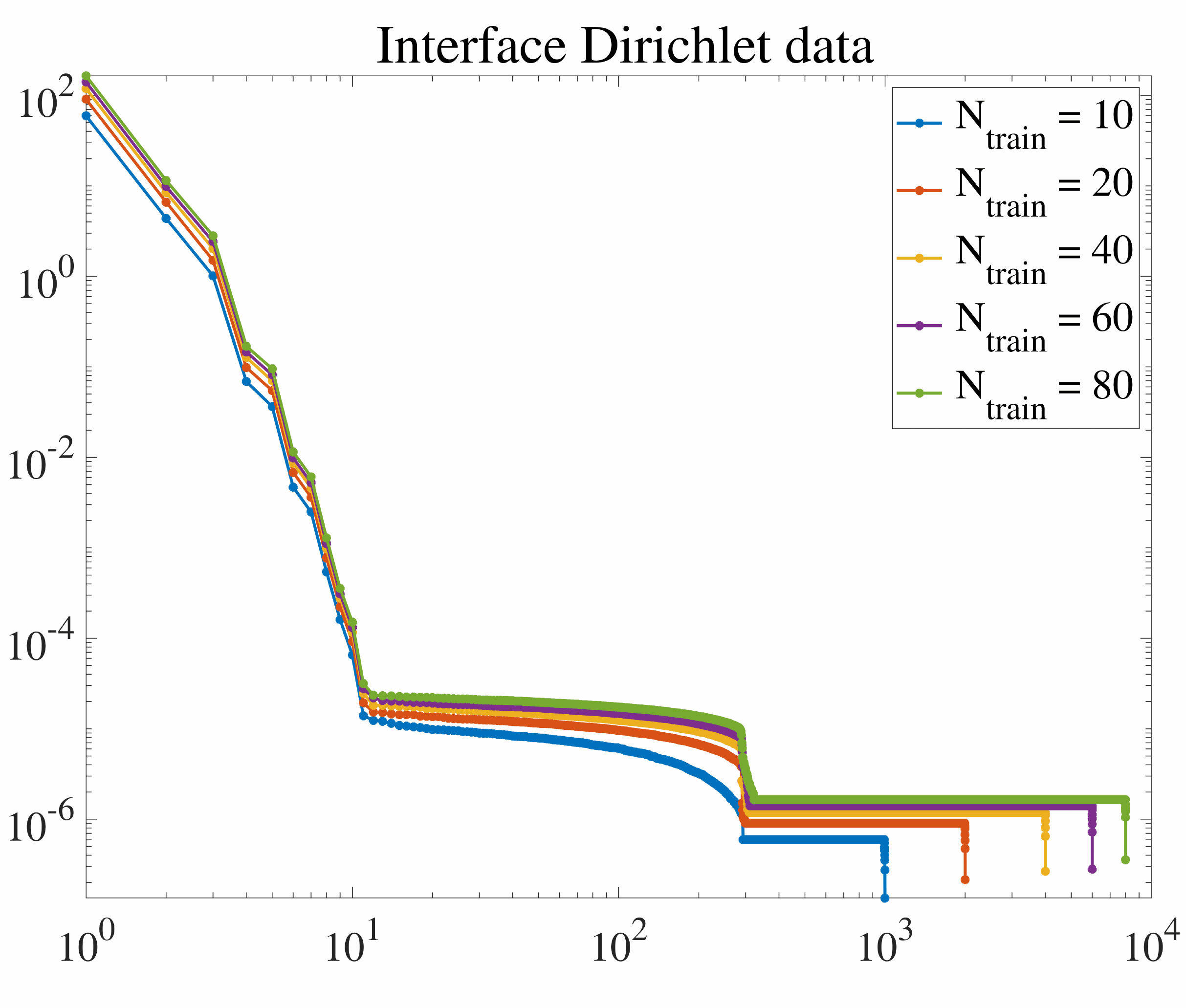}
	\end{subfigure}\hfill
	\begin{subfigure}{.33\textwidth}
		\centering
		\includegraphics[width=1\textwidth]{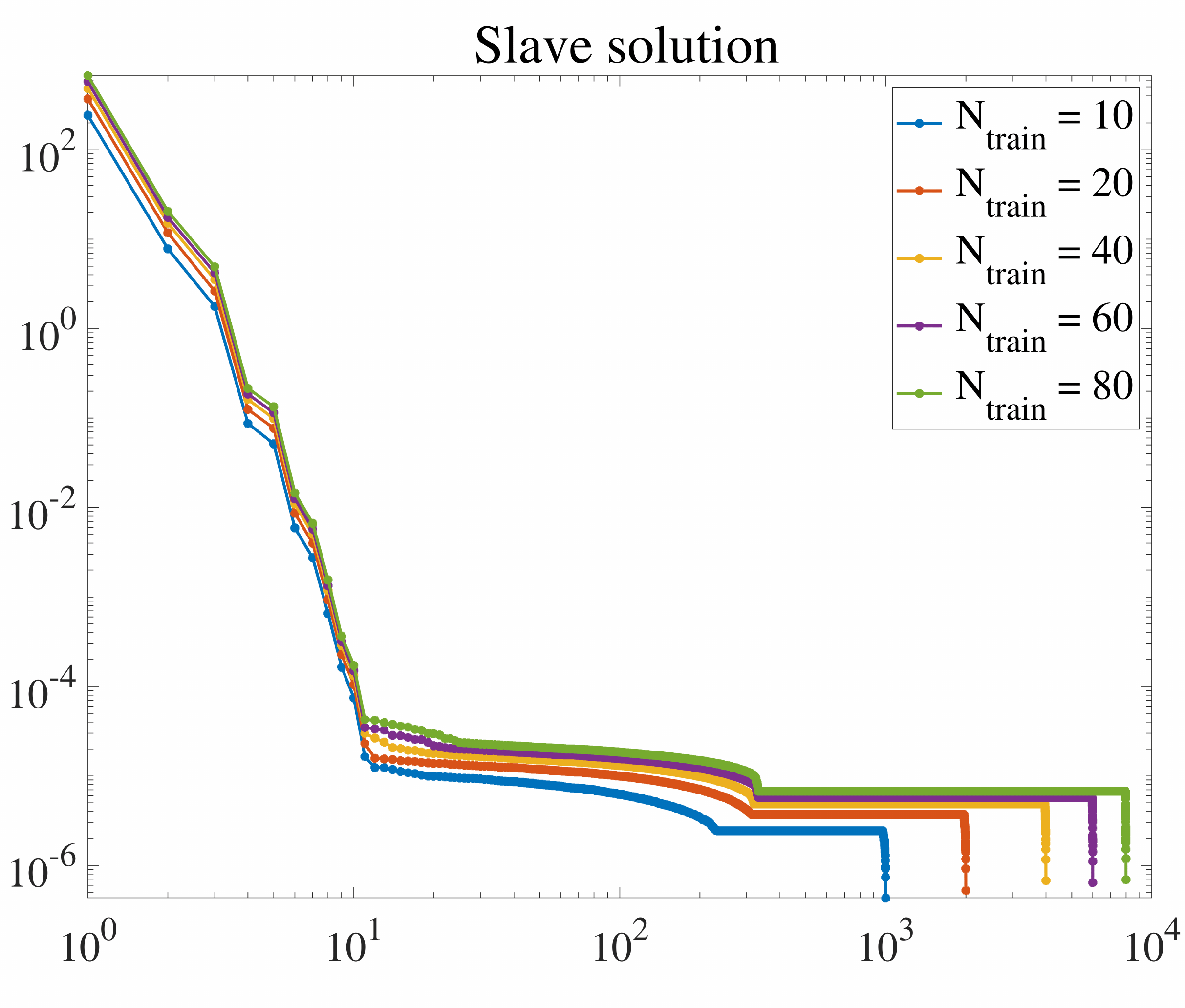}
	\end{subfigure}
	\caption{\emph{Test case ii.} Singular values decay of the master solution (left), interface Dirichlet data (center) and slave solution (right). }
	\label{Fig:singular_value_HD_mesh}
\end{figure}

We select $N_{test} = 5$ values of $\alpha$ to test our procedure and we estimate the reduced error on the slave domain as in test case \emph{i}, considering the mean of the 2-norm error over the $N_tN_{test}$ trial of the ROM and FOM slave solutions. We report the errors in Fig. $\ref{Fig:singular_value_HD_mesh}$ and $\ref{fig:HD_fixed_DEIM_error_order}$. We found that, once again, the reduced slave solution is more dependent from the reduction of the master solution than from that of the interface data. In any case, a good approximation of both quantities is required to obtain a good approximation of the slave solution, as for the time-independent test case \emph{i}.

\begin{figure}[h!]
	\centering
	\begin{subfigure}{.3\textwidth}
		\centering 
		\includegraphics[width=1.2\textwidth]{./Images/HD_error_fix_master_1e-2.tex}
	\end{subfigure}\hspace{40pt}
	\begin{subfigure}{.3\textwidth}
		\centering 
		\includegraphics[width=1.2\textwidth]{./Images/HD_error_fix_master_1e-3.tex}
	\end{subfigure}\\
	\bigskip
	\begin{subfigure}{.3\textwidth}
		\centering 
		\includegraphics[width=1.2\textwidth]{./Images/HD_error_fix_master_1e-4.tex}
	\end{subfigure}\hspace{40pt}
	\begin{subfigure}{.3\textwidth}
		\centering 
		\includegraphics[width=1.2\textwidth]{./Images/HD_error_fix_master_1e-5.tex}
		\label{Fig:HD_fixed_master_5}
	\end{subfigure}
	\caption{\emph{Test case ii.} $L^2(\Omega_2)$ mean slave solution error ($y$-axis) over $N_{test}$ trial fixing the master POD tolerance and varying the DEIM tolerance ($x$-axis) and the slave POD tolerance (legend).}
	\label{fig:HD_fixed_master_error}
\end{figure}
\begin{figure}[h!]
	\centering
	\begin{subfigure}{.3\textwidth}
		\centering
		\includegraphics[width=1.2\textwidth]{./Images/HD_error_fix_DEIM_1e-2.tex}
	\end{subfigure}\hspace{40pt}
	\begin{subfigure}{.3\textwidth}
		\centering
		\includegraphics[width=1.2\textwidth]{./Images/HD_error_fix_DEIM_1e-3.tex}
	\end{subfigure}\\
	\bigskip
	\begin{subfigure}{.3\textwidth}
		\centering
		\includegraphics[width=1.2\textwidth]{./Images/HD_error_fix_DEIM_1e-4.tex}
	\end{subfigure}\hspace{40pt}
	\begin{subfigure}{.3\textwidth}
		\centering
		\includegraphics[width=1.2\textwidth]{./Images/HD_error_fix_DEIM_1e-5.tex}
	\end{subfigure}
	\caption{\emph{Test case ii.} $L^2(\Omega_2)$ mean slave solution error ($y$-axis) over $N_{test}$ trial fixing the DEIM tolerance and varying the master POD tolerance ($x$-axis) and the slave POD tolerance (legend).}
	\label{fig:HD_fixed_DEIM_error_order}
\end{figure}

Fig. $\ref{Fig:HD_error_vs_time}$ and Table $\ref{Tab:HD_mesh_order}$ outline the performances of the ROM model related to a fixed prescribed POD tolerance of $10^{-5}$. The reported time values refer to a complete simulation in time with 100 time steps. Compared to test case \emph{i}, the overall performances of the ROM worsens since some expensive tasks are repeated in the ROM at each time step. In particular, according to the applied BDF formula, the right hand side of the master model depends of the FOM solution that, hence, must be reconstructed at each time step. This task can be avoided considering, for example, a hyper reduction technique for the right hand side. In any case, an overall speed up of about 2 and 3 times can be obtained for each submodel and a total speed up of 7 times  can be achieved for the coupled problem, since the manual interface extraction and interpolation are not required in the online phase -- two tasks that would require almost the 75\% of the FOM CPU time. As before, we point out that considering an accuracy in the reduction of $10^{-5}$ for the three parts of the model will increase of only about  3\% the total computational costs of the simulation with respect to a reduction with POD tolerance fixed to $10^{-5}$ for the master model and to $10^{-2}$ for the slave and the interface subproblems.
\begin{figure}[h!]
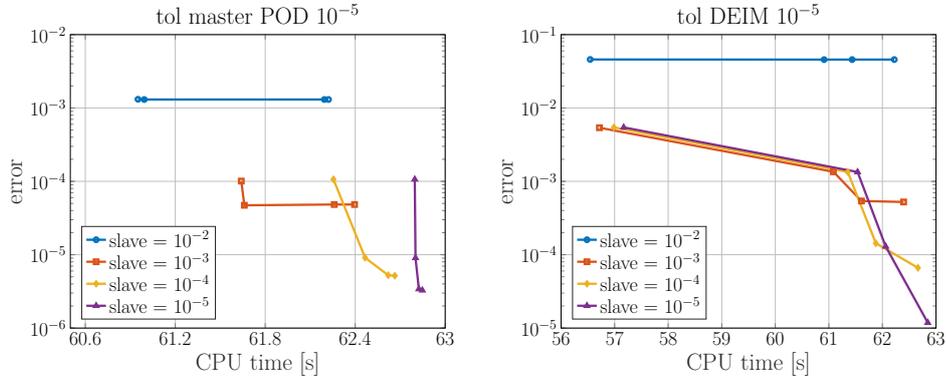

	\centering
	\begin{subfigure}{.3\textwidth}
		\includegraphics[width=1.2\textwidth]{./Images/HD_error_vs_time_fixed_master.tex}
	\end{subfigure}\hspace{40pt}
	\begin{subfigure}{.3\textwidth}
		\includegraphics[width=1.2\textwidth]{./Images/HD_error_vs_time_fixed_DEIM.tex}
	\end{subfigure}
	\caption{\emph{Test case ii.} $L^2(\Omega_2)$ mean slave solution error vs the CPU time fixing the master POD tolerance (left) and the interface DEIM tolerance (right) to $10^{-5}$ and varying the tolerances used for the reduction of the other quantities.}
	\label{Fig:HD_error_vs_time}
\end{figure}

\begin{table}[h!]
	\centering 
	\begin{tabular}{c| cc| cccc}
		\toprule
		& \multicolumn{2}{c| }{High fidelity model} &\multicolumn{4}{c}{Reduced order model} \\ 
		&\#FE &FE solution  &\#RB &Offline &Online &Speed up\\ 
		&DoFs &time &&time &time \\
		\hline &&&&&&\\ [-2ex]
		Master model &36$k$ &$\sim$ $103.55s$ &9 &$\sim$ 4780$s$ &$\sim$ $55.99s$ &\cellcolor{red!25}1.9$x$\\
		Slave model &5$k$ &$\sim$ $19.66s$ &7 &$\sim$ 847$s$ &$\sim$ $6.86s$ &\cellcolor{red!15}2.9$x$\\
		Interface data & &\cellcolor{red!25}$\sim$ $6m$ &15 &$\sim$ 862$s$  & \cellcolor{green!25}$0.00s$&\\ 
		Coupled model & &$\sim$$483.21s$ & &$\sim$ 6489$s$ &$\sim$ $62.84s$ &\cellcolor{green!25}7.7$x$\\
		\bottomrule
	\end{tabular}
	\caption{\emph{Test case ii.} High fidelity and reduced order model dimensions and CPU times.  We highlight the performances of the ROM model with respect to the interface Dirichlet data treatment and the speed up using colors from red (worst) to green (best).}
	\label{Tab:HD_mesh_order}
\end{table}	

\subsection{Test case $iii$: unsteady model - unsteady model}
\label{Sub:fluid_wall_mass_transport}
This last test case addresses a simplified mass transfer problem used to describe the exchange of substances in biology between blood and the arterial wall. In this model the unknowns are the solute concentration convected along the vessel by blood and absorbed by the arterial wall under the blood stress induced on the vascular tissue. 

Introduced in \cite{Karner2021,Quarteroni2002}, this  fluid-wall model is based on an advection-diffusion equation to describe the solute dynamics in the arterial lumen, coupled with a pure diffusive equation accounting for the mass diffusion in the arterial wall. Usual coupling conditions are of Robin type; here, however,  we perform a further  simplification considering an isolated arterial vessel. Hence, we first solve the advection-diffusion equation for the blood transport and, then, the pure diffusive equation in the arterial wall imposing our usual interface Dirichlet conditions.

Specifically, denoting $C_f(\mathbf{x},t)$ and $C_w(\mathbf{x},t)$ the dimensionless concentrations of the solute in the lumen $\Omega_f$ and in the wall $\Omega_w$, respectively, we end up with the following problems:
\begin{equation}
\label{Eq:fluid_equations}
\begin{cases}
\frac{\partial C_f}{\partial t} + \mathbf{v} \cdot \nabla C_f  - \alpha_f \Delta C_f = 0 & \text{in }\Omega_f \times \{0,T\} \\
C_f  = \zeta &\text{on }\Sigma_{f,in} \times \{0,T\} \\
\alpha_f\nabla C_f \cdot \mathbf{n}_f = 0 &\text{on }\Sigma_{f,out} \cup \Gamma \times \{0,T\} \\
C_f(0) = 2.58\cdot 10^{-1} & \text{in }\Omega_f.
\end{cases}
\end{equation}
and 
\begin{equation}
\label{Eq:wall_equations}
\begin{cases}
\frac{\partial C_w}{\partial t} - \alpha_w \Delta C_w  = 0 &\text{in }\Omega_w \times \{0,T\} \\
C_w = C_f &\text{on } \Gamma \times \{0,T\}\\
C_w = 0 &\text{on } \Sigma_{w,0} \times \{0,T\}\\
\alpha_w\nabla C_w\cdot \mathbf{n}_w = 0 &\text{on } \Sigma_{w,in/out} \times \{0,T\}\\
C_w(0) = 2.58\cdot 10^{-1} & \text{in }\Omega_w.
\end{cases}
\end{equation}
where we use subscripts $f$ and $w$ to refer to the fluid or the wall, respectively, in place of the usual indices 1 and 2. Here, $\mathbf{v}$ is the fluid velocity vector and $\alpha_f$ and $\alpha_w$ are the  blood and wall solute diffusivity constants, respectively. 

Then, we define as $\Omega_f$ a small tube of radius $r = 0.3~cm$ and length $1~cm$, while $\Omega_w$ has a thickness equal to 10\% of the vessel lumen (see Fig.~$\ref{Fig:FW_domains}$). We fixed the initial concentration of the solute for both fluid and wall, namely $C_f(0)= C_w(0) = 2.58\cdot 10^{-1}$, and we impose a parabolic profile to the fluid velocity with constant flow rate $Q = 2.0~cm^3 /s$. Moreover, we choose $\alpha_f = 1.2\cdot 10^{-3}~cm^2/s$ and $\alpha_w = 0.9\cdot 10^{-3}~cm^2/s$, so that the P\'{e}clet number of both problems is of order $10^{3}$.

\begin{figure}[h!]
	\centering
	\begin{subfigure}{.48\textwidth}
		\includegraphics[width=1\textwidth]{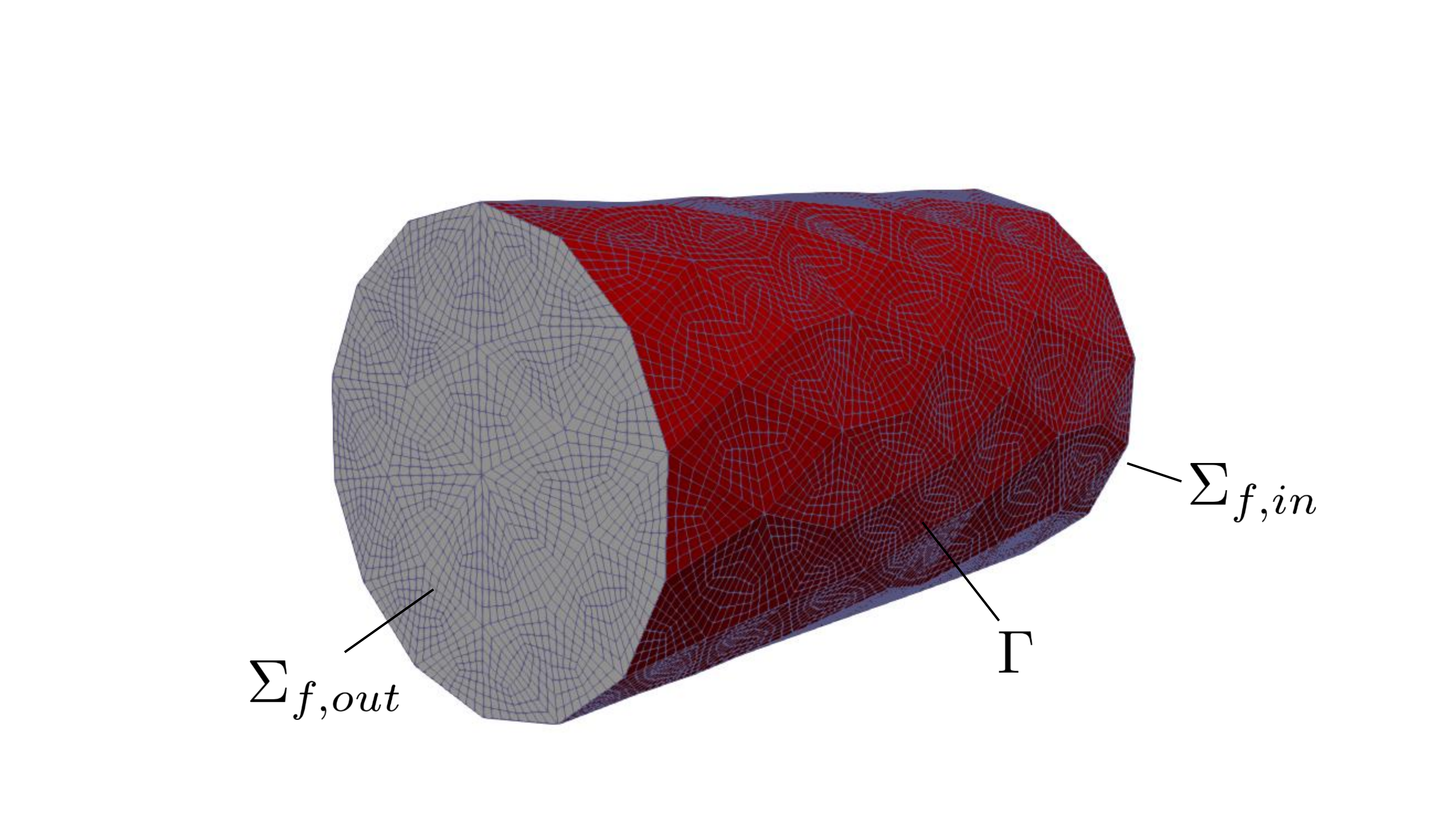}
	\end{subfigure}
	\begin{subfigure}{.48\textwidth}
		\includegraphics[width=1\textwidth]{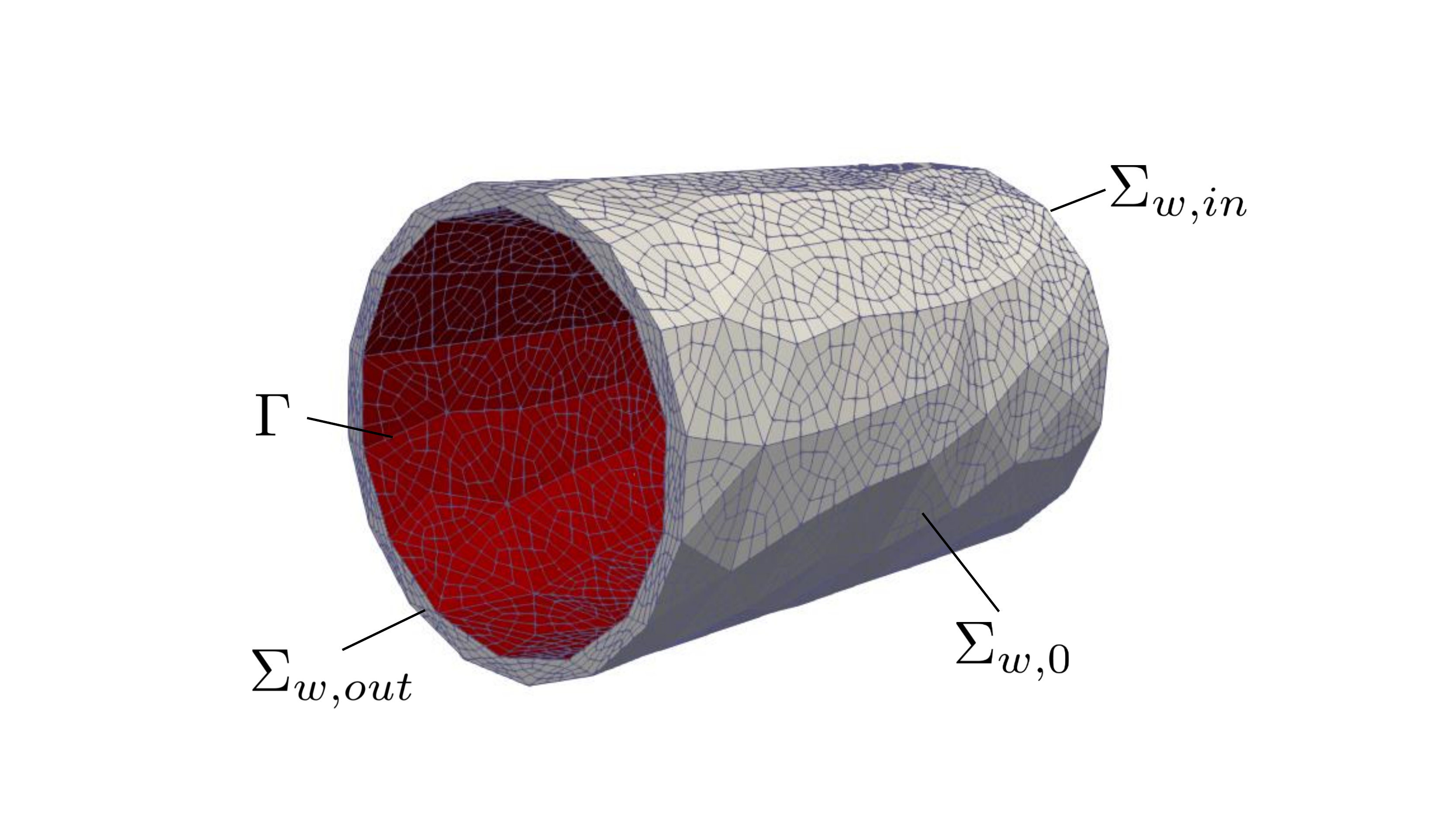}
	\end{subfigure}
	\caption{\emph{Test case iii.} Fluid (left) and wall (right) domains, a small tube of radius $r = 0.3 cm$ with a corresponding wall of thickness equal to the 10\% of the fluid domain lumen. Two different discretizations are considered, \emph{i.e.} $h_f = 0.0863505~cm$ and $h_w = 0.0260143~cm$. In red the interface boundary $\Gamma$.}
	\label{Fig:FW_domains}
\end{figure}

\begin{figure}[h!]
	\centering
	\begin{subfigure}{.2\textwidth}
		\centering
		\hspace{17pt} $t = 0.1s$\\
		\vspace{7pt}
		\includegraphics[width=1.2\textwidth]{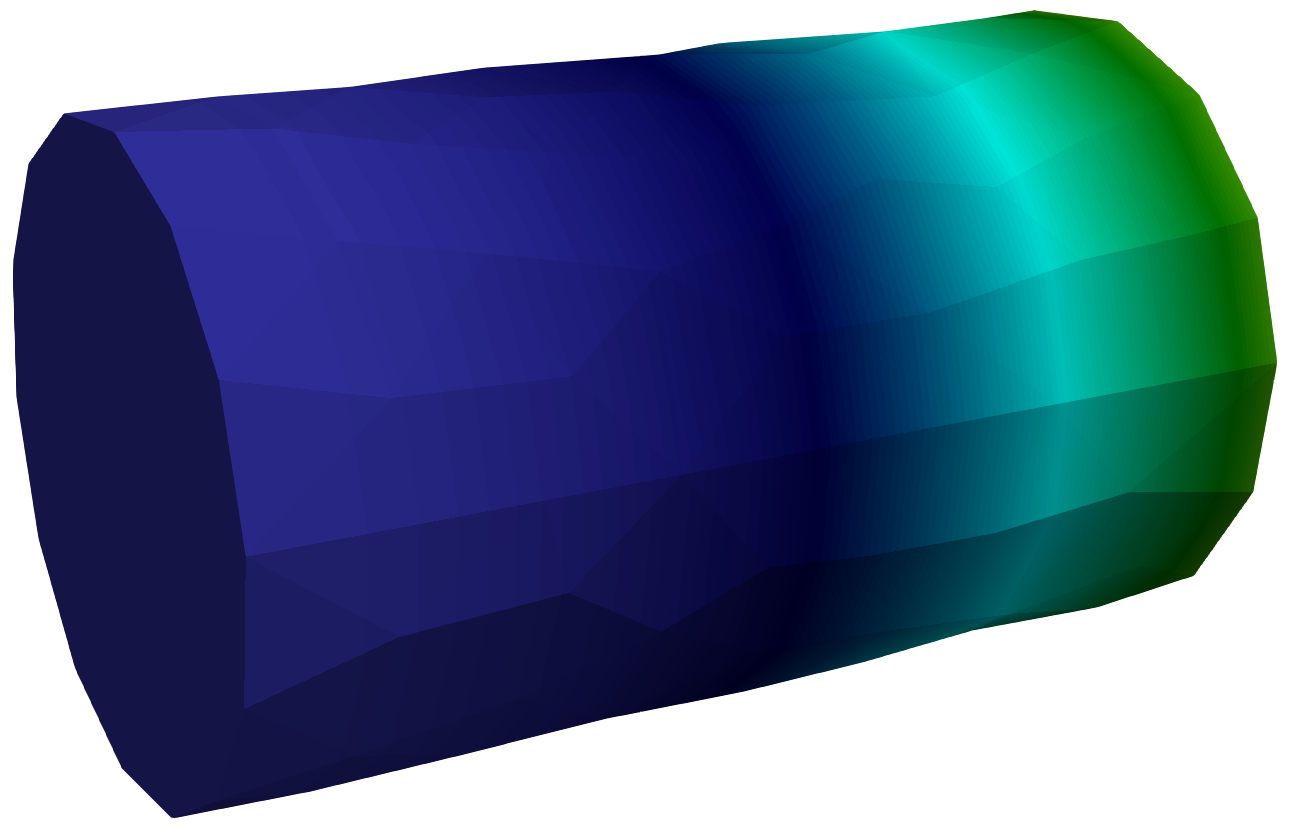}
	\end{subfigure}
	\hspace{15pt}
	\begin{subfigure}{.2\textwidth}
		\centering
		\hspace{17pt} $t = 0.45s$\\
		\vspace{7pt}
		\includegraphics[width=1.2\textwidth]{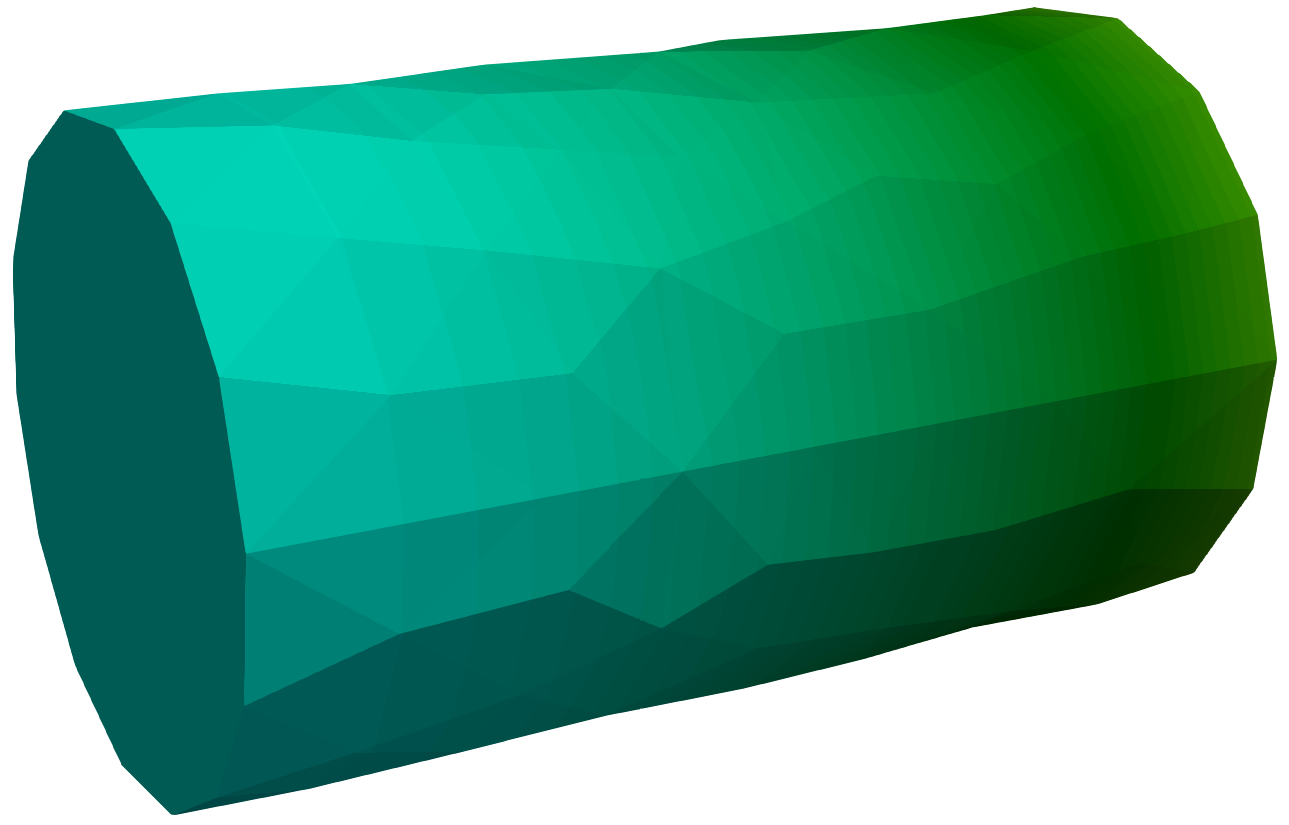}
	\end{subfigure}
	\hspace{15pt}
	\begin{subfigure}{.2\textwidth}
		\centering
		\hspace{17pt} $t = 0.7s$\\
		\vspace{7pt}
		\includegraphics[width=1.2\textwidth]{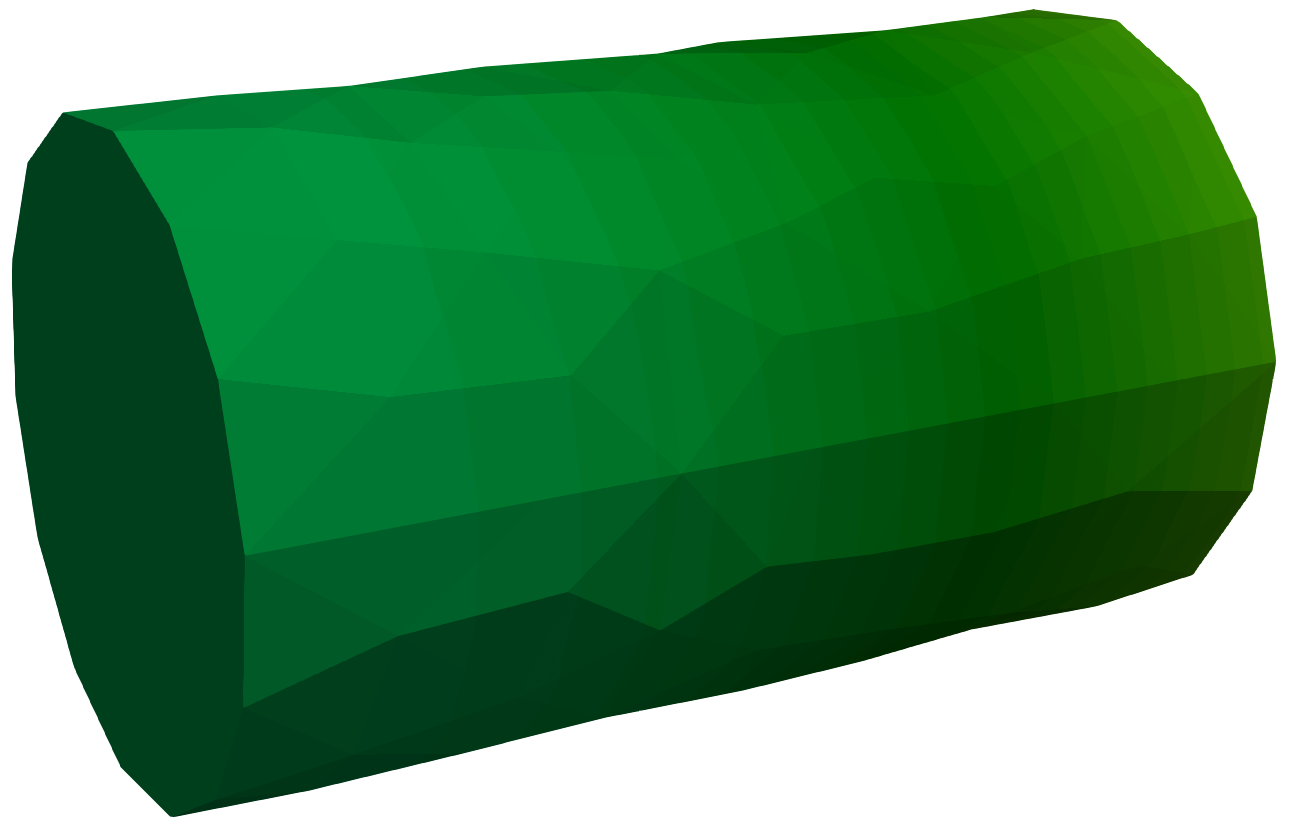}
	\end{subfigure} 
	\hspace{15pt}
	\begin{subfigure}{.2\textwidth}
		\centering
		\vspace{10pt}
		\includegraphics[width=0.6\textwidth]{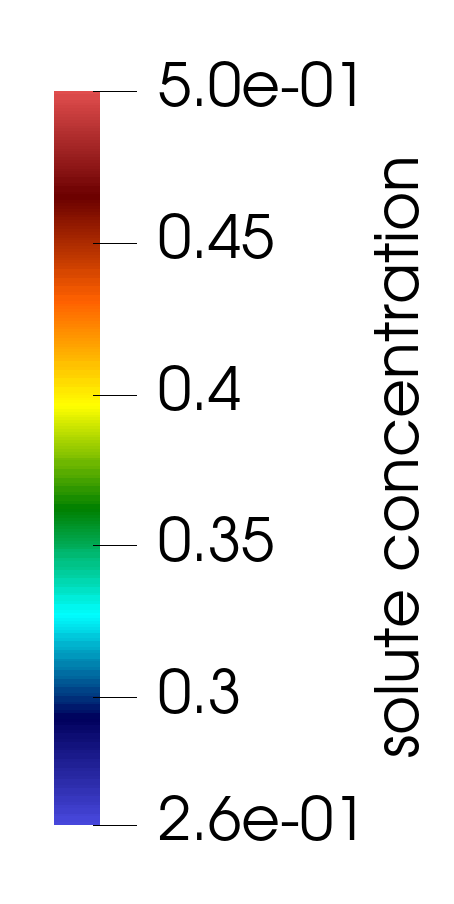}
	\end{subfigure}\\
	\begin{subfigure}{.2\textwidth}
		\centering
		\includegraphics[width=1.2\textwidth]{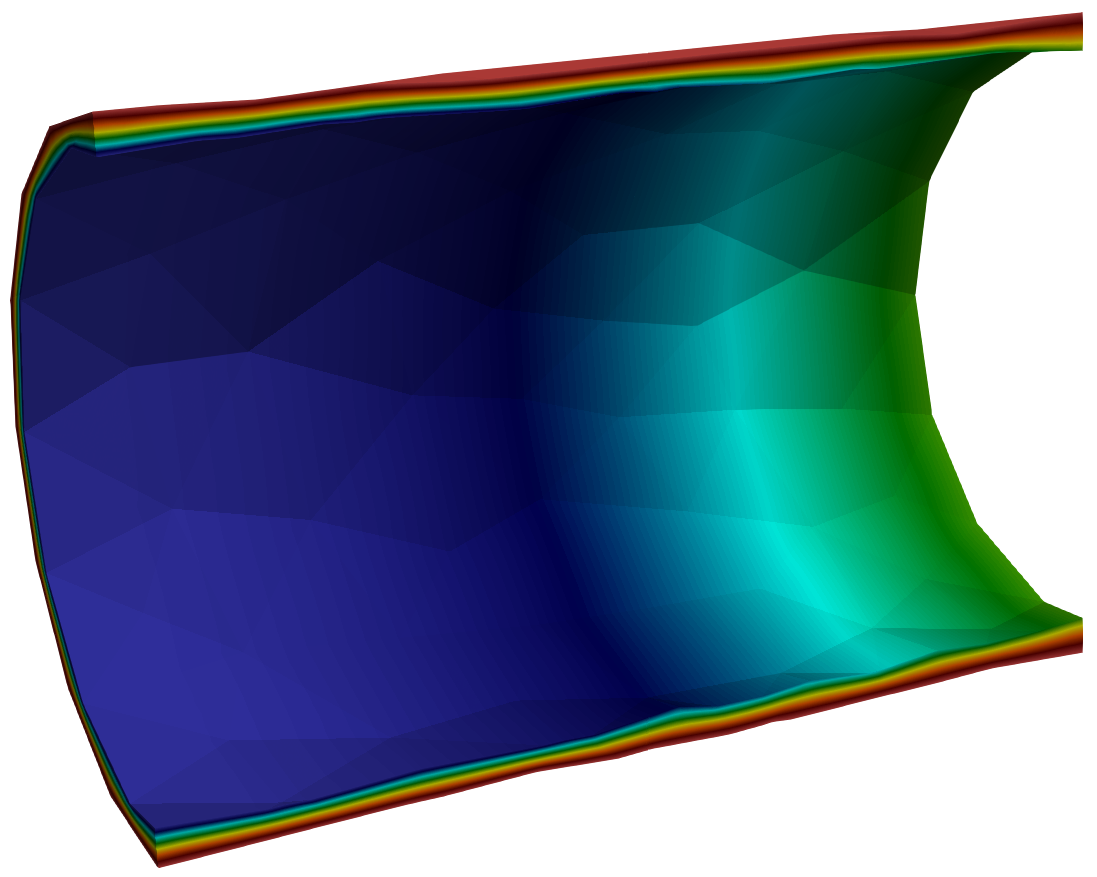}
	\end{subfigure}
	\hspace{15pt}
	\begin{subfigure}{.2\textwidth}
		\centering
		\includegraphics[width=1.2\textwidth]{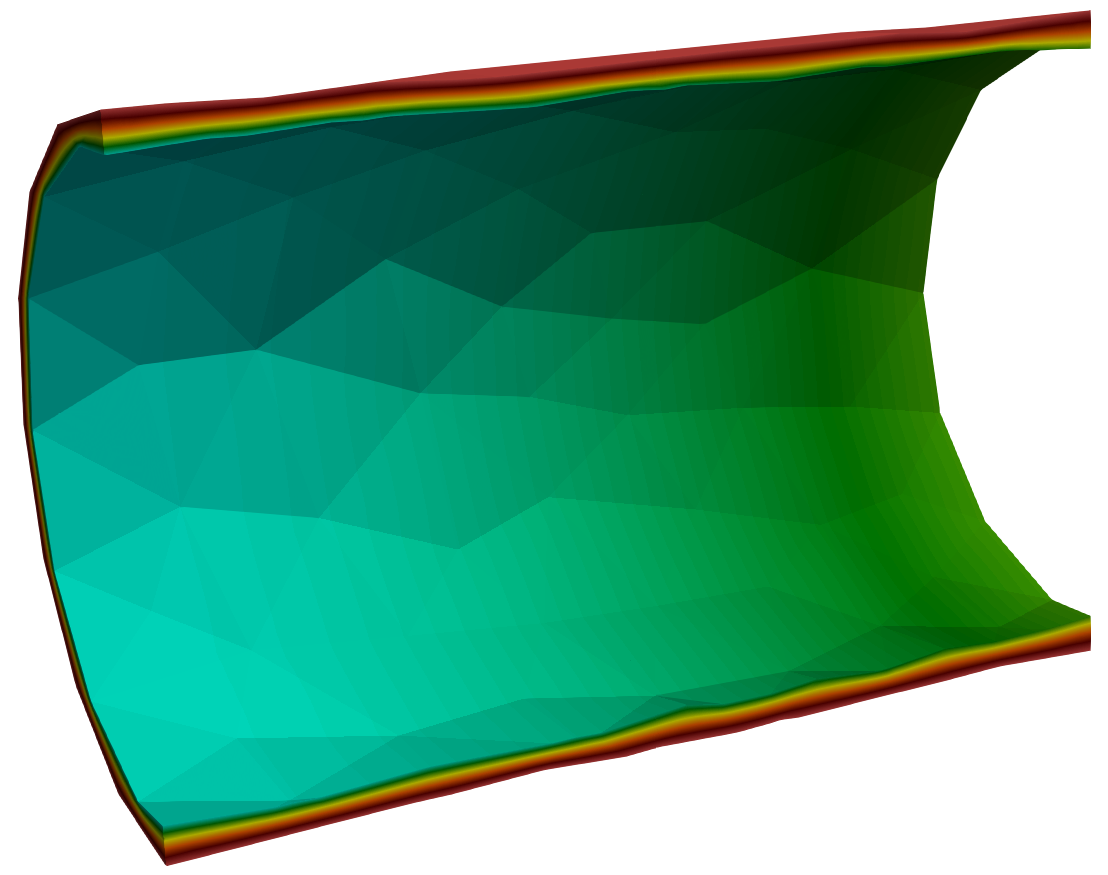}
	\end{subfigure}
	\hspace{15pt}
	\begin{subfigure}{.2\textwidth}
		\centering
		\includegraphics[width=1.2\textwidth]{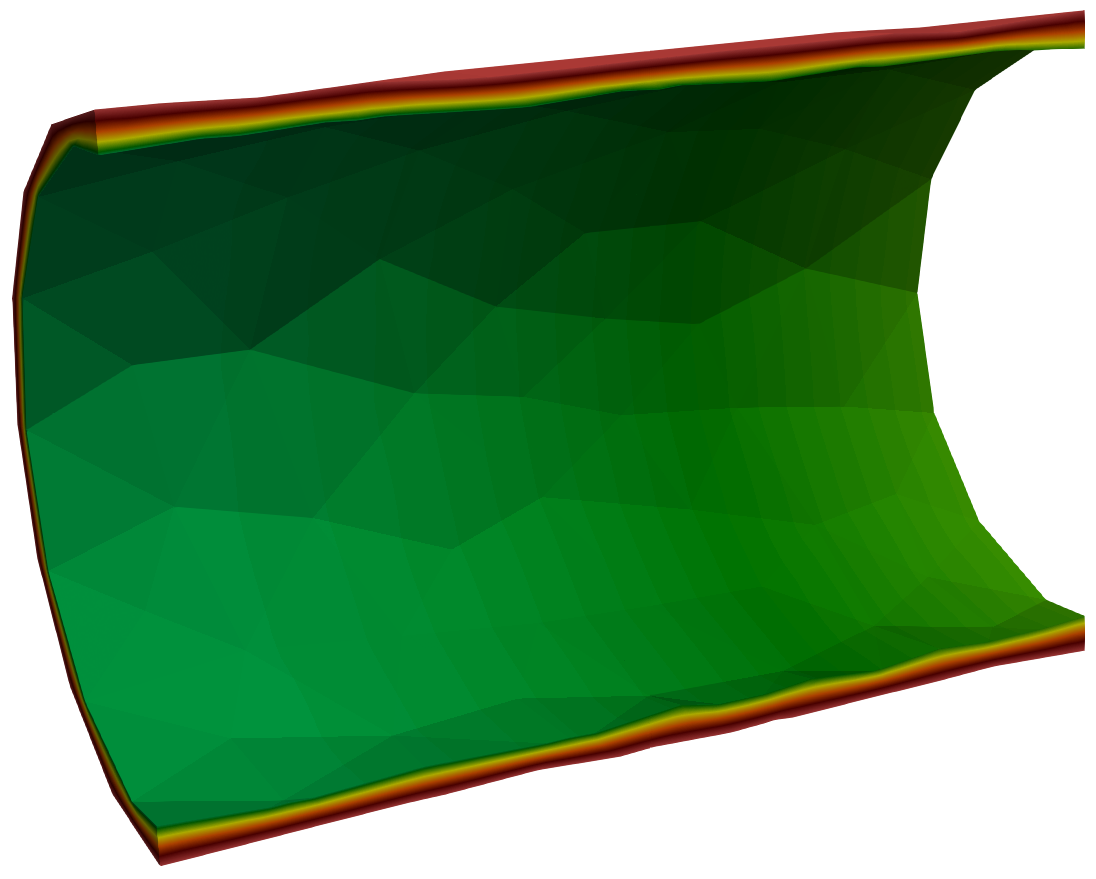}
	\end{subfigure} 
	\hspace{10pt}
	\begin{subfigure}{.2\textwidth}
		\centering
		\includegraphics[width=0.6\textwidth]{Images/bar_fluid_wall.png}
	\end{subfigure}
\begin{subfigure}{.2\textwidth}
	\centering
	\includegraphics[width=1.2\textwidth]{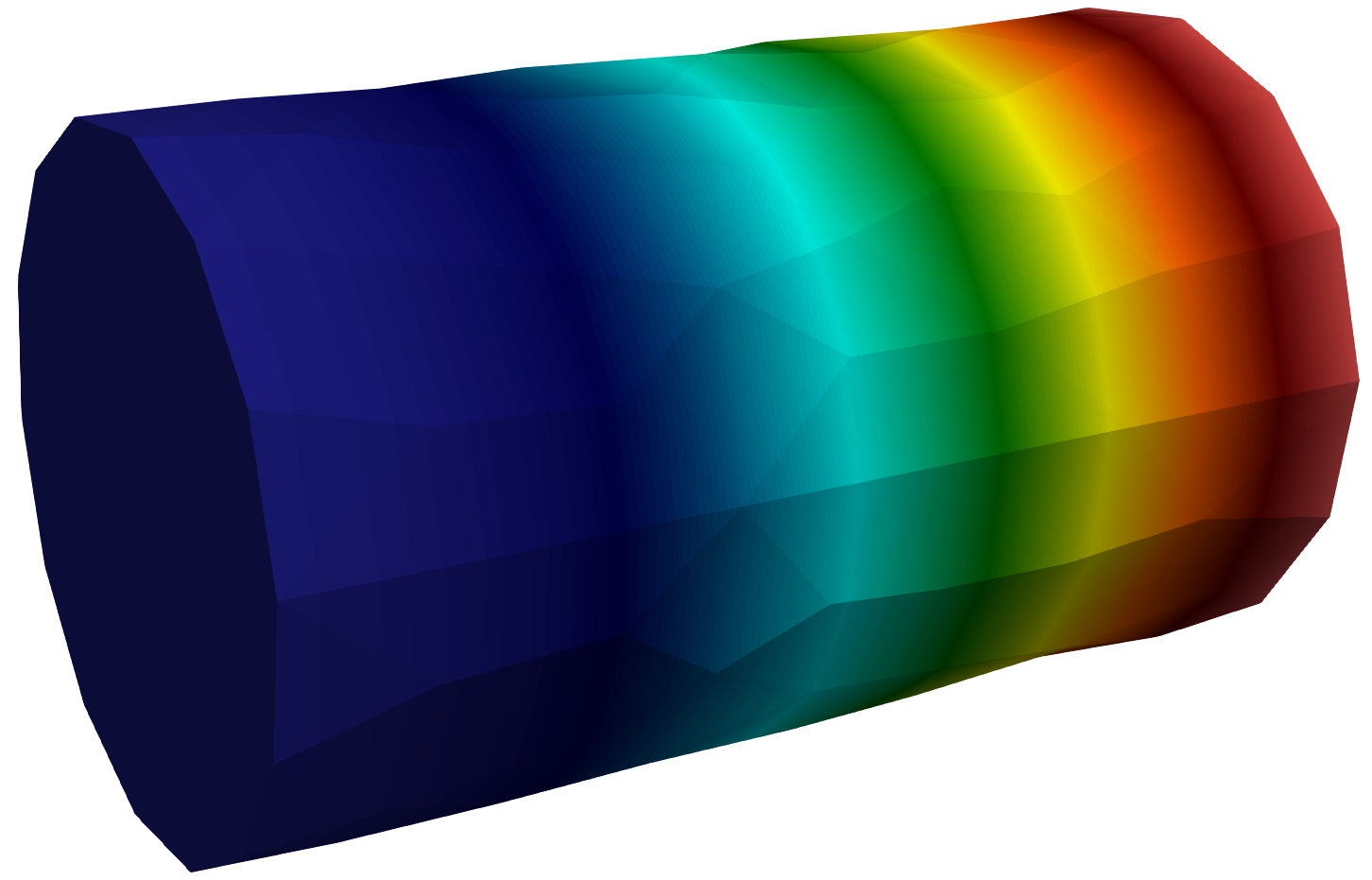}
\end{subfigure}
\hspace{15pt}
\begin{subfigure}{.2\textwidth}
	\centering
	\includegraphics[width=1.2\textwidth]{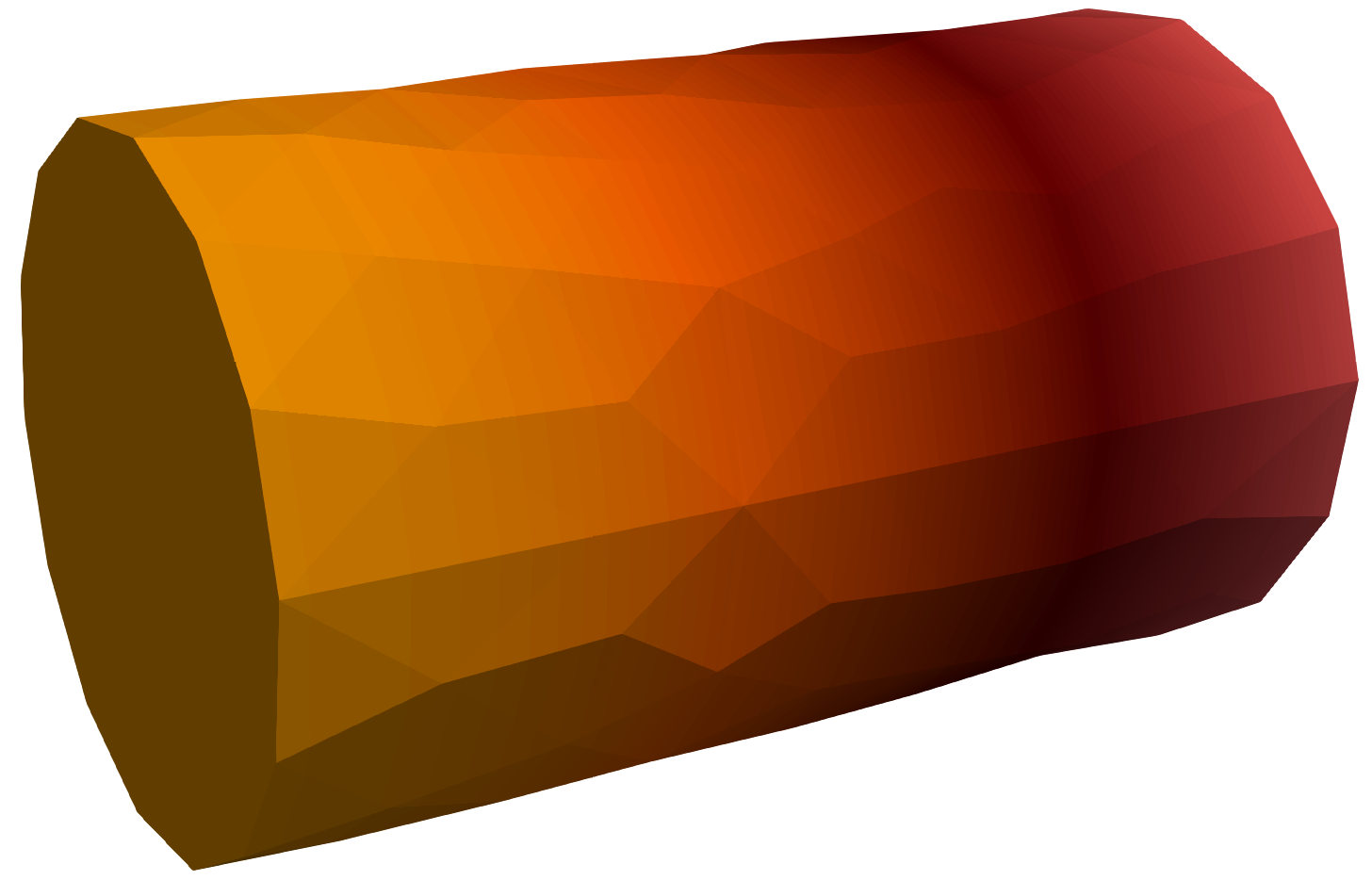}
\end{subfigure}
\hspace{15pt}
\begin{subfigure}{.2\textwidth}
	\centering
	\includegraphics[width=1.2\textwidth]{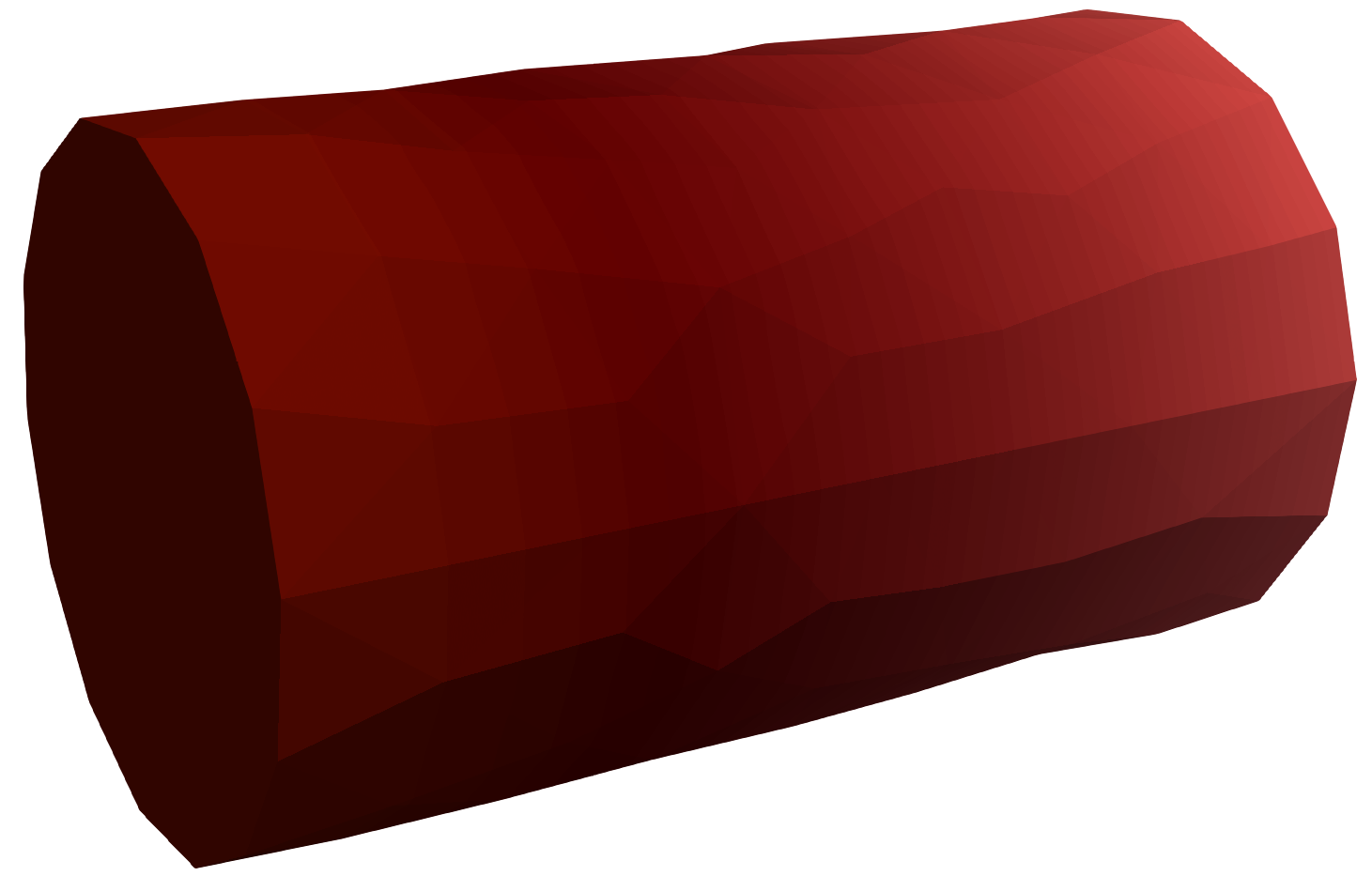}
\end{subfigure} 
\hspace{10pt}
\begin{subfigure}{.2\textwidth}
	\centering
	\vspace{10pt}
	\includegraphics[width=0.6\textwidth]{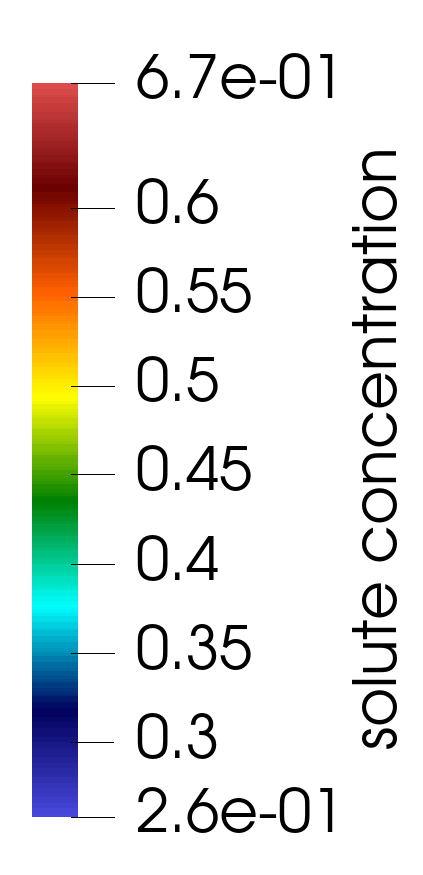}
\end{subfigure}\\
\begin{subfigure}{.2\textwidth}
	\centering
	\includegraphics[width=1.2\textwidth]{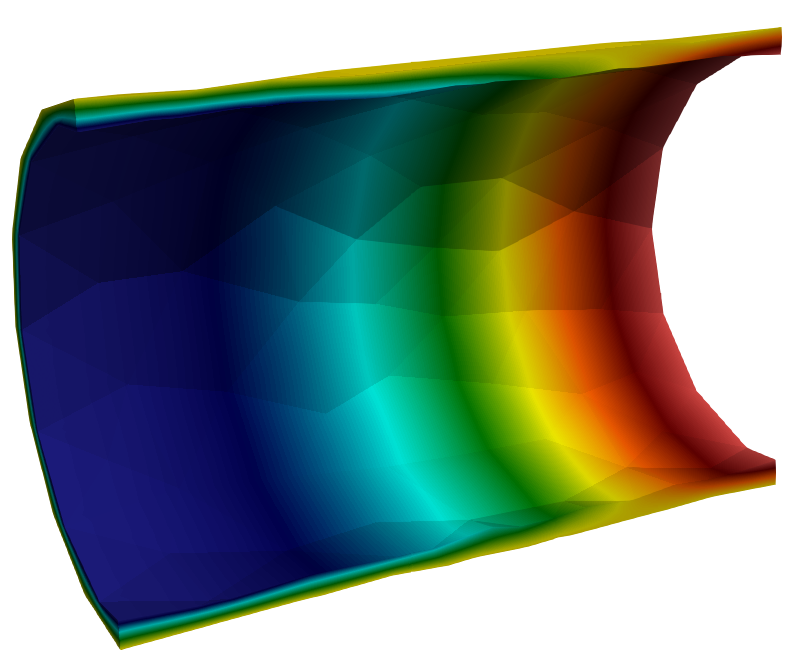}
\end{subfigure}
\hspace{15pt}
\begin{subfigure}{.2\textwidth}
	\centering
	\includegraphics[width=1.2\textwidth]{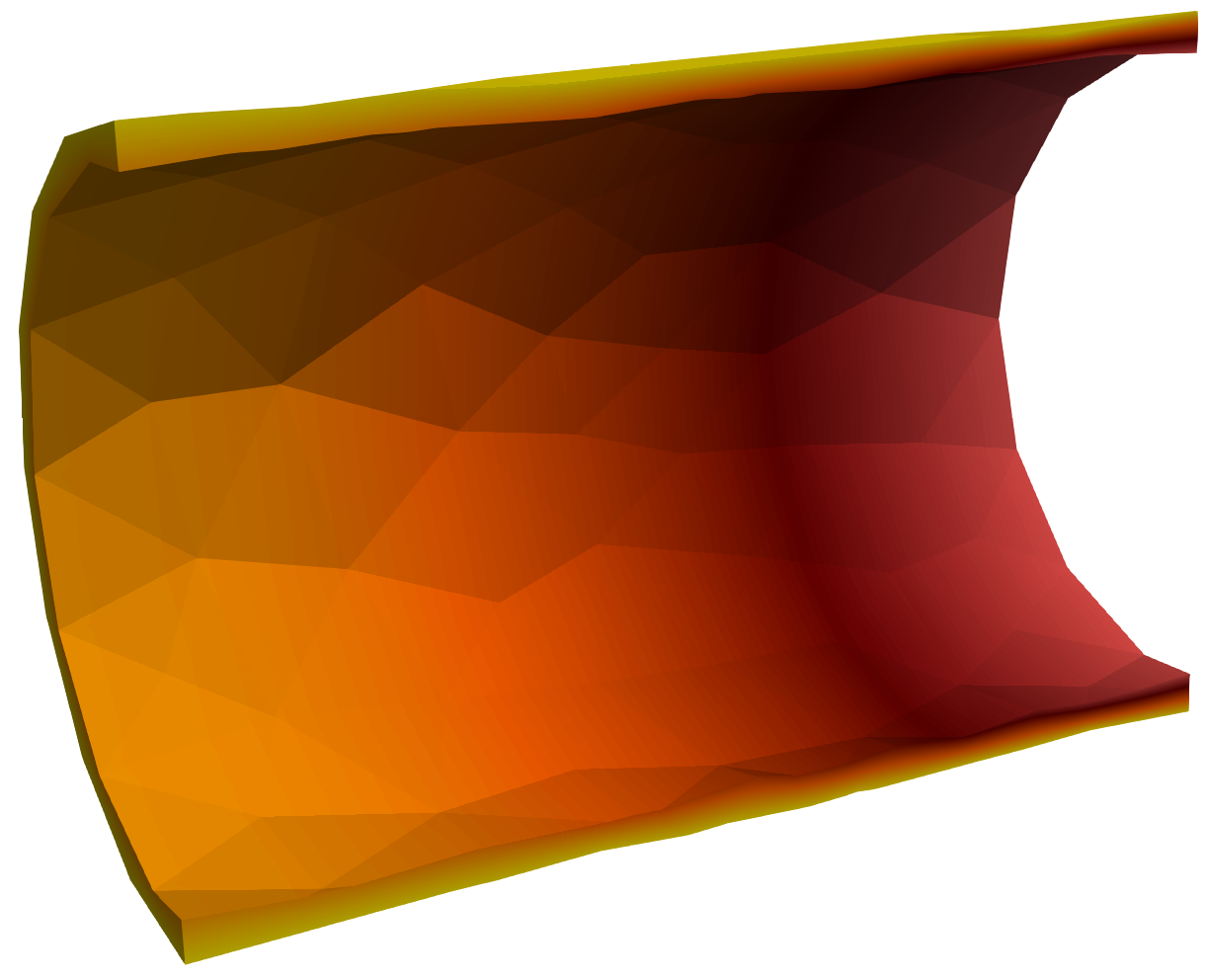}
\end{subfigure}
\hspace{15pt}
\begin{subfigure}{.2\textwidth}
	\centering
	\includegraphics[width=1.2\textwidth]{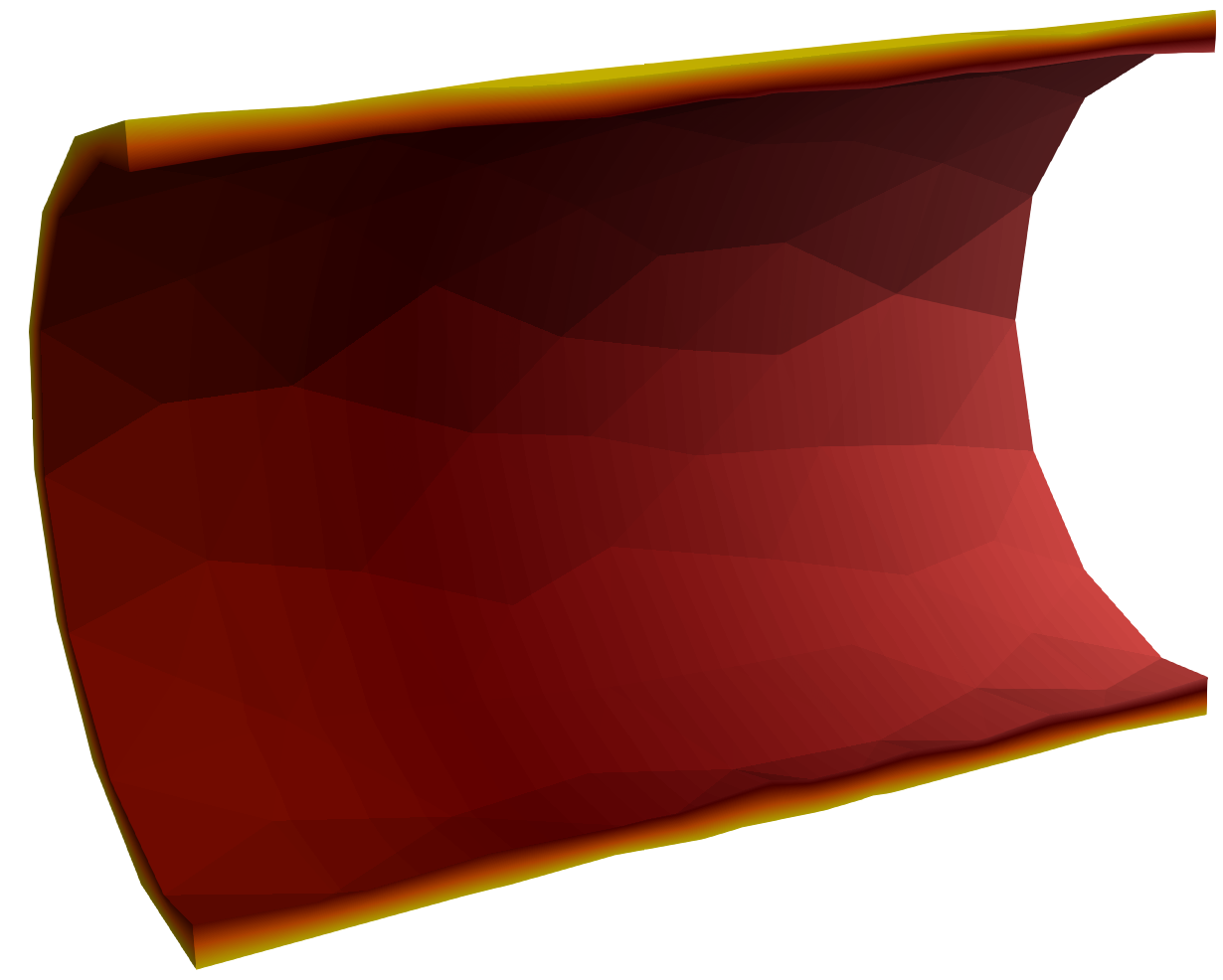}
\end{subfigure} 
\hspace{10pt}
\begin{subfigure}{.2\textwidth}
	\centering
	\includegraphics[width=0.6\textwidth]{Images/bar_fluid_wall_2.png}
\end{subfigure}
	\caption{\emph{Test case iii.} Fluid solute concentration (first and third rows) and wall solution concentration considering a Section of the wall domain (second and fourth rows) for three different time instants (columns), with $\zeta = 0.37$ (first two rows) and $\zeta = 0.67$ (second two rows).}
	\label{fig:snapshots_probl_3}
\end{figure}

Regarding the high-fidelity discretization, we define $T = 0.8~s$ (a cardiac beat in a real biological setting), $\Delta t  = 5\cdot 10^{-3}~s$ and apply a BDF scheme of order 1. We discretize differently $\Omega_f$ and $\Omega_w$, choosing  $h_f = 0.0863505~cm$ and $h_w = 0.0260143~cm$ so that $N_f = 144813$ and $N_w = 29624$ (see Fig. $\ref{fig:snapshots_probl_3}$ for some snapshots of the fluid and wall solution).  For the fluid model we consider as parameters the time variable and $\zeta \in [0.1,1]$, which described the solution constant concentration given to the tube in inlet, while for the wall model we consider a reduction in time and on the interface conditions, as usual. 

Again, first we evaluate the singular values decay related to the three set of snapshots varying the train set dimension $N_{train} = \{5, 10, 15, 20, 25\}$ of $\zeta$ according to a LHS distribution. As for test case \emph{ii}, given the time-dependent nature of the coupled problem, the corresponding set of snapshots has dimension $N_s = N_t N_{train}$, where $N_t = 160$ is the number of time steps considered for each simulation. The eigenvalues decays reported in Fig.$~\ref{Fig:singular_value_FD_mesh}$ show that $N_{train} = 15$ is enough to get a good reduction.
\begin{figure}[h!]
	\centering
	\begin{subfigure}{.32\textwidth}
		\centering
		\includegraphics[width=1.\textwidth]{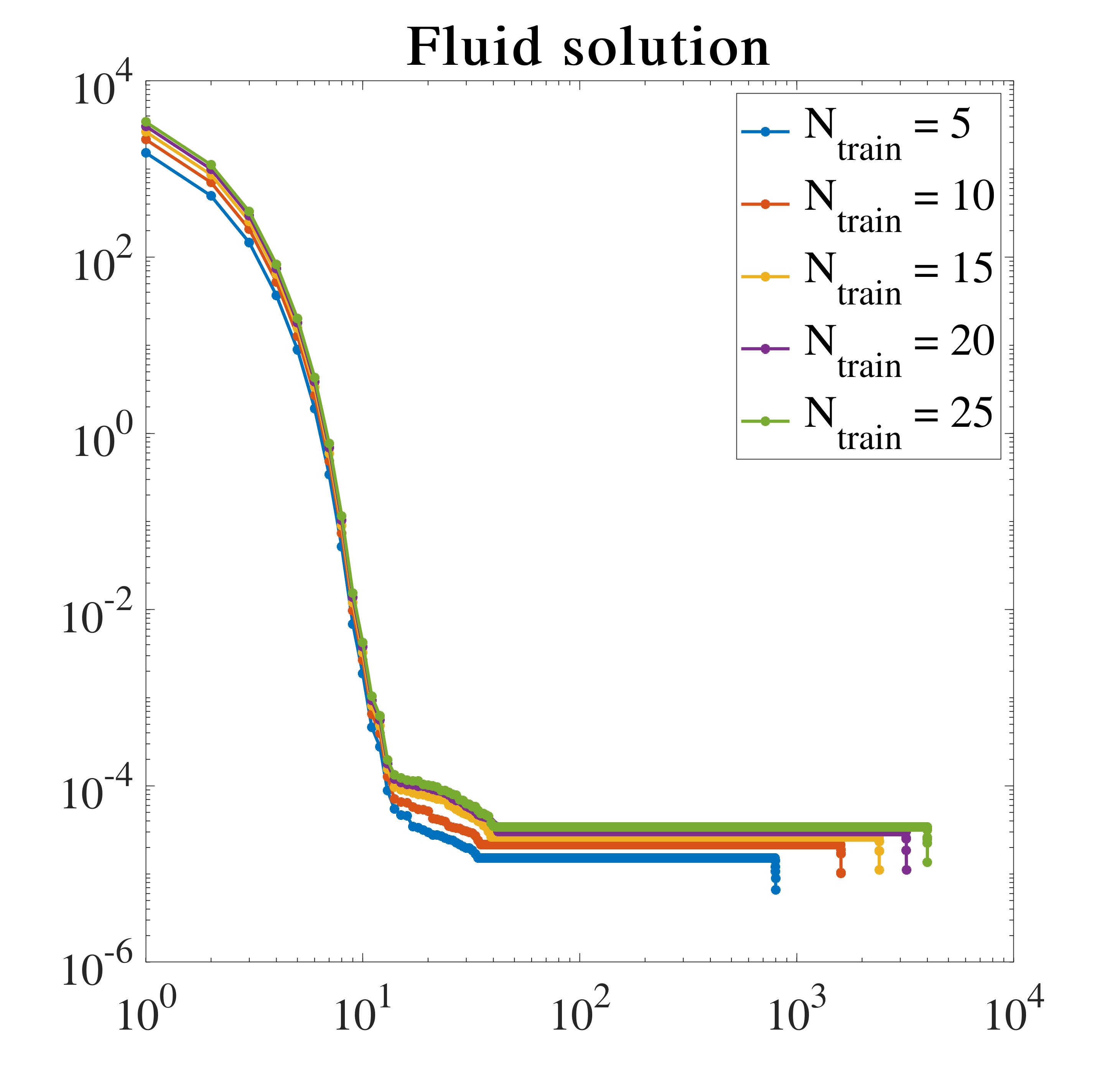}
	\end{subfigure}\hfill
	\begin{subfigure}{.32\textwidth}
		\centering
		\includegraphics[width=1.\textwidth]{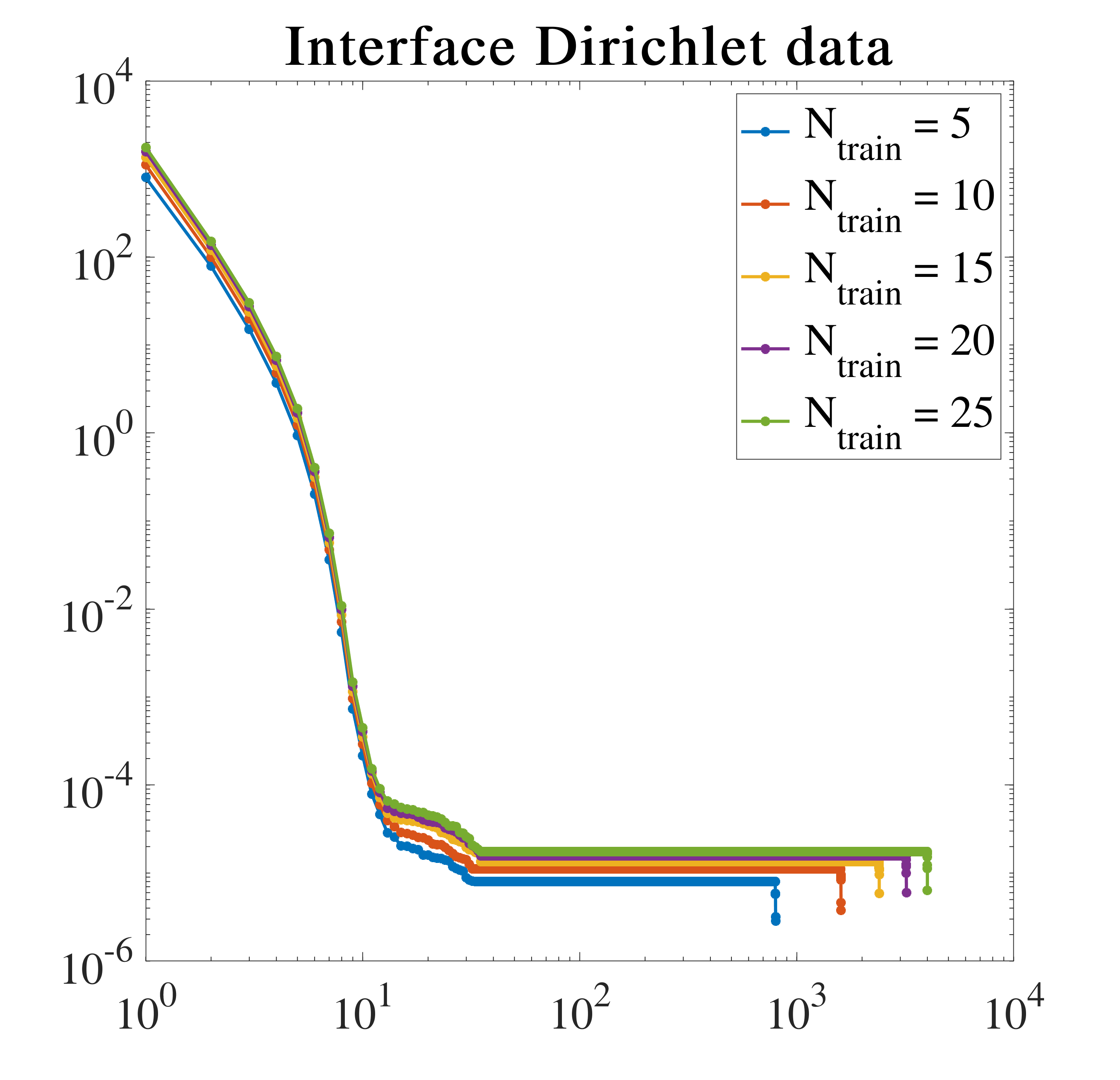}
	\end{subfigure}\hfill
	\begin{subfigure}{.32\textwidth}
		\centering
		\includegraphics[width=1.\textwidth]{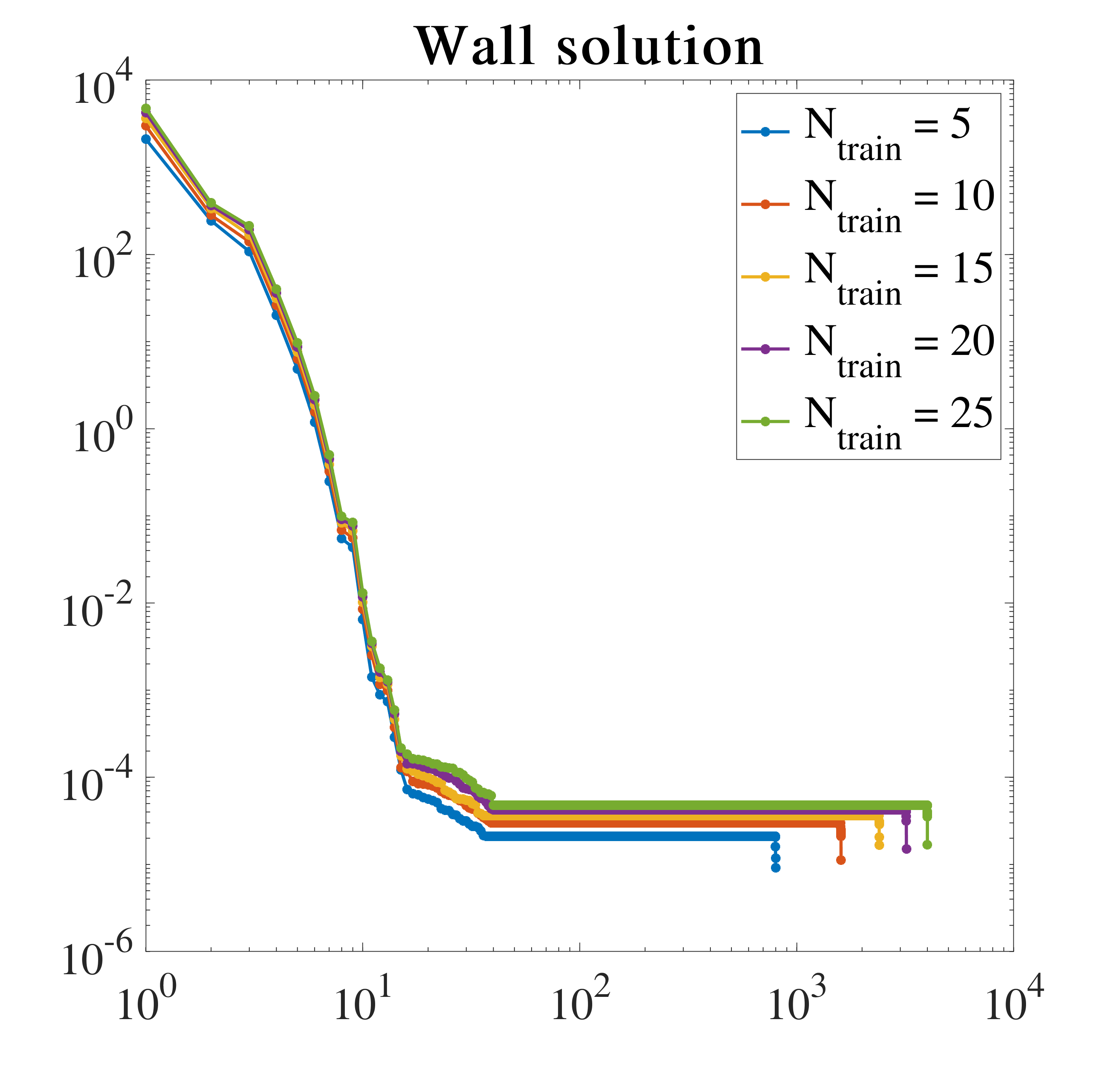}
	\end{subfigure}
	\caption{\emph{Test case iii.} Singular values decay of the fluid solution (left), interface Dirichlet data (center) and wall solution (right). }
	\label{Fig:singular_value_FD_mesh}
\end{figure}

Then, we select $N_{test} = 3$ values of $\zeta$ and we estimate the reduced error on the slave domain. Given the complexity of the fluid-wall problems and the high number of DoFs involved, the absolute mean 2-norm error over the wall solution could lead to an unbalance estimations. Therefore, in this case we rather use a relative mean 2-norm error, meaning that we consider the mean of the relative 2-norm errors 
$$\frac{\|  \mathbf{u}_{2,h_2}(\boldsymbol{\mu}_2) - \mathbb{V}_2  \mathbf{u}_{n_2}(\boldsymbol{\mu}_2)\| _2}{\|  \mathbf{u}_{2,h_2}(\boldsymbol{\mu}_2)\| _2}$$ computed for each reduced solution vector. Fig. $\ref{fig:FW_fixed_master_error}$ and $\ref{fig:FW_fixed_DEIM_error}$  outline the reduced errors obtained prescribing a fix accuracy of the fluid solution and the interface data reduction, respectively. Despite the higher complexity of this test case with respect to the previous ones, we highlight that the interface reduction does not impact on the final solution in terms of accuracy, as shown in test cases \emph{i} and \emph{ii}. In particular, prescribing a POD tolerance for the interface reduction between $10^{-3}$ and $10^{-5}$ does not change the final error on the slave solution, which, instead, is correctly influenced by the slave reduced solution. 

\begin{figure}[h!]
	\centering
	\begin{subfigure}{.3\textwidth}
		\centering 
		\includegraphics[width=1.2\textwidth]{./Images/FW_error_fix_master_1e-2.tex}
	\end{subfigure}\hspace{40pt}
	\begin{subfigure}{.3\textwidth}
		\centering 
		\includegraphics[width=1.2\textwidth]{./Images/FW_error_fix_master_1e-3.tex}
	\end{subfigure}\\
	\bigskip
	\begin{subfigure}{.3\textwidth}
		\centering 
		\includegraphics[width=1.2\textwidth]{./Images/FW_error_fix_master_1e-4.tex}
	\end{subfigure}\hspace{40pt}
	\begin{subfigure}{.3\textwidth}
		\centering 
		\includegraphics[width=1.2\textwidth]{./Images/FW_error_fix_master_1e-5.tex}
		\label{Fig:FW_fixed_master_5}
	\end{subfigure}
	\caption{\emph{Test case iii.} $L^2(\Omega_2)$ mean wall solution error ($y$-axis) over $N_{test}$ trial fixing the fluid POD tolerance and varying the DEIM tolerance ($x$-axis) and the wall POD tolerance (legend).}
	\label{fig:FW_fixed_master_error}
\end{figure}
\begin{figure}[h!]
	\centering
	\begin{subfigure}{.3\textwidth}
		\includegraphics[width=1.2\textwidth]{./Images/FW_error_fix_DEIM_1e-2.tex}
	\end{subfigure}\hspace{40pt}
	\begin{subfigure}{.3\textwidth}
		\includegraphics[width=1.2\textwidth]{./Images/FW_error_fix_DEIM_1e-3.tex}
	\end{subfigure}\\
	\bigskip
	\begin{subfigure}{.3\textwidth}
		\includegraphics[width=1.2\textwidth]{./Images/FW_error_fix_DEIM_1e-4.tex}
	\end{subfigure}\hspace{40pt}
	\begin{subfigure}{.3\textwidth}
		\includegraphics[width=1.2\textwidth]{./Images/FW_error_fix_DEIM_1e-5.tex}
	\end{subfigure}
	\caption{\emph{Test case iii.} $L^2(\Omega_2)$ mean wall solution error ($y$-axis) over $N_{test}$ trial fixing the DEIM tolerance and varying the fluid POD tolerance ($x$-axis) and the wall POD tolerance (legend).}
	\label{fig:FW_fixed_DEIM_error}
\end{figure}

Time performances for a fixed $10^{-5}$ POD tolerance for each reduction are instead reported in Fig. $\ref{Fig:FW_error_vs_time}$ and Table $\ref{Tab:FW_mesh_order}$. As for test case \emph{ii}, the computational costs refer to a complete simulation in time -- in this case, a 160 time-steps solution -- including necessary repeated operations such as the assembling of the reduced right hand side and of the full order solutions for both fluid and wall models. All this bottlenecks  can be eventually overcome considering, especially for the master reduction, a different assembling of the right hand side and hyper-reduction techniques  according to their complexity. Moreover, we point out the the gained speed up is of about 20 times, corresponding to a saving up of about 95\% of the computational costs of the wall simulations due to both the ROM strategy implemented and, more importantly, of the interface non conformity considered. This, with a saving up of 100\% of the interface extraction, ensures a 12 times speed up of the complete coupled problem solution, corresponding to a reduction of the 90\% of the CPU time compared to the FOM solution. Finally, we highlight that an accuracy of $10^{-5}$ in the slave solution requires only an increase of about 0.3\% of the total  cost of the solution of the same models imposing an accuracy of $10^{-5}$ on the master reduction, and of $10^{-2}$ on the slave and interface reduction, \emph{i.e.} passing from a final accuracy of $10^{-2}$ to $10^{-5}$ of the slave solution.

\begin{figure}[h!]
	\centering
	\begin{subfigure}{.3\textwidth}
		\includegraphics[width=1.2\textwidth]{./Images/FW_error_vs_time_fixed_master.tex}
	\end{subfigure}\hspace{40pt}
	\begin{subfigure}{.3\textwidth}
		\includegraphics[width=1.2\textwidth]{./Images/FW_error_vs_time_fixed_DEIM.tex}
	\end{subfigure}
	\caption{\emph{Test case iii.} $L^2(\Omega_2)$ mean wall solution error vs the CPU time fixing the fluid POD tolerance (left) and the interface DEIM tolerance (right) to $10^{-5}$ and varying the tolerances used for the reduction of the other quantities.}
	\label{Fig:FW_error_vs_time}
\end{figure}

\begin{table}[h!]
	\centering 
	\begin{tabular}{c| cc| cccc}
		\toprule
		& \multicolumn{2}{c| }{High fidelity model} &\multicolumn{4}{c}{Reduced order model} \\ 
		&\#FE &FE solution &\#RB &Offline &Online &Speed up
		\\ &DoFs &time &&time &time \\
		\hline &&&&&&\\ [-2ex]
		Fluid model &144$k$ &$\sim$ $394.68s$ &7 &$\sim$ 5974$s$ &$\sim$ $68.62s$ &\cellcolor{red!25}5.7$x$\\
		Wall model &29$k$ &$\sim$ $250.49s$ &9 &$\sim$ 3884$s$ &$\sim$ $12.26s$ &\cellcolor{green!25}20.4$x$\\
		Interface data & &\cellcolor{red!25}$\sim$ $6m$ &7 &$\sim$ 5525$s$  & \cellcolor{green!25}0.00$s$&\\ 
		Coupled model & &$\sim$$945.17s$ & &$\sim$ 15383$s$ &$\sim$ $80.88s$ &\cellcolor{red!15}11.7$x$\\
		\bottomrule
	\end{tabular}
	\caption{\emph{Test case iii.} High fidelity and reduced order model dimensions and CPU times.  We highlight the performances of the ROM model with respect to the interface Dirichlet data treatment and the speed up using colors from red (worst) to green (best).}
	\label{Tab:FW_mesh_order}
\end{table}	

\section{Conclusion}
In this paper we have proposed a new approach to deal with parametric coupled PDEs. The method, based on RB algorithms, can be used in combination with domain decomposition techniques when one-way coupled problems must be solved independently from each other and in sequence, following the Dirichlet interface conditions direction. The efficiency of the coupled ROM is ensured by the modular nature of the proposed strategy, enabling the possibility to treat in very different ways the master and slave reduction and solution, including different FE degrees. In particular, the main building blocks of this method are the slave and master models, to be reduced with tailored RB strategies, and the interface Dirichlet data, which is treated and passed between the interface domains through DEIM, without applying other expensive techniques, such as, \emph{e.g.}, Lagrange multipliers. Special emphasis has been put in the importance of using DEIM to handle interface data between conforming and, more importantly, non-conforming interface grids.

\emph{A posteriori} reduced error estimates for the proposed method in both the steady and the unsteady cases have been derived, showing the strong relation between the slave error and the master and interface errors. 

Our numerical tests show that our reduction strategy can be applied to very different coupled problems. The efficiency in the coupled ROM online phase outperforms the high-fidelity counterpart, gaining an overall speed up in the complete coupled problem computation from 200 times for the most simple steady case to 12 times for the more complex ones. However, the biggest advantage in CPU time can be seen in the interface treatment, both in the saving up of the 100\% of the interface extraction time and in the general speed up obtained through the slave model reduction, which can ultimately ensure a saving up of the 95\% of CPU time also in the most complex cases, such as in the fluid-wall mass transport problem of Subsection $\ref{Sub:fluid_wall_mass_transport}$. The committed error can be carefully controlled at each step of the reduction with small influence, as a whole, of the accuracy imposed on the interface reduction. 

On the basis of the results obtained with simple partitioned one-way coupled problems, we expect to be able to apply the present strategy to more complex and relevant coupled problems. Moreover, a natural extension of the presented strategy concerns the use of other kinds of interface (e.g.,  Neumann-like) conditions, and the treatment of more challenging two-way coupled problems; both these aspects represent the focus of a forthcoming publication.

\section*{Acknowledgment} 
This project has received funding from the European Research Council (ERC)
under the European Union's Horizon 2020 research and innovation programme. Grant agreement No. 740132: iHEART - An Integrated Heart Model for the simulation of the cardiac function, P.I. Prof. A. Quarteroni. 


\bibliography{Ref}

\begin{appendices}
\section{\emph{A posteriori} error estimator for unsteady reduced basis models}
\label{App:error_estimation}
To find an \emph{a posteriori} error estimates of a parametrized reduced simulation scheme for time-dependent model, according to \cite{Haasdonk2011}, we assume to have the following parametrized linear dynamical system for a vector $\mathbf{u}(t;\boldsymbol{\mu}) \in \mathbb{R}^n$, for $t \in [0,T]$:
\[ 
\begin{cases}\frac{d}{dt} \mathbf{u}(t;\boldsymbol{\mu}) = \mathbb{A}_{N}(t;\boldsymbol{\mu})\mathbf{u}(t;\boldsymbol{\mu}) + \mathbf{f}_{N}(t;\boldsymbol{\mu}) &\text{in }\Omega \times [0,T]\\
\mathbf{u}(0;\boldsymbol{\mu}) = \mathbf{u}_0(\boldsymbol{\mu}) &\text{on }\Omega \times \{0\}.
\end{cases}
\]
Here the matrix $\mathbb{A}_{N}(t;\boldsymbol{\mu})
\in \mathbb{R}^{N \times N}$ and vector $\mathbf{f}_{N}(t;\boldsymbol{\mu}) \in \mathbb{R}^{N}$, being $N$ the dimension of the reference finite dimensional space, are dependent from the parameters vector $\boldsymbol{\mu} \in \mathscr{P} \subset \mathbb{R}^d$. We assume to fix $\boldsymbol{\mu}$ for each single simulation of the dynamical system. Moreover, we define the projection matrix $\mathbb{V} \in \mathbb{R}^{N\times n}$ defined through RB methods, where $n \leq N$ is the reduced model order. Note that the error estimation technique here reported is not restricted on a particular choice of RB. Then, the reduced dynamical system is:
\begin{equation}
\begin{cases}
\frac{d}{dt} \mathbf{u}_n(t;\boldsymbol{\mu}) = \mathbb{A}_{n}(t;\boldsymbol{\mu})\mathbf{u}_n(t;\boldsymbol{\mu}) + \mathbf{f}_{n}(t;\boldsymbol{\mu}) &\forall t \in [0,T]\\
\mathbf{u}_n(0;\boldsymbol{\mu}) = \mathbf{u}_{n,0}(\boldsymbol{\mu})
\end{cases}
\end{equation}
where $\mathbb{A}_{n}(t;\boldsymbol{\mu}) = \mathbb{V}^T\mathbb{A}_{N}(t;\boldsymbol{\mu})\mathbb{V}$, $\mathbf{f}_{n}(t;\boldsymbol{\mu}) = \mathbb{V}^T\mathbf{f}_{N}(t;\boldsymbol{\mu}) \mathbb{V}$, $\mathbf{u}_n(t;\boldsymbol{\mu})$ is the reduced solution, \emph{i.e} $\mathbf{u}_N(t;\boldsymbol{\mu}) = \mathbb{V}\mathbf{u}_n(t;\boldsymbol{\mu})$, and $\mathbf{u}_{n,0}(\boldsymbol{\mu})$ is the projection on the reduced space of the initial condition $\mathbf{u}_{0}(\boldsymbol{\mu})$. 

The error analysis is residual based, thus first we define the error and the residual as
$$\mathbf{e}(t;\boldsymbol{\mu}) := \mathbf{u}(t;\boldsymbol{\mu}) - \mathbb{V}\mathbf{u}_n(t;\boldsymbol{\mu})$$
and 
$$\mathbf{r}(t;\boldsymbol{\mu}) := \mathbb{A}_N(t;\boldsymbol{\mu})\mathbb{V}\mathbf{u}_n(t;\boldsymbol{\mu}) + \mathbf{f}_{N}(t;\boldsymbol{\mu}) - \mathbb{V} \frac{d}{dt}\mathbf{u}_n(t;\boldsymbol{\mu}).$$ 

Furthermore, a suitable norm must be chosen. Assuming to have a symmetric positive definite inner product matrix $\mathbf{G} \in \mathbb{R}^{N \times N}$, we call the corresponding inner product $\langle \mathbf{u},\mathbf{u}' \rangle_{\mathbf{G}}$, so that we can compute the induced vector and matrix norm $\| \mathbf{u}\| _{\mathbf{G}} := \sqrt{\langle \mathbf{u},\mathbf{u}\rangle_{\mathbf{G}}}$ on $\mathbf{R}^{N}$ and $\| \mathbb{A}\| _{\mathbf{G}} := \text{sup}_{\| \mathbf{u}\| _{\mathbf{G}}}\| \mathbb{A}\mathbf{u}\| _{\mathbf{G}}$, for $\mathbb{A} \in \mathbb{R}^{N\times N}$. For example, if $\mathbf{G} = \mathbb{I}_{N \times N}$, \emph{i.e.} it is the identity matrix, than we obtain the simple 2-norm used in this work. Then, the following \emph{a posteriori} error estimator can be stated:
	
\begin{prop}[\emph{A Posteriori} Error Estimate] Assuming that $\mathbb{A}_N (t;\boldsymbol{\mu}) = \mathbb{A}(\boldsymbol{\mu})$ is time-invariant and has eigenvalues with negative real part for all $\boldsymbol{\mu} \in \mathscr{P}$, than the solution is bounded by
$$ \sup_t \| \exp(\mathbb{A}_N(\boldsymbol{\mu})t)\| _{\mathbf{G}} \leq C_1(\boldsymbol{\mu}),$$
where $C_1(\boldsymbol{\mu})$ is a computable constant.
Then, the following error estimates holds:
\begin{equation}
\label{Eq:general_error_estimate_unsteady}
\| \mathbf{u}(t;\boldsymbol{\mu}) - \mathbb{V}\mathbf{u}_n(t;\boldsymbol{\mu})\| _{\mathbf{G}} \leq C_1(\boldsymbol{\mu}) \left( \| \mathbf{e}(0;\boldsymbol{\mu})\| _{\mathbf{G}} + \int_0^T \| \mathbf{r}(\tau;\boldsymbol{\mu})\| _{\mathbf{G}} d \tau \right).
\end{equation}
\end{prop}	
\begin{proof}
From the residual definition, we obtain that
$$\mathbb{V}\frac{d}{dt}\mathbf{u}_n(t;\boldsymbol{\mu}) = \mathbb{A}_N(\boldsymbol{\mu})\mathbb{V}\mathbf{u}_n(t;\boldsymbol{\mu})+\mathbf{f}_N(t;\boldsymbol{\mu}) - \mathbf{r}(t;\boldsymbol{\mu}).$$

Subtracting this equation from the original system, we get an evolution system for the error, namely
\begin{equation}
\label{Eq:dynamical_system_errors}
\begin{cases}\frac{d}{dt}\mathbf{e}(t;\boldsymbol{\mu}) = \mathbb{A}_N(\boldsymbol{\mu})\mathbf{e}(t;\boldsymbol{\mu}) + \mathbf{r}(t;\boldsymbol{\mu})\\
\mathbf{e}(0;\boldsymbol{\mu}) = \mathbf{u}_0(\boldsymbol{\mu}) - \mathbb{V} \mathbf{u}_{n,0}(\boldsymbol{\mu}).
\end{cases}
\end{equation}

Then, there exists an explicit solution of this linear system, which is
$$\mathbf{e}(t;\boldsymbol{\mu}) = \text{exp}(\mathbb{A}(\boldsymbol{\mu})t)\mathbf{e}(0;\boldsymbol{\mu}) + \int_{0}^T \exp(\mathbb{A}(\boldsymbol{\mu})(T-\tau))\mathbf{r}(\tau;\boldsymbol{\mu})d\tau.$$

The proof is concluded assuming the boundedness of $\| \exp(\mathbb{A}(\boldsymbol{\mu})s)\| _{\mathbf{G}} \leq C_1(\boldsymbol{\mu})$ for $s \in \mathbb{R}^{+}$.
\end{proof}

Error relations similar to  $\eqref{Eq:general_error_estimate_unsteady}$ can also be found  for time dependent systems, meaning when $\mathbb{A}_{N}(t;\boldsymbol{\mu})$ depends on time, by a suitable modification of $C_1(\boldsymbol{\mu})$. To do this, we first point out that the error evolution system $\eqref{Eq:dynamical_system_errors}$ holds also for time-variants systems. Then, integrating, we get
$$ \mathbf{e}(t;\boldsymbol{\mu}) = \mathbf{e}(0;\boldsymbol{\mu}) + \int_0^{T} \mathbb{A}(\tau;\boldsymbol{\mu})\mathbf{e}(\tau;\boldsymbol{\mu}) + \mathbf{r}(\tau;\boldsymbol{\mu})d\tau.$$

Denoting by $\boldsymbol{\Phi}(t) := \| \mathbf{e}(t;\boldsymbol{\mu})\| _{\mathbf{G}}$, $\boldsymbol{\alpha}(t):= \| \mathbf{e}(0;\boldsymbol{\mu})\| _{\mathbf{G}} + \int_0^T \|  \mathbf{r}(\tau;\boldsymbol{\mu})\| _{\mathbf{G}}d\tau$ and $\boldsymbol{\beta}(t) := \| \mathbb{A}(\tau;\boldsymbol{\mu})\| _{\mathbf{G}}$, we can obtain
$$\boldsymbol{\Phi}(t) \leq \boldsymbol{\alpha}(t) + \int_0^T \boldsymbol{\beta}(\tau)\boldsymbol{\Phi}(\tau)d\tau.$$

Moreover, assuming an upper bound $\| \mathbb{A}(t;\boldsymbol{\mu})\| _{\mathbf{G}} \leq C_3(\boldsymbol{\mu})$ for $t \in [0,T]$, $\boldsymbol{\mu} \in \mathscr{P}$, using the Gronwall inequality, we can write
\[
\begin{split}
\boldsymbol{\Phi}(t) &\leq \boldsymbol{\alpha}(t) + \int_0^T \boldsymbol{\alpha}(\tau)\boldsymbol{\beta}(\tau)\exp \left(\int_s^T \boldsymbol{\beta}(r)dr \right)d\tau \leq \boldsymbol{\alpha}(t)(1 + C_3(\boldsymbol{\mu})t \exp(C_3 t)).
\end{split}
\]
Then, equation $\eqref{Eq:general_error_estimate_unsteady}$ can be found denoting by $C_1 := 1 + C_3(\boldsymbol{\mu})T \exp(C_3 T)$.
\end{appendices}
\end{document}